\documentclass{amsart}

\usepackage{graphicx}
\usepackage{listings}
\usepackage{amssymb}
\usepackage{hyperref}
\usepackage{amsmath}
\usepackage{amssymb}
\usepackage{nicefrac}
\usepackage{soul}
\usepackage{pgfplots}
\pgfplotsset{compat=newest}

\numberwithin{equation}{section}
\newcommand{\be}{\begin{equation}}
\newcommand{\ee}{\end{equation}}
\newcommand{\HT}{{\mathcal{H}}_T}
\newcommand{\calD}{{\mathcal D}}
\newcommand{\bmp}{{\boldsymbol{p}}}
\newcommand{\muhp}{\mu_{\mathrm{hp}}}
\newtheorem{theorem}{Theorem}[section]
\newtheorem{definition}[theorem]{Definition}
\newtheorem{lemma}[theorem]{Lemma}
\newtheorem{proposition}[theorem]{Proposition}

\newtheorem{remark}[theorem]{Remark}

\newcommand{\cA}{\mathcal A}

\newcommand{\cG}{\mathcal G}
\newcommand{\cK}{\mathcal K}
\newcommand{\cL}{\mathcal L}

\newcommand{\cT}{\mathcal T}

\newcommand{\cW}{\mathcal W}

\newcommand{\IN}{\mathbb N}
\newcommand{\IP}{\mathbb P}
\newcommand{\IR}{\mathbb R}

\newcommand{\domain}{\mathrm{D}}
\newcommand{\semigroup}{E}
\newcommand{\abs}[1]{\left\lvert{#1}\right\rvert}
\newcommand{\norm}[1]{{\left\lVert{#1}\right\rVert}}
\newcommand{\spf}[2]{{\left\langle{#1},{#2}\right\rangle}}
\newcommand{\normiii}[1]{{\left\vert\kern-0.25ex\left\vert\kern-0.25ex\left\vert #1 
    \right\vert\kern-0.25ex\right\vert\kern-0.25ex\right\vert}}
\newcommand{\s}{r}
\newcommand{\rr}{q}
\DeclareMathOperator*{\essinf}{ess\,inf}
\newcommand{\dtime}{\delta}

\title[Space-Time $hp$]{
Exponential Convergence of $hp$-Time-Stepping
\\
in Space-Time Discretizations of Parabolic PDEs
}

\author{Ilaria Perugia}
\address{\mbox{Faculty of Mathematics, University of Vienna, Vienna, Austria}}
\email{{ilaria.perugia,marco.zank}@univie.ac.at}
\thanks{
This work was performed in autumn 2021 when ChS visited
Uni Vienna while on sabbatical leave from ETH, and completed
at the
2021 S\"ollerhaus Workshop in October 2021 in Hirschegg, Austria.
Excellent working conditions there are acknowledged.
\\
IP has been funded by the Austrian Science Fund (FWF)
through project F~65 ``Taming Complexity in Partial Differential
Systems'' and project P~33477-N}

\author{Christoph Schwab}
\address{\mbox{Seminar for Applied Mathematics, ETH Z\"urich, 8092 Z\"urich, Switzerland}}
\email{schwab@math.ethz.ch}

\author{Marco Zank}

\date{}
\dedicatory{}


\begin{document}

\begin{abstract}
For linear parabolic initial-boundary value problems with
    self-ad\-joint, time-homogeneous elliptic spatial
    operator in divergence form with Lip\-schitz-continuous coefficients,
    and for incompatible, time-analytic forcing term in polygonal/polyhedral
    domains $\domain$, we prove time-analyticity of solutions.
    Temporal analyticity is quantified in terms of weighted, analytic function classes,
    for data with finite, low spatial regularity and without boundary compatibility.
    Leveraging this result, we prove exponential convergence 
    of a conforming, semi-discrete $hp$-time-stepping approach. 
    We combine this semi-discretization in time 
    with first-order, so-called ``$h$-version'' Lagrangian Finite Elements 
    with corner-refinements in space into a tensor-product, 
    conforming discretization of a space-time formulation. 
    We prove that, under appropriate corner-
    and corner-edge mesh-refinement of~$\domain$, 
    error vs. number of degrees of freedom 
    in space-time behaves essentially (up to logarithmic terms), 
to what standard FEM provide for one elliptic boundary value problem solve 
in $\domain$. 
We focus on two-dimensional spatial domains and comment on
the one- and the three-dimensional case.
\end{abstract}

\maketitle 

%%%%%%%%%%%%%%%%%%%%%%%%%%%%%%%%%%%%%%%%%%%%%%%%%%%%%%%%%%%%%%%%%%%%%%%%%%%%%%
\section{Introduction} \label{sec:intro}
\label{sec:Intro}
%%%%%%%%%%%%%%%%%%%%%%%%%%%%%%%%%%%%%%%%%%%%%%%%%%%%%%%%%%%%%%%%%%%%%%%%%%%%%%
Efficient numerical solution of parabolic evolution problems is required
in many applications. In addition to the plain numerical solution of 
associated initial-boundary value problems, in recent years the efficient
numerical treatment of optimal control problems and of uncertain input data
has been considered. 
Here, often a large number of  cases needs to be treated, and the 
(numerical) solution must be stored in a data-compressed format.
Rather than the (trivial) option of a~posteriori compressing a 
numerical solution obtained by a standard scheme, 
novel algorithms have emerged featuring some form of 
\emph{space-time compressibility} in the numerical solution process. 
I.e., the numerical scheme will obtain directly, at runtime, a 
numerical solution in a compressed format. As examples,
we mention only sparse-grid and wavelet-based methods  (e.g., \cite{GrOeltz}),
and wavelet-based compressive schemes (e.g., \cite{AKChS,ScSt} and the references there).
Key to successful compressive space-time discretizations is an
appropriate \emph{variational formulation} of the evolution problem under consideration.
Accordingly, recent years have seen the development of a variety of,
in general nonequivalent, space-time variational formulations of parabolic initial-boundary value
problems. 
Departing from the classical, Bochner-space perspective used to establish well-posedness,
the novel formulations adopt the perspective of treating the parabolic evolution problem
as an operator equation between appropriate function spaces,
the primary motivation being accomodation of efficient, compressive space-time
numerical schemes. 
We mention only~\cite{ScSt,Andr,Ost,LMN_2016,CDG_2017,OStMZ,MMST_2020,FueKar,GGRSt,SW_2021} and the references there. A comprehensive account of the numerical analysis of 
\emph{fixed order time-discretizations} is provided in~\cite{ThomeeBk2nd} and the references there. 
In the results given in that volume, the semigroup perspective is adopted,
and the mathematical setting is based on homogeneous Sobolev spaces $\dot{H}^s(\domain)$,
which impose implicit boundary compatibilities of regular data, see \cite[Chapter~19]{ThomeeBk2nd}.

The presently investigated time-discretization approach 
is based on the space-time variational formulation in \cite{OStMZ}. 
It is of Petrov--Galerkin type, and is based on a fractional order Sobolev
space in the temporal direction. 
It has been proposed and developed in a series of 
papers~\cite{OStMZ, SteinbachZankNoteHT, ZankExactRealizationHT, langer2020efficient, SteinbachMissoni2022}.
We briefly recapitulate it here, and refer to \cite{OStMZ}
for full development of details. 
The compressive aspect is here realized 
by the $hp$-time discretization for this formulation.

Throughout, we denote by $\domain\subset \IR^d$ a bounded 
interval (if $d=1$), or a bounded
polygonal (if $d=2$) or polyhedral (if $d=3$) domain, with 
a Lipschitz boundary $\Gamma = \partial\domain$ consisting of a 
finite number of plane faces, and by $T>0$ a finite time horizon.
In the space-time cylinder $Q=(0,T)\times \domain$, 
we consider the parabolic initial-boundary value problem
(IBVP for short) governed by the partial differential equation
\be\label{eq:IBVP}
Bu := \partial_t u + A(\partial_x)u = g \quad \mbox{in}\quad (0,T)\times \domain.
\ee
Here, the forcing function $g:Q\to \IR$ is assumed to belong to 
$\cA([0,T];L^2(\domain))$, i.e., it is analytic as a map from $[0,T]$ into $L^2(\domain)$.
The spatial differential operator $A(\partial_x)$
is assumed linear, self-adjoint, in divergence form, 
i.e.,
\[
A(\partial_x) = -\nabla_x \cdot(A(x) \nabla_x) 
\]
with $A\in L^\infty(\domain;\IR^{d\times d})$ 
being a symmetric, 
positive definite matrix function of $x\in \domain$ 
which does not depend on the temporal variable $t$.
The PDE \eqref{eq:IBVP} is completed by initial condition
\be\label{eq:IC}
u|_{t=0} = u_0\;,
\ee
and by mixed boundary conditions
\be\label{eq:BC}
\gamma_0(u) = u_D \quad \mbox{on}\quad \Gamma_D\;,
\quad 
\gamma_1(u) = u_N \quad \mbox{on}\quad \Gamma_N\;.
\ee
Here, $\Gamma_D$ and $\Gamma_N$ denote a partitioning
of $\Gamma = \partial \domain$ into a Dirichlet and a Neumann part,
$\gamma_0$
denotes the Dirichlet trace map, 
and 
$\gamma_1$ denotes the conormal trace operator, given (in strong form)
by $\gamma_1(v) = n_x \cdot (A(x) \nabla_x v)|_\Gamma$, 
with $\Gamma = \partial\domain$ denoting the %Lipschitz
boundary of $\domain$,
and $n_x \in L^\infty(\Gamma;\IR^d)$ the exterior unit normal vector
field on~$\Gamma$.
\begin{remark}\label{rmk:HomBC}
In the rest of this paper, the results are formulated for $u_0=0$, $u_D = 0$,
and $u_N = 0$. Since the IBVP~\eqref{eq:IBVP}--\eqref{eq:BC} is linear, 
superposition for a sufficiently regular function $U(x,t)$ in $Q$, 
which satisfies~\eqref{eq:IC} and~\eqref{eq:BC}, will imply 
that the function $u-U$ will solve~\eqref{eq:IBVP}--\eqref{eq:BC} with
$g-BU$ in place of $g$ in~\eqref{eq:IBVP}, and with homogeneous
initial and boundary data in~\eqref{eq:IC} and~\eqref{eq:BC}.
All regularity hypotheses which we will impose
below on the source term $g$ in \eqref{eq:IBVP} 
(in particular, time-analyticity \eqref{eq:TimeReg}) 
entail via $U$ corresponding assumptions on $u_0$, $u_D$, and $u_N$.
\end{remark}

Exploiting the analytic semigroup property of the parabolic evolution operator,
we provide in Section \ref{sec:tReg} 
sufficient conditions for the time analyticity of solutions when considered
as maps from the time interval $[0,T]$ into a 
suitable Sobolev space $W\subset L^2(\domain)$
on the bounded spatial domain $\domain \subset \IR^d$.

\emph{Contributions of the present paper} are a weighted analytic, 
temporal regularity analysis
based on the analytic semigroup theory for linear, parabolic evolution equations, 
for source terms and coefficients of finite spatial regularity,
and the proof of exponential convergence of a temporal $hp$-discretization.
For polygonal spatial domain $\domain\subset \IR^2$, and for data without 
boundary compatibility, we establish 
\emph{a~priori} convergence rate bounds for 
fully discrete, space-time approximations 
which are based on a fractional order space-time formulation,
on $hp$-time-stepping and on $h$-FEM with 
corner-refined, regular graded triangulations in~$\domain$.
The diffusion coefficient $A(x)$ is assumed to be independent of $t$, and 
to belong to $W^{1,\infty}(\domain; \IR^{2\times 2})$. 
We comment on the cases $d=1$ (when $\domain$ is a bounded interval)
and $d=3$ (when $\domain$ is a polyhedron).

The layout of this paper is as follows:
In Section~\ref{sec:FctSpcxtForm}, we introduce notation and function spaces
of tensor product and of Bochner type, which will be used in the following.
We also provide the space-time variational formulation in fractional order
spaces and the subspaces used in discretization.
Section~\ref{sec:Reg} addresses the solution regularity, with particular attention
to temporal analytic regularity in weighted, analytic Bochner spaces
of functions taking values in corner-weighted, Kondrat'ev type spaces on 
the domain $\domain$. Section~\ref{sec:Approx} then introduces the 
Galerkin approximations in space and time that will be used,
and their approximation properties.
Section~\ref{sec:ConvRate} contains the main results on the convergence 
rate of the discretization. Section~\ref{sec:NumExp} describes the 
numerical realization of the nonlocal temporal bilinear form, and
reports numerical results which are in full agreement with the 
convergence rate analysis.

We use standard notation: 
$\IN = \{1,2,\dots\}$ 
shall denote the natural numbers, and 
$\IN_0:=\IN\cup \{0\}$. 
For Banach spaces $X$ and $Y$, $\cL(X,Y)$ denotes the space of bounded
linear operators from $X$ to $Y$, and $X':=\cL(X,\IR)$ denotes the dual of $X$. 
For $q\in [1,\infty]$,
the usual notation $L^q(\domain)$ is adopted 
for Lebesgue spaces of $q$-integrable functions $u:\domain\to\IR$ over some (bounded)
domain $\domain$ in the Euclidean space $\IR^d$. 
For nonnegative integers~$k$, Hilbertian 
Sobolev spaces (where $q=2$) on such domains $\domain$ are denoted by
$H^k(\domain)$. 
For $k=0$, as usual, $H^0(\domain) = L^2(\domain)$. 
Hilbertian Sobolev spaces of noninteger order $s = k + \theta$
for $k\in \IN_0$ and $0<\theta<1$ are defined by interpolation 
(real method, with fine index $2$).

%%%%%%%%%%%%%%%%%%%%%%%%%%%%%%%%%%%%%%%%%%%%%%%%%%%%%%%%%%%%%%%%%%%%%%%%%%%%%%
\section{Function Spaces and Space-time Variational Formulation}
\label{sec:FctSpcxtForm}
%%%%%%%%%%%%%%%%%%%%%%%%%%%%%%%%%%%%%%%%%%%%%%%%%%%%%%%%%%%%%%%%%%%%%%%%%%%%%%
We introduce several Bochner-type Sobolev spaces 
in the space-time cylinder $Q:=J\times \domain$, with the finite time
interval $J := (0,T)$ and the bounded spatial domain $\domain\subset \IR^d$.

%%%%%%%%%%%%%%%%%%%%%%%%%%%%%%%%%%%%%%%%%%%%%%%%%%%%%%%%%%%%%%%%%%%%%%%%%%%%%%
\subsection{Function Spaces}
\label{sec:FctSpc}
%%%%%%%%%%%%%%%%%%%%%%%%%%%%%%%%%%%%%%%%%%%%%%%%%%%%%%%%%%%%%%%%%%%%%%%%%%%%%%
Bochner-type function spaces defined on the space-time cylinder
$Q=J\times \domain$ 
are spaces of strongly measurable maps $u \colon \, J \to H^l(\domain)$,
such that $u\in H^k(J;H^l(\domain))$ for nonnegative integers $k$,$l$.
Due to the Hilbertian structure of $H^k$, these separable Hilbert spaces 
admit tensor product structure, i.e.,
\[
H^k(J;H^l(\domain))
\simeq 
H^k(J)\otimes H^l(\domain)
\simeq 
H^l(\domain;H^k(J))\;,
\]
where 
$\simeq$ denotes (isometric) isomorphism and 
$\otimes$ the Hilbertian tensor product.

For any integer $k\geq 1$, we denote by $H^k_0$ the closed subspace of $H^k$
of functions with homogeneous boundary values in the sense of closure of
$C^\infty_0$ with respect to the norm of $H^k$. 
For instance, $H^1_0$ 
denotes the closed nullspace of the Dirichlet trace operator $\gamma_0$.

To consider mixed boundary value problems on $\domain$, we partition $\Gamma = \partial \domain$
into two disjoint pieces $\Gamma_D$ and $\Gamma_N$. 
Assuming positive $(d-1)$-dimensional %surface
measure of $\Gamma_D$ if $d=2,3$, or
that $\Gamma_D$ contains at least one endpoint of $D$ if $d=1$, 
we set
\[
H^1_{\Gamma_D}(\domain) 
:= 
\{ v \in H^1(\domain) |\ \gamma_0(v)_{\mid_{\Gamma_D}} = 0 \}
\;.
\]
Evidently, for $\Gamma_D\subset\Gamma$,
$H^1_0(\domain) = H^1_\Gamma(\domain) \subset H^1_{\Gamma_D}(\domain) \subset H^1(\domain)$.

In the following, 
we introduce Sobolev spaces for functions defined on an interval $(a,b) \subset \IR$ with $a<b$. 
For simplicity, 
we consider real-valued functions 
$v \colon \, (a,b) \to \IR$. 
All results and proofs can be generalized straightforwardly 
to $X$-valued functions $v \colon \, (a,b) \to X$ for a Hilbert space $X$, i.e., Bochner--Sobolev spaces. 
We write 
\[
\begin{split}
H^1_{0,}(a,b) &= H^1_{\{a\}}(a,b) = \{v \in H^1(a,b) |\ v(a) = 0 \}, \\
%\;\;
H^1_{,0}(a,b) &= H^1_{\{b\}}(a,b) = \{v \in H^1(a,b) |\ v(b) = 0 \}\;.
\end{split}
\]
In either of these two spaces, the seminorm 
$| \circ |_{H^1(a,b)} = \| \partial_t \circ \|_{L^2(a,b)}$ is a norm. 
Thus, $| \circ |_{H^1(a,b)}$ is considered as the norm in $H^1_{0,}(a,b)$ and $H^1_{,0}(a,b)$, 
whereas the space $H^1(a,b)$ is endowed with the norm 
$\|\circ\|_{H^1(a,b)} = ( \|\circ\|_{L^2(a,b)}^2 + \|\partial_t \circ\|_{L^2(a,b)}^2 )^{1/2}$. 

Fractional order spaces shall be defined by interpolation, 
via the real method
of interpolation (see, e.g., \cite[Chapter~1]{Triebel}). 
We use the fine index $q=2$ to preserve the Hilbertian structure. 
Of particular interest will be the space
\[
H^{1/2}_{0,}(a,b) := (H^1_{0,}(a,b),L^2(a,b))_{1/2,2}
\;,
\]
where $|\circ|_{H^1(a,b)}= \| \partial_t \circ \|_{L^2(a,b)}$ is the norm of the space $H^1_{0,}(a,b)$.
The Sobolev space $H^{1/2}_{0,}(a,b)$ is a Hilbert space endowed 
with the interpolation norm (see \cite[Section~2.3]{OStMZ} for $(a,b)=(0,T)$) 
defined by
\be \label{eq:H120def}
\| v \|_{H^{1/2}_{0,}(a,b)}
:= \left(
 \sum_{k=0}^\infty \frac{\pi (2k+1)}{2(b-a)} |v_k|^2
\right)^{1/2}, \quad v \in H^{1/2}_{0,}(a,b),
\ee
where the Fourier coefficients $v_k$ are given by
$
v_k 
= 
\int_a^b v(s) V_k(s) \mathrm ds
\;.
$
Here, we use that any $z \in L^2(a,b)$ admits a representation as a Fourier series
\be \label{eq:FourierRepresentation}
    z(t) = \sum_{k=0}^\infty z_k V_k(t), \quad z_k = \int_a^b z(s) V_k(s) \mathrm ds, \; k \in \IN_0,
\ee
where $V_k$ denotes an 
eigenfunction corresponding to 
eigenvalue $\lambda_k = \frac{\pi^2 (2k+1)^2}{4(b-a)^2}$ 
of
\be \label{time:eigenvalues}
  -\partial_{tt} V_k(t) = \lambda_k V_k(t) \; \text{ for } t \in (a,b), \quad V_k(a)=\partial_t V_k(b)=0, \quad \norm{V_k}_{L^2(a,b)} = 1.
\ee
In particular for $J=(0,T)=(a,b)$, we have
\[
\| v \|_{H^{1/2}_{0,}(J)}
= \left( \frac{\pi}{2T} 
\sum_{k=0}^\infty  (2k+1) |v_k|^2
\right)^{1/2}, \quad v \in H^{1/2}_{0,}(J),
\]
with the Fourier representation
\be \label{FS:FourierSeries}
    v(t) = \sum_{k=0}^\infty v_k \sqrt{\frac{2}{T}} \sin\left(\left(\frac{\pi}{2}+k\pi\right)\frac{t}{T}\right), \; v_k = \int_0^T v(s) \sqrt{\frac{2}{T}} \sin\left(\left(\frac{\pi}{2}+k\pi\right)\frac{s}{T}\right) \mathrm ds.
\ee
Analogous to $H^{1/2}_{0,}(J)$, the Hilbert space
%\[
  $  H^{1/2}_{,0}(J) := (H^1_{,0}(J),L^2(J))_{1/2,2}$
%\]
is endowed with the Hilbertian norm (see \cite[Section~2.3]{OStMZ}) defined by
\[
\| w \|_{H^{1/2}_{,0}(J)}
:= \left( \frac{\pi}{2T} 
\sum_{k=0}^\infty  (2k+1) |w_k|^2
\right)^{1/2}, \quad w \in H^{1/2}_{,0}(J),
\]
where the Fourier coefficients are given by
$
w_k 
= 
 \int_0^T w(s) \sqrt{\frac{2}{T}} \cos\left(\left(\frac{\pi}{2}+k\pi\right)\frac{s}{T}\right) \mathrm ds.
$

To prove exponential convergence of a temporal $hp$-discretization, 
we need further investigations of the Sobolev space $H^{1/2}_{0,}(a,b)$
and its norm $\| \circ \|_{H^{1/2}_{0,}(a,b)}$. 
For this purpose, 
let the classical Sobolev space $H^{1/2}(a,b)$ be endowed with the Slobodetskii norm \cite[p.~74]{McLean2000}
\be \label{Sob:Slobodetskii}
\normiii{v}_{H^{1/2}(a,b)} := 
\left(
 \| v \|_{L^2(a,b)}^2
+ 
 |v|_{H^{1/2}(a,b)}^2 
 \right)^{1/2}  
\ee
for $v \in H^{1/2}(a,b)$ with
\be  \label{SlobodetskiiSemi}
    |v|_{H^{1/2}(a,b)} := \left( \int_a^b  \int_a^b \frac{|v(s)-v(t)|^2}{|s-t|^2} \mathrm ds \mathrm dt \right)^{1/2}.
\ee
With the Slobodetskii norm \eqref{Sob:Slobodetskii},
we endow $H^{1/2}_{0,}(a,b)$ with the norm
\be \label{Sob:NormTriple}
\normiii{v}_{H^{1/2}_{0,}(a,b)} := 
\left(
 \| v \|_{L^2(a,b)}^2
+ 
 |v|_{H^{1/2}(a,b)}^2 
+
 \int_a^b \frac{|v(t)|^2}{t-a} \mathrm dt
 \right)^{1/2}
\ee
for $v \in H^{1/2}_{0,}(a,b)$. 
We have the following 
equivalence result
for the norms defined in~\eqref{eq:H120def} and~\eqref{Sob:NormTriple},
which is proven, e.g., in~\cite{McLean2000} 
(see the proof in Appendix~\ref{sec:NormEquivalence} for the characterization of the equivalence constants).
\begin{lemma} \label{lem:NormEquivalence}
    There are constants $C_{\mathrm{Int},1}$, $C_{\mathrm{Int},2}>0$, which are independent of $a,b$, such that
    \begin{equation*}
        C_{\mathrm{Int},1} \| v \|_{H^{1/2}_{0,}(a,b)} 
        \leq 
        \normiii{v}_{H^{1/2}_{0,}(a,b)} 
        \leq 
        C_{\mathrm{Int},2} \sqrt[4]{ 1 + \frac{4(b-a)^2}{\pi^2}} \| v \|_{H^{1/2}_{0,}(a,b)}
    \end{equation*}
    for all $v \in H^{1/2}_{0,}(a,b)$.
\end{lemma}
The next result is used for the proof of the temporal $hp$-error estimate in Section~\ref{sec:ConvRate}.
It localizes the $H^{1/2}(a,b)$ norm in a certain sense. 
We report its proof in Appendix~\ref{sec:NormEquivalence}, and refer to
\cite{Faermann2000} for a more general localization result. 
\begin{lemma} \label{lem:FractionalNormPointtau}
 For a number $\tau \in (a,b)$, the estimate
 \begin{equation*}
    | v |_{H^{1/2}(a,b)}^2 
    \leq  
    | v |_{H^{1/2}(a,\tau)}^2  + 4 \int_a^{\tau} \frac{|v(t)|^2}{\tau-t} \mathrm dt  
    + 4 \int_{\tau}^b \frac{|v(s)|^2}{s-\tau} \mathrm ds  +   | v |_{H^{1/2}(\tau,b)}^2
 \end{equation*}
 holds true for $v \in H^{1/2}(a,b)$, if all occurring integrals on the right side exist.
\end{lemma}

%%%%%%%%%%%%%%%%%%%%%%%%%%%%%%%%%%%%%%%%%%%%%%%%%%%%%%%%%%%%%%%%%%%%%%%%%%%%%%
\subsection{Hilbert Transformation $\HT$}
\label{sec:HilbTr}
%%%%%%%%%%%%%%%%%%%%%%%%%%%%%%%%%%%%%%%%%%%%%%%%%%%%%%%%%%%%%%%%%%%%%%%%%%%%%%
A key role in the space-time variational formulation of IBVP \eqref{eq:IBVP} 
is taken by the nonlocal operator $\HT\in \cL(L^2(J),L^2(J))$,
which is defined by
\be\label{eq:DefHT}
(\HT v)(t) 
:= 
\sum_{k=0}^\infty v_k \sqrt{\frac{2}{T}} \cos\left(\left(\frac{\pi}{2}+k\pi\right)\frac{t}{T}\right), \quad t \in J.
\ee
Here, $v \in L^2(J)$ and its Fourier coefficients $v_k = \int_0^T v(s) \sqrt{\frac{2}{T}} \sin\left(\left(\frac{\pi}{2}+k\pi\right)\frac{s}{T}\right) \mathrm ds$ are represented as in \eqref{FS:FourierSeries}.
We collect some properties of $\HT$. 
\begin{proposition}[{\cite[Section~2.4]{OStMZ}, \cite{SteinbachZankNoteHT,ZankExactRealizationHT}}] \label{prop:HT}
The modified Hilbert transformation $\HT$ defined in \eqref{eq:DefHT} is a linear isometry as mapping
\be\label{eq:HT1}
    \HT \colon H^{\nu}_{0,} (J) \to H^{\nu}_{,0} (J) \quad \text{ for } \nu \in \{0,  1/2, 1 \}
\ee
and is $H^{1/2}_{0,}(J)$-elliptic, satisfying
\be\label{eq:HT2}
\forall v\in H^{1/2}_{0,}(J):\quad 
\langle \partial_t v , \HT v \rangle_{L^2(J)} 
=
\| v \|^2_{H^{1/2}_{0,}(J)}.
\ee
Additionally, $\HT$ fulfills the following properties:
\be\label{eq:HT3} 
\forall v,w \in H^{1/2}_{0,}(J) : \quad
\langle \partial_t w , \HT v \rangle_{L^2(J)}
=
\langle \HT w,\partial_t v \rangle_{L^2(J)}
=
\langle w,v \rangle_{H^{1/2}_{0,}(J)}
\;,
\ee
\be\label{eq:HT4} 
\forall w \in H^1_{0,}(J), \forall v \in L^2(J) : \quad 
\langle \partial_t \HT w,v \rangle_{L^2(J)} 
=
-\langle \HT^{-1} \partial_t w,v \rangle_{L^2(J)} 
\;,
\ee
\be\label{eq:HT5} 
\forall v, w\in L^2(J) :
\quad 
\langle \HT v,w \rangle_{L^2(J)} 
=
\langle v, \HT^{-1} w \rangle_{L^2(J)}
\;,
\ee
\be\label{eq:HT6} 
\forall v \in L^2(J): \quad \langle v, \mathcal{H}_T v \rangle_{L^2(J)} \geq 0
\;,
\ee
\be\label{eq:HT7} 
\forall \nu \in \{1/2, 1\}, \forall v \in H^{\nu}_{0,}(J), v \neq 0 : \quad \langle v, \mathcal{H}_T v \rangle_{L^2(J)} > 0
\;.
\ee
\end{proposition}
\begin{remark}\label{rmk:Tinfty}
We remark that \eqref{eq:HT1}--\eqref{eq:HT7} are valid for all $T>0$.
In particular, these identities remain stable under passage to the
limit $T\to \infty$, with appropriate modifications of spaces.
We refer to \cite{DD20} for a space-time variational formulation
and a discussion of a Petrov--Galerkin discretization for the resulting
limiting problems.
\end{remark}

%%%%%%%%%%%%%%%%%%%%%%%%%%%%%%%%%%%%%%%%%%%%%%%%%%%%%%%%%%%%%%%%%%%%%%%%%%%%%%
\subsection{Model Scalar Initial Value Problem}
\label{sec:IVP}
%%%%%%%%%%%%%%%%%%%%%%%%%%%%%%%%%%%%%%%%%%%%%%%%%%%%%%%%%%%%%%%%%%%%%%%%%%%%%%
In $J=(0,T)$, for a given right-hand side $f$, 
consider the scalar IVP to find a function $u:J\to \IR$
such that
\[
    \partial_t u = f \text{ in } J, \quad u(0) = 0 \;.
\]
A weak formulation relevant for treatment of IBVP \eqref{eq:IBVP} is 
to find $u\in H^{1/2}_{0,}(J)$ such that
\be\label{eq:IVPw}
    \forall w \in H^{1/2}_{,0}(J) : \; \langle \partial_t u,w \rangle_{L^2(J)} = \langle f , w \rangle_{L^2(J)}
\ee
for given $f\in [H^{1/2}_{,0}(J)]'$. Here, $\langle \circ, \circ \rangle_{L^2(J)}$ denotes the inner product in $L^2(J)$ and as continuous extension of it, also the duality pairing with respect to $[H^{1/2}_{,0}(J)]'$ and $H^{1/2}_{,0}(J)$.
The continuous bilinear form on the left side of \eqref{eq:IVPw} is inf-sup stable:
\be\label{eq:infsup}
\inf_{0\ne u\in H^{1/2}_{0,}(J)} \sup_{0 \ne w \in H^{1/2}_{,0}(J)}
\frac{\langle \partial_t u , w \rangle_{L^2(J)} }{\| u \|_{H^{1/2}_{0,}(J)} \| w \|_{H^{1/2}_{,0}(J)}}
\geq 1 \;.
\ee
This is shown in~\cite[Rem.~2.10]{OStMZ} by 
observing that, for every $u\in H^{1/2}_{0,}(J)$,
\[
\| u \|_{H^{1/2}_{0,}(J)} 
= 
\frac{\langle \partial_t u , \HT u\rangle_{L^2(J)}}{\| \HT u \|_{H^{1/2}_{,0}(J)}}
\leq 
\sup_{0\ne w \in H^{1/2}_{,0}(J)} \frac{\langle \partial_t u , w \rangle_{L^2(J)} }{ \| w \|_{H^{1/2}_{,0}(J)} }
\;.
\]
For every $f\in [H^{1/2}_{,0}(J)]'$, 
IVP \eqref{eq:IVPw} then admits a unique solution $u\in H^{1/2}_{0,}(J)$.

For the derivation of a space-time variational formulation of \eqref{eq:IBVP}, 
it is useful to consider a \emph{parametric IVP}: 
for a given parameter $\mu\geq 0$ (eventually in the spectrum of the spatial
operator of \eqref{eq:IBVP})
and for $f\in  [H^{1/2}_{,0}(J)]'$, 
find $u\in H^{1/2}_{0,}(J)$ such that
$\partial_t u + \mu u = f $ in $ [H^{1/2}_{,0}(J)]' $. 
A \emph{Petrov--Galerkin variational form} of this problem
is to find $u\in H^{1/2}_{0,}(J)$ such that 
\be\label{eq:pIVPPG} 
\forall w \in H^{1/2}_{,0}(J) : \, \langle \partial_t u, w\rangle_{L^2(J)} + \mu \langle u, w \rangle_{L^2(J)} 
= 
\langle f,w \rangle_{L^2(J)}
\;.
\ee
A \emph{Bubnov--Galerkin variational form} 
with equal trial and test function spaces
is to find $u\in H^{1/2}_{0,}(J)$ such that
\be\label{eq:pIVPBG}
\forall v\in H^{1/2}_{0,}(J) : \, \langle \partial_t u , \HT v \rangle_{L^2(J)} 
+ 
\mu \langle u,\HT v \rangle_{L^2(J)} 
=
\langle f,\HT v \rangle_{L^2(J)}
\;.
\ee
Both formulations \eqref{eq:pIVPPG} and \eqref{eq:pIVPBG} admit unique solutions
due to \eqref{eq:infsup} and \eqref{eq:HT6}.

%%%%%%%%%%%%%%%%%%%%%%%%%%%%%%%%%%%%%%%%%%%%%%%%%%%%%%%%%%%%%%%%%%%%%%%%%%%%%%
\subsection{Temporal $hp$-Discretization}
\label{sec:thp}
%%%%%%%%%%%%%%%%%%%%%%%%%%%%%%%%%%%%%%%%%%%%%%%%%%%%%%%%%%%%%%%%%%%%%%%%%%%%%%
To discretize \eqref{eq:pIVPBG}, we use 
some space $V^M_t \subset H^{1/2}_{0,}(J)$
of finite dimension $M = {\rm dim}(V^M_t)$.
In the $hp$-time discretization, we build $V^M_t$ as follows:
on a partition $\cG = \{ I_j \}_{j=1}^m$ of $J$  into $m$ time intervals
$I_j:= (t_{j-1},t_j)$, where $0 =: t_0 < t_1 < \dots < t_m := T$,
we choose $V^M_t$ as a space of continuous, piecewise polynomials 
of degrees $p_j \geq 1$, which we collect in the 
\emph{degree vector} 
$\bmp := (p_j)_{j=1}^m\in \IN^m$.  
We define
\be\label{eq:DefVMt}
V^M_t = S^{\bmp,1}_{0,}(J;\cG) 
:= 
\{ v \in H^1_{0,}(J): v_{\mid_{I_j}} \in \IP^{p_j},\; I_j \in \cG\}
\;.
\ee
Here, continuity between adjacent time-intervals is required to ensure 
$S^{\bmp,1}_{0,}(J;\cG) \subset H^{1/2}_{0,}(J)$.
Then 
$M = {\rm dim}(S^{\bmp,1}_{0,}(J;\cG)) = \left( \sum_{j=1}^m (p_j+1)\right) - m = \sum_{j=1}^m p_j$.

We restrict~\eqref{eq:pIVPBG} to $V^M_t$ to obtain the temporal 
$hp$-approximation:
find $u^M_t\in V^M_t$ such that
\be\label{eq:pIVPBGhp}
\forall v\in  V^M_t : \,
\langle \partial_t u^M_t , \HT v \rangle_{L^2(J)}
+
\mu \langle u^M_t , \HT v \rangle_{L^2(J)}
=
\langle f,\HT v \rangle_{L^2(J)}
\;.
\ee
Due to the inf-sup stability \eqref{eq:infsup} and $\mu\geq 0$, 
the discretization \eqref{eq:pIVPBGhp} is well-posed with inf-sup
constant independent of $\cG$ and of $\bmp$.
Its numerical implementation will require, similar to \cite{OStMZ,DD20},
the efficient evaluation of $\HT v$ for $v\in V^M_t$.
We shall address this in Section \ref{sec:NumHT}
below.

%%%%%%%%%%%%%%%%%%%%%%%%%%%%%%%%%%%%%%%%%%%%%%%%%%%%%%%%%%%%%%%%%%%%%%%%%%%%%%
\subsection{Space-Time Variational Formulation}
\label{sec:xtVarForm}
%%%%%%%%%%%%%%%%%%%%%%%%%%%%%%%%%%%%%%%%%%%%%%%%%%%%%%%%%%%%%%%%%%%%%%%%%%%%%%
We consider the source problem corresponding to the spatial part
of \eqref{eq:IBVP}. 
Its variational form reads: given a source term $f\in L^2(\domain)$,
find 
\be\label{eq:xPbm}
w\in H^1_{\Gamma_D}(\domain) \;\; \mbox{such that } \forall v\in  H^1_{\Gamma_D}(\domain) : \,
a(w,v) = \langle f,v \rangle_{L^2(\domain)} \;.
\ee
Here, $a(w,v) = \int_\domain A(x)\nabla_x w(x) \cdot \nabla_x v(x) \mathrm dx$.
We assume uniform positive definiteness of~$A$:
\be\label{eq:Acoerc}
a_{\min}:= 
\essinf_{x\in \domain} 
          \inf_{0 \ne \xi\in \IR^d} \frac{\xi^\top A(x) \xi}{\xi^\top\xi} > 0  
\;.
\ee
With assumption \eqref{eq:Acoerc}, we have
\[
\forall w\in  H^1_{\Gamma_D}(\domain): 
\;\;
a(w,w) \geq a_{\min} \| \nabla_x w \|_{L^2(\domain)}^2 \geq a_{\min}c\| w \|_{H^1(\domain)}^2
\]
due to $|\Gamma_D|>0$ if $d=2,3$ or $\Gamma_D\ne\emptyset$ if $d=1$, 
and the Poincar\'{e} inequality.

The spectral theorem and the symmetry $a(w,v) = a(v,w)$ for all $v,w\in H^1(\domain)$
ensure that the corresponding eigenvalue problem to find 
\be\label{eq:xtevp}
0 \ne \phi\in H^1_{\Gamma_D}(\domain), \; \mu\in \IR: \;\; \forall v\in  H^1_{\Gamma_D}(\domain) : \,
a(\phi,v) = \mu \langle \phi,v \rangle_{L^2(\domain)}
\ee
admits a sequence of eigenpairs $\{(\mu_k,\phi_k)\}_{k\geq 1}$ enumerated
in increasing order of the real eigenvalues $\mu_k > 0$, repeated according to multiplicity,
with the eigenfunctions $\phi_k$ orthonormal in $L^2(\domain)$ and orthogonal in $H^1_{\Gamma_D}(\domain)$, 
and with $\mu_k$ accumulating only at $\infty$.
In view of the forthcoming analysis, 
\emph{in what follows, we endow $H^1_{\Gamma_D}(\domain)$ with the ``energy'' norm $a(\circ,\circ)^{1/2}$}.
We remark that, for $v\in H^1_{\Gamma_D}(\domain)$, 
$a(v,v)=\sum_{i=1}^\infty \mu_i|v_i|^2$, 
where $v_i=\langle v,\phi_i \rangle_{L^2(\domain)}$.

The space-time variational formulation of \eqref{eq:IBVP}
is based on the intersection space
\[
H^{1,1/2}_{\Gamma_D;0,}(Q) 
:= 
\left(L^2(J) \otimes H^1_{\Gamma_D}(\domain)\right) \cap \left(H^{1/2}_{0,}(J)\otimes L^2(\domain)\right)\;,
\]
which we equip with the corresponding sum norm.
The space $H^{1,1/2}_{\Gamma_D;,0}(Q)$ is defined analogously.
Proceeding as in \cite[Thm.~3.2]{OStMZ},
the \emph{initial-boundary value problem \eqref{eq:IBVP}--\eqref{eq:BC} 
    is set as a well-posed operator equation}.
\begin{theorem}\label{thm:Biso}
Consider \eqref{eq:IBVP}--\eqref{eq:BC} with homogeneous data $u_0 = 0$
in~\eqref{eq:IC} and $u_D, u_N = 0$ in~\eqref{eq:BC}. Assume
$|\Gamma_D|>0$ if $d=2,3$ or $\Gamma_D\ne\emptyset$ if $d=1$, 
and that the coefficient
$A\in L^\infty(\domain;\IR^{d\times d}_{\mathrm{sym}})$ satisfies~\eqref{eq:Acoerc}.

Then, the space-time variational formulation of \eqref{eq:IBVP} to
find $u\in H^{1,1/2}_{\Gamma_D;0,}(Q)$ such that
\be\label{eq:IBVPVar}
    \forall v\in H^{1,1/2}_{\Gamma_D;,0}(Q) : \, \langle \partial_t u , v \rangle_{L^2(Q)} 
    +
    \langle A \nabla_x u, \nabla_x v \rangle_{L^2(Q)} 
    = 
    \langle g, v \rangle_{L^2(Q)} 
\ee
induces an isomorphism 
\[
B:= \partial_t + A(\partial_x)  \in \cL_{\mathrm{iso}}(H^{1,1/2}_{\Gamma_D;0,}(Q),[H^{1,1/2}_{\Gamma_D;,0}(Q)]')\;.
\]
In particular, 
for every $g\in [H^{1,1/2}_{\Gamma_D;,0}(Q)]'$, IBVP $Bu = g$
in \eqref{eq:IBVP} 
admits a unique solution $u\in H^{1,1/2}_{\Gamma_D;0,}(Q)$.
\end{theorem}
We remark that $\langle \circ, \circ \rangle_{L^2(Q)}$ denotes the inner product in $L^2(Q)$ and as continuous extension of it, also the duality pairing with respect to $[H^{1,1/2}_{\Gamma_D;,0}(Q)]'$ and $H^{1,1/2}_{\Gamma_D;,0}(Q)$.
The \emph{space-time discretization} of \eqref{eq:IBVPVar} 
is straightforward: 
for any conforming,
spatial finite element subspace $V^N_x \subset H^1_{\Gamma_D}(\domain)$ of finite dimension $N$, 
and for the temporal $hp$-subspace $V^M_t\subset H^{1/2}_{0,}(J)$ introduced in \eqref{eq:DefVMt},
we restrict \eqref{eq:IBVPVar} to the space-time approximation space
\be\label{eq:xtApprSpc}
V^M_t\otimes V^N_x \subset H^{1,1/2}_{\Gamma_D;0,}(Q) \;.
\ee
That is, we seek an approximate solution 
$u^{MN}\in V^M_t\otimes V^N_x$ such that
\be\label{eq:IBVPVarxt}
\langle \partial_t u^{MN} , v \rangle_{L^2(Q)}
+
\langle A\nabla_x u^{MN}, \nabla_x v \rangle_{L^2(Q)}
=
\langle g, v \rangle_{L^2(Q)}
\ee
holds true for all $v\in (\HT V^M_t) \otimes V^N_x \subset
H^{1,1/2}_{\Gamma_D;,0}(Q)$.

For these choices of test function spaces and for \emph{any} subspace $V^N_x\subset V$ of finite dimension $N$,
as in ~\cite[Sect.~3]{OStMZ}, 
existence and uniqueness of the discrete solution 
$u^{MN}\in V^M_t\otimes V^N_x\subset H^{1,1/2}_{\Gamma_D;0,}(Q)$ 
of~\eqref{eq:IBVPVarxt}
follow from the \emph{continuous inf-sup condition}
\begin{equation*}
\inf_{0\ne u\in H^{1,1/2}_{\Gamma_D;0,}(Q)} \sup_{0 \ne w \in H^{1,1/2}_{\Gamma_D;,0}(Q)}
\frac{\langle \partial_t u , w \rangle_{L^2(Q)}
+\langle A\nabla_x u, \nabla_x w \rangle_{L^2(Q)}}{\| u \|_{H^{1,1/2}_{\Gamma_D;0,}(Q)} \| w \|_{H^{1,1/2}_{\Gamma_D;,0}(Q)}}
\geq \frac12 \;.
\end{equation*}
With $H^1_{\Gamma_D}(\domain)$ endowed with the $a(\circ,\circ)^{1/2}$
  norm, the proof of this condition with constant independent of
  $A$ follows \emph{verbatim} that of~\cite[Thm.~3.2, Cor.~3.3]{OStMZ}
  for the case $A=\mathbb{I}$.
Evidently, the stability of the discrete problem 
is a consequence of the choice of the
test function space~$\HT V^M_t$, 
whose efficient numerical realization will be discussed in
Section~\ref{sec:NumExp}.

%%%%%%%%%%%%%%%%%%%%%%%%%%%%%%%%%%%%%%%%%%%%%%%%%%%%%%%%%%%%%%%%%%%%%%%%%%%%%%
\section{Regularity}
\label{sec:Reg}
%%%%%%%%%%%%%%%%%%%%%%%%%%%%%%%%%%%%%%%%%%%%%%%%%%%%%%%%%%%%%%%%%%%%%%%%%%%%%%
To obtain convergence rate bounds, we address the regularity of the solution
$u \in H^{1,1/2}_{\Gamma_D;0,}(Q)$. We consider separately the temporal and spatial
regularity. 
The solution operator to the parabolic equation \eqref{eq:IBVP} 
being an analytic semigroup, for time-analytic forcing $g$ in \eqref{eq:IBVP}
we expect time-analyticity of $u$.
This, in turn, is well-known to imply 
exponential convergence of $hp$-time-stepping as shown, e.g., 
in \cite{SS00_340,DD20} and the references there.
We shall verify this in Sections \ref{sec:Approx} and \ref{sec:ConvRate} below.

%%%%%%%%%%%%%%%%%%%%%%%%%%%%%%%%%%%%%%%%%%%%%%%%%%%%%%%%%%%%%%%%%%%%%%%%%%%%%%
\subsection{Time-Analyticity}
\label{sec:tReg}
%%%%%%%%%%%%%%%%%%%%%%%%%%%%%%%%%%%%%%%%%%%%%%%%%%%%%%%%%%%%%%%%%%%%%%%%%%%%%%
We quantify the temporal analyticity of the solution $u:t\mapsto u(t)$ with $u(t) := u(t,\circ) \in L^2(\domain)$.
To this end, we recall the eigenvalue problem~\eqref{eq:xtevp}.
Setting $H:=L^2(\domain)$ and thus denoting by 
$\langle \circ,\circ \rangle_H$ the $L^2(\domain)$ inner product, 
the solution $u(t)$ of~\eqref{eq:IBVP} at time $t>0$ 
for $g=0$, $u_D = u_N = 0$, and for initial data $u_0\in H$ 
may be written as
\be\label{eq:T(t)}
u(t) = \semigroup(t)u_0 := \sum_{i=1}^\infty \exp(-\mu_i t) \langle u_0,\phi_i \rangle_H \phi_i 
\ee
with convergence of the series in $H$. 
The operators $\{ \semigroup(t) \}_{t\geq 0}$ satisfy the semigroup property
in~$H$, i.e.,
\[
\forall s,t > 0: \;\; \semigroup(s+t) = \semigroup(s)
\semigroup(t),\;\; \semigroup(0) = \operatorname{Id}\;.
\]
For $\s\geq 0$, we define
the scale of spaces $X_\s\subset H=X_0$
\be\label{eq:Xs}
X_\s := \{ v\in H: \| v \|_{X_\s}^2 := \sum_{i=1}^\infty \mu_i^\s |v_i|^2 <\infty \}\;.
\ee
Here, $v_i = \langle v,\phi_i \rangle_H$ denotes the $i$-th coefficient in the eigenfunction expansion
of $v$ (recall from \eqref{eq:xtevp} that the sequence $\{\phi_i\}_{i\geq 1}$ was assumed to be
an orthonormal basis of $H = X_0$).
We remark that the norm $\| \circ \|_{X_1}$ is the energy-norm on the space 
$V = H^1_{\Gamma_D}(\domain)$, due to
\[
  \forall v\in V:\quad 
  \| v \|_{X_1}^2 = a(v,v) = \sum_{i=1}^\infty \mu_i |v_i|^2 \;.
\]
For $|\Gamma_D|>0$ if $d=2,3$ or $\Gamma_D\ne\emptyset$ if $d=1$, 
$\| \circ \|_{X_1}$ is equivalent to the $H^1(\domain)$ norm on $V$
and the norm bounds 
$\| v \|_{X_\s} \leq c \| v \|_{X_{\s'}}$ for $\s'\geq \s$ 
follow from \eqref{eq:Xs} and the assumed enumeration of the
real eigenvalues $\mu_i>0$ with $\mu_i\uparrow\infty$
as $i\uparrow\infty$:
\be \label{ineq:XrEmbedding}
  \| v \|_{X_\s}^2 
  = \sum_{i=1}^\infty \mu_i^\s |v_i|^2 
  \leq  
  \left( \sup_{m \in \IN} \mu_m^{\s-\s'} \right) \sum_{i=1}^\infty \mu_i^{\s'}|v_i|^2
  \leq  \mu_1^{\s-\s'} \| v \|_{X_{\s'}}^2.
\ee
For $\theta,\s\geq 0$ and for any $t>0$, $\semigroup(t)$ in
  \eqref{eq:T(t)} belongs to $\cL(X_\theta,X_\s)$. In fact, for any
  $t>0$ and $v\in
  X_\theta$, we have
\be\label{eq:newEt}
\| \semigroup(t) v \|_{X_\s}^2  
=
\displaystyle
\sum_{i=1}^\infty \mu_i^\s \exp(-2\mu_i t) |v_i|^2 
=
\sum_{i=1}^\infty \mu_i^{\s-\theta} \exp(-2\mu_i t) \mu_i^\theta |v_i|^2.
\ee
For $\theta \geq  \s\geq 0$, identity~\eqref{eq:newEt} implies
\[
  \| \semigroup(t) v \|_{X_\s}^2
  \leq 
  \mu_1^{-(\theta-\s)} \exp(-2\mu_1 t)\sum_{i=1}^\infty \mu_i^\theta |v_i|^2
  =
  \mu_1^{-(\theta-\s)} \exp(-2\mu_1 t) \| v \|_{X_\theta}^2
\]
for all $v\in X_\theta$, i.e., $\semigroup(t) \in \cL(X_\theta,X_\s)$ for any $t>0$ with
\be\label{eq:gstThetaGreaterR}
 \forall \theta\geq \s\geq 0, \; \forall t>0 : 
\quad      \| \semigroup(t) \|_{\cL(X_\theta,X_\s)}^2 \leq \mu_1^{-(\theta-\s)} \exp(-2\mu_1 t) \;.
\ee
For $\s\geq \theta \geq 0$, for any $t>0$ and $v\in
  X_\theta$, identity~\eqref{eq:newEt} implies
\be\label{eq:gst}
\| \semigroup(t) v \|_{X_\s}^2
\leq  
\displaystyle
\sup_{i \in \IN} \{ \mu_i^{\s-\theta} \exp(-2\mu_i t) \} \sum_{i=1}^\infty \mu_i^\theta |v_i|^2
=:
\displaystyle
G_{\s-\theta}(t) \| v \|_{X_\theta}^2 \;.
\ee
To provide an upper bound for $G_{\s-\theta}(t)$, 
we observe that, for fixed $t, \sigma>0$, the function 
$0 < \mu \mapsto \mu^{2\sigma}\exp(-2\mu t)$ 
takes its maximum at $\mu_*:= \sigma / t $ whence
\be\label{eq:bstBound}
\forall t>0: \quad 
G_{2\sigma}(t) \leq G_{\max}(\sigma,t) := [\mu_*^{\sigma}\exp(-\mu_* t)]^2
=
\left( \frac{\sigma}{t \mathrm{e}} \right)^{2\sigma} 
\;.
\ee
Inserting~\eqref{eq:bstBound} with
$\sigma = (\s-\theta)/2 > 0$ into~\eqref{eq:gst},
we arrive at 
\[
\forall \s\geq \theta \geq 0, \; \forall t>0 : 
\quad 
\| \semigroup(t)\|_{\cL(X_\theta,X_\s)}^2\leq
\left( \frac{\s-\theta}{2t \mathrm{e}} \right)^{\s-\theta}\;.
\]
The exponential decay of the Fourier coefficients for $t>0$ 
implied by the exponential weighting $\exp(-\mu_it)$
entails time-analyticity of the solution $t\mapsto u(t)$ for $t>0$.
To prove exponential convergence rates of $hp$-approximation  
in $J = (0,T)$, we quantify the time regularity of the solution  $u$
of \eqref{eq:IBVP} for $u_0=0$ and $u_D=u_N=0$ with the Duhamel representation  
(see, e.g.,~\cite{Pazy})
\be\label{eq:Duhamel}
u(t) = \int_0^t \semigroup(t-s)g(s)\mathrm ds, \quad 0<t\leq T\;.
\ee
We work under the following
\emph{time-analyticity assumption} on the forcing $g$ in \eqref{eq:IBVP}: 
There exist constants $C>0$ and $\dtime \ge 1$
such that, for some $\varepsilon \in (0,1)$, 
we have
\be\label{eq:TimeReg}
\forall \;l\in\IN_0:\;\; 
\sup_{0\leq t \leq T} 
\| g^{(l)}(t) \|_{X_\varepsilon} \leq C \dtime^l\Gamma(l+1),
\ee
where $\Gamma(\circ)$ denotes the gamma function fulfilling $\Gamma(l) = (l-1)!$ for all $l \in \IN$.
\emph{Formally differentiating} \eqref{eq:Duhamel} 
$l$-times with respect to $t$, 
upon writing it equivalently as $u(t) = \int_0^t \semigroup(s)g(t-s)\mathrm ds$,
gives
\be\label{eq:Tlg}
\frac{\mathrm d^l}{\mathrm dt^l}u(t) 
=
\sum_{i=0}^{l-1} \semigroup^{(i)}(t) g^{(l-1-i)}(0) 
+
\int_0^t \semigroup(s) g^{(l)}(t-s) \mathrm ds\;,
\quad 
l\in \IN, t>0\;.
\ee
The right limits at $t=0$ of the time-derivatives of the forcing $g$ in \eqref{eq:IBVP} 
contribute to the time-regularity.
We estimate the norm of the operators $\semigroup^{(l)}(t)$ in $\cL(X_\theta, X_\s)$.
\begin{lemma}\label{lem:Tt}
For $\s\geq \theta \geq 0$, we have
\be\label{eq:T'TBd}
\forall l\in \IN_0, \;\forall t>0 :\;\;
\| \semigroup^{(l)}(t) \|_{\cL(X_\theta,X_\s)}^2
\leq 
\frac{1}{\sqrt{2\pi}}\left(\frac12\right)^{2l+\s-\theta} %d^{2l+s-\theta}
\Gamma(2l+1+\s-\theta)\, t^{-2l-(\s-\theta)} \;.
\ee
\end{lemma}
\begin{proof}
For $v\in H=X_0$, with $v_i = \langle v,\phi_i \rangle_H$,
the time-derivative of order $l\in \IN$ applied to $v(t) = \semigroup(t)v$
represented as in \eqref{eq:T(t)} yields (with formal, term-by-term
differentiation)
\[
\frac{\mathrm d^l}{\mathrm dt^l}v(t) 
=
\sum_{i=1}^\infty (-\mu_i)^l \exp(-\mu_i t) v_i \phi_i 
\]
with convergence in $H$ for arbitrary, fixed $t>0$. 
Therefore
\[
\forall t>0:\;\;
\| v^{(l)}(t) \|_{X_\s}^2
=
\sum_{i=1}^\infty \mu_i^{2l+\s-\theta} \exp(-2\mu_i t) \mu_i^\theta |v_i|^2
\;.
\]
It follows from \eqref{eq:bstBound} that 
for every $t>0$
\[
\| v^{(l)}(t) \|_{X_\s}^2
\leq 
G_{2(l+[\s-\theta]/2)}(t) \| v \|_{X_\theta}^2 
\leq 
G_{\max}(l+[\s-\theta]/2,t) \| v \|_{X_\theta}^2
\;.
\]
Therefore, for every $v\in X_\theta$ and every $\s\geq \theta \geq 0$, 
we have
\[
  \begin{split}
\forall l\in \IN, t>0:\quad 
\| v^{(l)}(t) \|_{X_\s}^2
&\leq
\left(\frac{2l+ \s-\theta}{2t\mathrm{e}}\right)^{2l+\s-\theta} \| v
\|_{X_\theta}^2\\
&=\left(\frac12\right)^{2l+\s-\theta}\left(\frac{2l+\s-\theta}{\mathrm{e}}\right)^{2l+\s-\theta}
t^{-2l-(\s-\theta)}\| v\|_{X_\theta}^2 
\;.
\end{split}
\]
For $x\in\mathbb R_+$, the Stirling's formula \eqref{Stirling} states
  $\sqrt{2\pi}\,x^{x-1/2}\mathrm{e}^{-x}\le \Gamma(x)$,
which implies 
$(x/\mathrm{e})^x\le \frac{1}{\sqrt{2\pi}} x^{-1/2}\Gamma(x+1)$. With $x=2l+\s-\theta$,
this gives the claimed bound, as $(2l+\s-\theta)^{-1/2}\le 1$.
\end{proof}
\begin{lemma}\label{lem:DtBound}
Assume \eqref{eq:TimeReg} with some $\varepsilon \in (0,1)$ and some $\dtime \geq 1$.
For $\s \in [0,2]$, there exists a constant $C>0$ (independent of $\dtime,$ $l$, $t$) such that,
for every $l\in \IN_0$ and $t>0$, 
we have
\be\label{eq:dtlu}
\| u^{(l)}(t) \|_{X_\s} 
\leq 
C \dtime^l\Gamma(l+1) 
\left( t^{(2-\s+\min\{r,\varepsilon\})/2} + \sum_{i=0}^{l-1} t^{-i-\s/2 + \varepsilon/2} \right) 
\;.
\ee
For $l=0$, this bound is valid without the sum.
\end{lemma}
\begin{proof}
From \eqref{eq:Tlg}, we estimate for every $0<t\leq T$
\begin{multline*}
\| u^{(l)}(t) \|_{X_\s} 
\leq 
\sum_{i=0}^{l-1} \|\semigroup^{(i)}(t)\|_{\cL(X_\varepsilon,X_\s)} \| g^{(l-i-1)}(0) \|_{X_\varepsilon} \\
+
\int_0^t \|\semigroup(s)\|_{\cL(X_\varepsilon ,X_\s)}
\|g^{(l)}(t-s) \|_{X_\varepsilon}\mathrm{d}s.
\end{multline*}
To estimate the sum, we use \eqref{eq:T'TBd} with $\theta = \varepsilon$ and
assumption \eqref{eq:TimeReg} and obtain
\begin{align*}
  \sum_{i=0}^{l-1} &\|\semigroup^{(i)}(t)\|_{\cL(X_\varepsilon,X_\s)} \| g^{(l-i-1)}(0) \|_{X_\varepsilon} \\
  &\leq \sum_{i=0}^{l-1} C \left(\frac12\right)^{i+r/2-\varepsilon/2} \Gamma(2i+1+\s-\varepsilon)^{1/2} t^{-i-\s/2+\varepsilon/2}\dtime^{l-1-i}\Gamma(l-i) \\
  &\leq  C {\dtime}^{l-1} \sum_{i=0}^{l-1}\left(\frac12\right)^{i+r/2-\varepsilon/2} \Gamma(2i+1+\s-\varepsilon)^{1/2}\Gamma(l-i)\, t^{-i-\s/2+\varepsilon/2} \\
  &\leq C {\dtime}^{l-1} \sum_{i=0}^{l-1} \Gamma(i+1+\s/2-\varepsilon/2)\Gamma(l-i)\, t^{-i-\s/2+\varepsilon/2} \\
  &\leq C {\dtime}^{l-1} \Gamma(l+1) \sum_{i=0}^{l-1} t^{-i-\s/2+\varepsilon/2} \;,
\end{align*}
where in the third inequality we have used the duplication formula
  $\Gamma(z)\Gamma(z+1/2)=2^{1-2z}\sqrt{\pi}\Gamma(2z)$ with
  $z=(2i+1+r-\varepsilon)/2$, and the fourth inequality follows from
  $\max_{0\le i\le l-1}\Gamma(i+1+\s/2-\varepsilon/2)\Gamma(l-i)\le
    \max_{0\le i\le l-1}\Gamma(i+2)\Gamma(l-i)\le \Gamma(l+1)$.
\\
To estimate the integral term for $\varepsilon \leq \s\leq 2$, 
we use assumption \eqref{eq:TimeReg} with $\varepsilon \in (0,1)$
and \eqref{eq:T'TBd} with $l=0$, $\theta = \varepsilon > 0$ and
$\varepsilon \leq \s\leq 2$, and obtain, for every $l\in \IN_0$,
\begin{align*}
    \int_0^t \| \semigroup(s) \|_{\cL(X_\varepsilon, X_\s)} \| g^{(l)}(t-s) \|_{X_\varepsilon} \mathrm{d}s
    &\leq 
    C \dtime^l \Gamma(l+1) \int_0^t s^{-(\s-\varepsilon)/2} \mathrm{d}s  \\
    &=
    C C_\varepsilon \dtime^l \Gamma(l+1) t^{(2-\s+\varepsilon)/2}
\end{align*}
with $C_\varepsilon=2/(2-\s+\varepsilon)$.
It remains to estimate the integral term for $0 \leq \s \leq
\varepsilon$. In this case, for every $l \in \IN_0$, we have
\begin{align*}
    \int_0^t \| \semigroup(s) \|_{\cL(X_\varepsilon, X_\s)} \| g^{(l)}(t-s) \|_{X_\varepsilon} \mathrm{d}s
    &\leq 
    C \dtime^l \Gamma(l+1) \int_0^t \mu_1^{-(\varepsilon-\s)/2} \exp(-\mu_1 s) \mathrm{d}s  \\
    &\leq
    C \widetilde{C}_\varepsilon \dtime^l \Gamma(l+1) t\;,
\end{align*}
where the bound \eqref{eq:gstThetaGreaterR} is used
($\widetilde{C}_\varepsilon=\mu_1^{-(\varepsilon-\s+2)/2}$). This completes the proof of the assertion.
\end{proof}
\begin{remark}\label{rmk:s<2}
For $0\leq \s < 2$, the preceding result is valid under hypothesis 
\eqref{eq:TimeReg} with $\varepsilon = 0$, as used, e.g., in \cite{SS00_340}, 
but with $C(\s) \uparrow \infty$ as $\s \uparrow 2$.
\end{remark}
\begin{lemma} \label{lem:boundab}
Assume \eqref{eq:TimeReg} with some $\varepsilon \in (0,1)$ and some $\dtime \geq 1$. 
Let $u$ be the solution of \eqref{eq:IBVPVar}. For $T\geq b > a \geq 1,$ the estimate
\be \label{eq:L2abX2}
      \forall l \in \IN_0 : \quad \left(\int_a^b \| u^{(l)}(t) \|^2_{X_2} \mathrm dt\right)^{1/2}  \leq \dtime^l \Gamma(l+2) C(\varepsilon,a,b)
\ee
holds true with a constant $C(\varepsilon,a,b)>0$ independent of $l$ and $\dtime$. 
Furthermore, for $J=(0,T)$, we have $u \in H^1_{0,}(J;H)$ with
\be \label{eq:SobH1}
    \left(\int_0^T \| u(t) \|^2_H \mathrm dt\right)^{1/2}  
    \leq 
    C T^{3/2}, \; \left(\int_0^T \| u'(t) \|^2_H \mathrm dt\right)^{1/2} 
    \leq 
    C \delta \, (T^{1+\varepsilon}+T^3)^{1/2}
\ee
and $u \in L^2(J;X_2)$ with
\be \label{eq:L2X2} 
      \| u\|_{L^2(J;X_2)} \leq C T^{\varepsilon/2+1/2},
\ee
where the constant $C>0$ is independent of $\varepsilon$, $\dtime$ and $T.$
\end{lemma}
\begin{proof}
   The bound \eqref{eq:L2abX2} follows from \eqref{eq:dtlu}. For the estimates \eqref{eq:SobH1} and \eqref{eq:L2X2}, we use~\eqref{eq:dtlu} for $r=0$ with $l=0$ or $l=1$, and $r=2$ with $l=0$, respectively.
\end{proof}
\begin{proposition}\label{Prop:tExpBd}
  Assume \eqref{eq:TimeReg} with some $\varepsilon \in (0,1)$ and some $\dtime \geq 1$.
  For $r \in [0,2]$, there exists a constant $C>0$ (independent of $l$, $\dtime$, $a$, $b$, $t$, $q$) such that
  the solution $u$ of \eqref{eq:IBVPVar}
  satisfies
  \be\label{eq:dltuBd}
      \forall l \in \IN_0, \, \forall t \in (0, \min\{1,T\}] : \, \| u^{(l)}(t) \|_{X_\s} 
      \leq 
      C \dtime^l \Gamma(l+2) t^{-l+1-\s/2 + \varepsilon/2} \;,
  \ee
  and, for $0<a<b \leq \min\{1,T\}$,
  \be\label{eq:IntEst}
  \forall l \in \IN, l \geq 2 : \quad
  \left( \int_a^b \| u^{(l)}(t) \|^2_{X_\s} \mathrm dt \right)^{1/2} 
  \leq 
  C \dtime^{l} \Gamma(l+2) a^{-l+3/2-\s/2 + \varepsilon/2}\;,
  \ee
  and, for arbitrary $\rr\geq 2$,
  \be\label{eq:SobEst}
  \| u \|_{H^\rr((a,b);X_\s)} \leq C \dtime^{\rr} \Gamma(\rr+3) a^{-\rr+3/2-\s/2 + \varepsilon/2} \;.
  \ee
\end{proposition}
\begin{proof}
  The bound \eqref{eq:dltuBd} follows from \eqref{eq:dtlu}.
  Estimate \eqref{eq:IntEst} is obtained by integrating the pointwise bound~\eqref{eq:dltuBd}.
  The Sobolev bound \eqref{eq:SobEst} follows by interpolation.
\end{proof}
For the proof of the exponential convergence rate of the 
space-time discretization proposed in this work, 
we need the following regularity result,  which is proven in Appendix~\ref{sec:ProofLemH12X2Reg}.
\begin{lemma}\label{lem:H12X2Reg}
Assume \eqref{eq:TimeReg} with some $\varepsilon \in (0,1)$ and some $\dtime \geq 1$ for $l=0,1$.
Then, for $b \in (0,T]$, the solution $u$ of~\eqref{eq:IBVPVar} belongs to 
$H^{1/2}_{0,}((0,b);X_2)$ and the estimate
  \[
    \normiii{u}_{H^{1/2}_{0,}((0,b);X_2)} \leq \sqrt[4]{\frac{2}{\pi}}\, \frac{1}{\varepsilon}\, b^{\varepsilon/2} 
    \left( \frac{b}{1+\varepsilon}   + \frac{3}{\varepsilon} + 
            \frac{4 b^2}{(\varepsilon +1)(\varepsilon + 2)}  
    \right)^{1/2} C_g 
  \]
  holds true,
with
\begin{equation}\label{eq:boundg}
C_g:=\| g \|_{W^{1,\infty}((0,b);X_\varepsilon)}=\max\left\{\sup_{0 \leq t \leq b} \norm{g(t)}_{X_\varepsilon},
    \sup_{0 \leq t \leq b} \norm{g'(t)}_{X_\varepsilon}
    \right\}.
\end{equation}
\end{lemma}
\begin{remark}
  The assertion of Lemma~\ref{lem:H12X2Reg} can be generalized. For this purpose, define the interpolation space $H^{\theta}_{0,}(a,b) := (H^1_{0,}(a,b),L^2(a,b))_{\theta,2}$ for $\theta \in [1/2,1]$ with the usual Slobodetskii norm $\normiii{\circ}_{H^{\theta}(a,b)}$ as in \cite[p.~74]{McLean2000}, where $a<b$, $a,b \in \IR.$ Then, under the assumption of Lemma~\ref{lem:H12X2Reg}, we have
  \begin{equation*}
    \normiii{u}_{H^{\theta}_{0,}((0,b);X_2)} \leq C(b,\varepsilon,\theta, g)
  \end{equation*}
  for $\theta \in [1/2,1/2+\varepsilon) \cap [1/2,1]$ with a constant $C(b,\varepsilon,\theta, g)>0$.
\end{remark}

%%%%%%%%%%%%%%%%%%%%%%%%%%%%%%%%%%%%%%%%%%%%%%%%%%%%%%%%%%%%%%%%%%%%%%%%%%%%%%
\subsection{Spatial Regularity}
\label{sec:SpReg}
%%%%%%%%%%%%%%%%%%%%%%%%%%%%%%%%%%%%%%%%%%%%%%%%%%%%%%%%%%%%%%%%%%%%%%%%%%%%%%
We elaborate here on the regularity of the solution with respect to the 
spatial variable $x\in \domain$. For \eqref{eq:IBVP}, this regularity is,
of course, dependent on the temporal variable $t$, and the spaces $X_\s$ 
defined in \eqref{eq:Xs} via eigensystems, which are intrinsic to the spatial
operator \eqref{eq:xPbm} with \eqref{eq:Acoerc}, play a prominent role. 
In order to leverage spatial approximation results, 
we relate these spaces to 
standard ($d=1$) or corner-weighted ($d\ge 2$) Sobolev spaces.
As we shall consider in detail only $\IP^1$-Lagrangian FEM approximation 
in $\domain$, for the ensuing convergence rate analysis in Section
\ref{sec:Approx} we are mainly interested in the spaces 
$X_\s$ for $\s=0,1,2$ as defined in \eqref{eq:Xs}.
The cases $0\leq \s \leq 1$ coincide with standard Sobolev spaces
endowed with equivalent norms.
\begin{proposition}\label{prop:H0H1}
For space dimension $d\geq 2$, assume that
$\domain\subset \IR^d$ is a bounded Lipschitz domain.
Assume further that $A\in L^\infty(\domain;\IR^{d\times d}_{\mathrm{sym}})$ 
is uniformly positive definite %coercive
in the sense that \eqref{eq:Acoerc} is satisfied.
Then, 
$X_0 = L^2(\domain)$ and $X_1 \simeq 
H^1_{\Gamma_D}(\domain)$ 
and for $0<\s<1$, 
$X_\s \simeq (L^2(\domain), H^1_{\Gamma_D}(\domain))_{\s,2}$.
\end{proposition}
Consider next $1<\s\leq 2$. 
Once we characterize $X_2$, for $1<\s<2$, 
$X_\s$ is characterized by real interpolation.
To characterize $X_2$, 
we consider the source diffusion problem \eqref{eq:xPbm}, 
with assumption \eqref{eq:Acoerc} in place. 
In addition, we assume
\be\label{ass:ALip}
f\in L^2(\domain),\;\; A\in W^{1,\infty}(\domain;\IR^{d\times d}_{\mathrm{sym}}).
\ee
Then, 
eigenfunction expansions of $f\in L^2(\domain)$ 
imply that the unique solution $u\in X_1$ of \eqref{eq:xPbm}
belongs to $X_2$. 
Furthermore, the solution operator is bijective,
since from~\eqref{eq:Xs} and \eqref{ass:ALip} it follows that
\be\label{eq:X2L2}
\| u \|_{X_2}^2 
= \sum_{k=1}^\infty \mu_k^2 |u_k|^2 
= \| A(\partial_x) u \|_H^2
= \sum_{k=1}^\infty |f_k|^2 
= \| f \|^2_{L^2(\domain)}.
\ee
It remains to relate the space 
$X_2$, which is  defined in terms of the spatial operator $A(\partial_x)$,
to an intrinsic function space in $\domain$.
Due to \eqref{eq:X2L2}, $X_2 = (A(\partial_x))^{-1} L^2(\domain)$.
To characterize elements in $X_2$, 
we use the elliptic regularity of the BVP \eqref{eq:xPbm}
with time-independent data $f \in L^2(\domain)$ in
standard (if $d=1$) 
or 
corner-weighted (if $d\ge 2$)
Sobolev spaces in $\domain\subset \IR^d$. 

%%%%%%%%%%%%%%%%%%%%%%%%%%%%%%%%%%%%%%%%%%%%%%%%%%%%%%%%%%%%%%%%%%%%%%5
\subsubsection{{\bf Case $d=1$}}
%%%%%%%%%%%%%%%%%%%%%%%%%%%%%%%%%%%%%%%%%%%%%%%%%%%%%%%%%%%%%%%%%%%%%% 5
The spatial domain $\domain$ is an open, bounded and connected interval,
and, by \eqref{ass:ALip}, the diffusion coefficient is a scalar 
$a \in W^{1,\infty}(\domain)$ such that \eqref{eq:Acoerc} is satisfied.

Standard elliptic regularity results imply that 
there exists a constant $c>0$ such that, for every 
$f\in L^2(\domain)$, the solution $u=A(\partial_x)^{-1}f$ belongs to
$H^2(\domain)$ and satisfies 
$\| v \|_{H^2(\domain)} \leq c \| f \|_{L^2(\domain)}$.
This, combined with~\eqref{eq:X2L2}, 
gives that $X_2\subset H^2(\domain)$ and
\begin{equation}\label{eq:X2d=1}
\forall v\in X_2 : \quad 
\| v \|_{H^2(\domain)} \leq c \| v \|_{X_2}\;.
%= c \| f \|_{L^2(\domain)} \;.
\end{equation}
\begin{remark}\label{rem:1Dtransmission}
For $d=1$, a continuous embedding of $X_2$ into a nonintrinsic
function space can be easily established also for transmission problems.
Assume $\domain$ to be partitioned into $n_{\rm sub}$ disjoint, open and connected
subintervals $\calD = \{\domain_i\}_{i=1}^{n_{\rm sub}}$ 
and denote the corresponding broken Sobolev spaces
$W^{1,\infty}(\calD)=\{ a\in L^\infty(\domain):
a_{\mid_{\domain_i}} \in W^{1,\infty}(\domain_i), \; i=1,\dots, n_{\rm sub} \}$
and 
$H^2(\calD) := \{ v\in H^1(\domain): v_{\mid_{\domain_i}} \in
H^2(\domain_i), \; i=1,\dots, n_{\rm sub} \}$.
We set
$\|v\|_{H^2(\calD)}^2:=\|v\|_{H^1(\domain)}^2+\sum_{i=1}^{n_{\rm sub}} |v|_{H^2(\domain_i)}^2$.
We assume that the diffusion coefficient $a$ belongs to
$W^{1,\infty}(\calD)$ and satisfies
\eqref{eq:Acoerc}.
In this case,
standard elliptic regularity results imply that 
there exists a constant $c>0$ such that, for every $f\in
  L^2(\domain)$, $u=A(\partial_x)^{-1}f\in H^2(\calD)$ and $\| v \|_{H^2(\calD)} \leq c \| f
  \|_{L^2(\domain)}$.
This, combined with~\eqref{eq:X2L2}, gives $X_2\subset
H^2(\calD)$ and~\eqref{eq:X2d=1} is valid with $\| v \|_{H^2(\calD)}$ on the left side.
\end{remark}

%%%%%%%%%%%%%%%%%%%%%%%%%%%%%%%%%%%%%%%%%%%%%%%%%%%%%%%%%%%%%%%%%%%%%%5
\subsubsection{{\bf Case $d=2$}} \label{Sec:P1:d2}
%%%%%%%%%%%%%%%%%%%%%%%%%%%%%%%%%%%%%%%%%%%%%%%%%%%%%%%%%%%%%%%%%%%%%%5 
Under \eqref{ass:ALip}, for polygonal domains $\domain \subset \IR^2$,
weak solutions of the source problem \eqref{eq:xPbm}
are known to belong to a weighted Sobolev space of Kondrat'ev type 
which is defined as follows.
\begin{definition}[Kondrat'ev Spaces in dimension $d=2$] \label{def:Kma}
Assume that $\domain\subset \IR^2$ is a bounded polygonal domain
with $\geq 3$ corners and straight sides, 
whose boundary $\partial \domain$ is Lipschitz.

Denote by
$r_\domain:\domain\to \IR_{\ge 0}$
a smooth function that locally,
in a (sufficiently small) open neighborhood
of each corner of $\domain$, 
coincides with the Euclidean distance to that corner. 
Then, for $m\in \IN_0$ and for some constant $a>0$, 
the Kondrat'ev corner-weighted Sobolev space $\cK^m_a(\domain)$ 
is defined as
\be\label{eq:DefKma}
\cK^m_a(\domain) 
:= 
\left\{ v \colon \, \domain\to\IR : \; \forall |\alpha|\leq m : \, 
   r_\domain^{|\alpha|-a} \partial^\alpha v \in L^2(\domain) \right\}
   \;,
\ee
with 
$\| u \|_{\cK^m_{a}(\domain)}^2
  :=
   \sum_{|\alpha|\le m}\| r_\domain^{|\alpha|-a} \partial^\alpha v\|_{L^2(\domain)}^2$.
\end{definition}
The regularity result in question is a special case of 
\cite[Thm. 4.4]{BLN2017}, which we state here for definiteness in the form
required by us.
\begin{proposition}\label{prop:BLN2017}
Assume that $\domain \subset \IR^2$ is a bounded polygon 
with boundary $\partial\domain$ consisting of a finite number of
straight sides. 
Consider the elliptic source problem \eqref{eq:xPbm}
with assumptions \eqref{eq:Acoerc} and \eqref{ass:ALip} in place.

Then, there exist $c>0$ and a constant $a>0$ such that, 
for every $f\in L^2(\domain)$, 
the weak solution $u \in X_1 = H^1_{\Gamma_D}(\domain)$ 
of \eqref{eq:xPbm} 
belongs to $\cK^2_{a+1}(\domain)$ and 
satisfies the a~priori estimate
\be\label{eq:X2apriori}
\| u \|_{\cK^2_{a+1}(\domain)} \leq c \| f \|_{L^2(\domain)}. 
\ee
In particular, therefore,  $X_2 \subset \cK^2_{a+1}(\domain)$ 
and there exists $c>0$ such that
\be\label{eq:X2K2a}
\forall v \in X_2: \quad \| v \|_{\cK^2_{a+1}(\domain)} \leq c \| v \|_{X_2} 
\;.
\ee
\end{proposition}
\begin{proof}
  Assumption \eqref{ass:ALip} implies that $A \in \cW^{1,\infty}(\domain)$ 
as defined in \cite[Eqn. (5)]{BLN2017}, 
and that 
$\| A \|_{\cW^{1,\infty}(\domain)} \leq C(\domain) \| A \|_{W^{1,\infty}(\domain)}$.
We may then use \cite[Thm. 4.4]{BLN2017} with $b_i = c =0$, $m=1$, to conclude
the \emph{a~priori} estimate
\[
\| u \|_{\cK^2_{a+1}(\domain)} \leq c \| f \|_{\cK^0_{a-1}(\domain)}
\]
for all $|a|<\eta$ for some (sufficiently small) $\eta > 0$. 
We assume, without loss of generality, that $0 < \eta < 1$. 
Then, definition \eqref{eq:DefKma} states that
$f\in \cK^0_{a-1}(\domain)$  means $r_\domain^{-(a-1)}f\in L^2(\domain)$.
As $-(a-1) > 0$, $r_\domain^{-(a-1)}\in L^\infty(\domain)$, so that 
$\| f \|_{\cK^0_{a-1}(\domain)} \leq c(a,\domain) \| f \|_{L^2(\domain)}$.
The \emph{a~priori} estimate implies then \eqref{eq:X2apriori}.
Since $\| f \|_{L^2(\domain)} =\| u \|_{X_2}$
(see \eqref{eq:X2L2}),
the \emph{a~priori} estimate also implies \eqref{eq:X2K2a}.
\end{proof}
\begin{remark}\label{rmk:Transm2d}
For transmission problems in a polygonal domain $\domain$,
with \emph{piecewise constant, isotropic coefficients}
in materials occupying a finite number $n_{\rm sub}$ of 
polygonal subdomains $\domain_i\subset \domain$,
regularity in the weighted spaces $\cK^2_{a+1}(\domain)$ with
radial weights also at multi-material intersection points in 
$\domain$ are stated in \cite[Theorem~3.7]{FEMTransmission}.
The assumptions in~\cite{FEMTransmission} on $A$ 
are more restrictive than just~\eqref{eq:Acoerc} and 
$A\in W^{1,\infty}(\calD;\IR^{d\times d}_{\mathrm{sym}})$.
The regularity result in \cite[Theorem~3.7]{FEMTransmission}
with $m=1$ will imply for $u\in X_2$ a splitting $u=u_{\mathrm{reg}}+w_s$,
with the bound \eqref{eq:X2apriori} for $u_{\mathrm{reg}}|_{\domain_i}$ 
on each subdomain $\domain_i$, 
and with $w_s$ in a finite-dimensional space $W_s$,
  see~\cite[Sect.~3.2]{FEMTransmission}.
\end{remark}

%%%%%%%%%%%%%%%%%%%%%%%%%%%%%%%%%%%%%%%%%%%%%%%%%%%%%%%%%%%%%%%%%%%%%%5
\subsubsection{{\bf Case $d=3$}}
\label{sec:Regd=3}
%%%%%%%%%%%%%%%%%%%%%%%%%%%%%%%%%%%%%%%%%%%%%%%%%%%%%%%%%%%%%%%%%%%%%%5
Proposition \ref{prop:BLN2017} remains valid in space dimension $d=3$. 
To detail a precise statement, we still assume \eqref{ass:ALip}.
Then, \cite[Theorem 1.1]{AmannNistor08} implies \eqref{eq:X2apriori} 
and \eqref{eq:X2K2a} in bounded, polyhedral domains $\domain\subset \IR^3$
with Lipschitz boundary $\partial\domain$ consisting of a finite number
of plane faces. Similar results are shown in \cite{MR3d} and,
for the Poisson equation with $\Gamma = \Gamma_D$, in \cite[Theorem 1.2]{BNZ2005}
(with $\mu=1$ in the statement of that theorem).

%%%%%%%%%%%%%%%%%%%%%%%%%%%%%%%%%%%%%%%%%%%%%%%%%%%%%%%%%%%%%%%%%%%%%%%%%%%%%%
\section{Approximation}
\label{sec:Approx}
%%%%%%%%%%%%%%%%%%%%%%%%%%%%%%%%%%%%%%%%%%%%%%%%%%%%%%%%%%%%%%%%%%%%%%%%%%%%%%
We introduce the spatial and temporal (quasi-) interpolation
operators that shall allow us to deduce convergence rates of the space-time
variational approximation of formulation \eqref{eq:IBVPVar}. 
In order to use the tensor product construction of subspaces in~\eqref{eq:xtApprSpc},
we specify the choice of temporal subspaces $V^M_t\subset H^{1/2}_{0,}(J)$ for the temporal domain $J=(0,T)$.
In the spatial domain $\domain$, 
$V^N_x \subset H^1_{\Gamma_D}(\domain)$ will be specified in Section \ref{Sec:xAppr}
below.

%%%%%%%%%%%%%%%%%%%%%%%%%%%%%%%%%%%%%%%%%%%%%%%%%%%%%%%%%%%%%%%%%%%%%%%%%%%%%%
\subsection{$hp$-Approximation in $\overline{J}=[0,T]$}
\label{sec:hpApprJ}
%%%%%%%%%%%%%%%%%%%%%%%%%%%%%%%%%%%%%%%%%%%%%%%%%%%%%%%%%%%%%%%%%%%%%%%%%%%%%%
To specify the $hp$-subspace $V^M_t\subset H^{1/2}_{0,}(J)$ in~\eqref{eq:xtApprSpc}, 
we fix the geometric subdivision
parameter $\sigma \in (0,1)$ and the number of elements $m:= m_1 + m_2 \in \IN$ 
with given $2<m_1 \in \IN$, $m_2 \in \IN_0$. 
We set $T_1 := \min \{1,T\}$. 
Then, we define the time steps by
\be \label{def:tj}
  t_j := \begin{cases}
            0,                                      & j=0,  \\
            T_1 \sigma^{m_1-j},                     & j \in \{ 1,\dots, m_1 \}, \\
            \frac{T-T_1}{m_2} \cdot (j-m_1) + T_1,  & j \in \{ m_1+1, \dots, m_1+m_2 \}, \quad \text{ if } m_2 > 0,
        \end{cases}
\ee
where the last line is omitted in the case $T_1=T$, i.e., we assume $m_2=0$ whenever $T_1=T$. 
Furthermore, 
we denote by $I_j = (t_{j-1},t_j)\subset J$ the corresponding time intervals of lengths 
$k_j := |I_j| = t_j - t_{j-1}$, fulfilling
\be \label{def:kj}
  k_j = \begin{cases}
          T_1\sigma^{m_1-1},            & j=1,  \\
          T_1\sigma^{m_1-j}(1-\sigma),  & j \in \{ 2,\dots, m_1 \}, \\
          k_T:= \frac{T-T_1}{m_2},      & j \in \{ m_1+1, \dots, m_1+m_2 \},  \quad \text{ if } m_2 > 0.
        \end{cases}
\ee
Note that the splitting of $\overline{J}=[0,T]$ into the parts $[0,T_1]$ and $[T_1,T]$ 
is necessary for the proofs of the $hp$-error estimate in Section~\ref{sec:ConvRate}, 
since Proposition~\ref{Prop:tExpBd} states estimates for $b \leq T_1 = \min\{1,T\}$ only. 
In other words, we apply the temporal $hp$-FEM in $[0,T_1]$, 
whereas in $[T_1,T]$ we use a temporal $p$-FEM in the case $T>1$. 
With this notation, 
we define a geometric partition $\cG^m_\sigma = \{I_j\}_{j=1}^m$ of $J=(0,T).$ On $\cG^m_\sigma$, 
we introduce the distribution $\bmp = (p_1,\dots,p_m) \in \IN^m$
of polynomial degrees
as follows: For a given slope parameter $\muhp \in \IR$, $\muhp \geq 1$, we set
\be \label{def:pj}
  p_j := \begin{cases}
          1, & j=1, \\
          \lfloor \muhp j \rfloor,    & j \in \{ 2,\dots,m_1 \}, \\
          p_T:= \lfloor \muhp m_1 \rfloor,  & j \in \{ m_1+1, \dots, m_1+m_2 \}, \quad \text{ if } m_2 > 0,
        \end{cases}
\ee
where $\lfloor \circ \rfloor$ denotes the floor function. 
Again, in the case $m_2=0$, the last line is omitted.
Thus, we set
$S^{\bmp,1}(J;\cG^m_\sigma) := \{ v \in C^0(\overline{J}): v_{\mid_{I_j}} \in \IP^{p_j} \},$
and the temporal subspace $V^M_t$ in \eqref{eq:xtApprSpc}
is defined as 
\be\label{eq:VMt} 
S^{\bmp,1}_{0,}(J;\cG^m_\sigma) 
:= 
\{ v \in S^{\bmp,1}(J;\cG^m_\sigma): v(0) = 0 \} 
\subset 
H^{1/2}_{0,}(J)
\;.
\ee
Due to the continuity requirement at $t_j$ for $j=1,\dots,m-1$, which 
is mandated by the $H^{1/2}$-conformity, and the zero trace at $t=0$,
it holds that
\[
  M = {\rm dim}(S^{\bmp,1}_{0,}(J;\cG^m_\sigma)) = \sum_{j=1}^m p_j.
\]
We introduce the temporal quasi-interpolant 
$\Pi^{\bmp,1}_{\cG^m_\sigma} v$ for a sufficiently smooth function $v \colon \, [0,T] \to \IR$ by
\be \label{def:projection}
  \left( \Pi^{\bmp,1}_{\cG^m_\sigma}v\right)(t) :=
  \begin{cases}
    v(t_1) t/t_1,                                                                   & t \in \overline{I_1} \\
    v(t_{j-1}) + \int^t_{t_{j-1}} (\Pi^{p_j-1}_{L^2(I_j)}v' ) (\xi)\mathrm d \xi,   & t \in \overline{I_j}, \;  j \in \{2,\dots,m\},
  \end{cases}
\ee
where $\Pi^{p_j-1}_{L^2(I_j)}$ 
denotes the $L^2(I_j)$ projection onto $\IP^{p_j-1}$. As~\eqref{def:projection} uses point values of the interpolated
function, $\Pi^{\bmp,1}_{\cG^m_\sigma}$ is only defined on a subspace of the continuous functions $C^0(\overline{J})$. Note that the nodal property
\be \label{projNodal}
  \forall j \in \{0, \dots, m \}: \quad \left( \Pi^{\bmp,1}_{\cG^m_\sigma}v\right)(t_j) = v(t_j)
\ee
holds true for a sufficiently smooth function $v$ with $v(0)=0$.
Our approach to convergence rate bounds in the fractional Sobolev norms is to 
first obtain estimates in the additive integer order $L^2$ and $H^1$ norms
in the usual fashion by scaling estimates on unit size reference domains,
then to interpolate the global $L^2$ and $H^1$ norm error bounds.
For $j\geq 2$, the error bounds in $I_j$ are standard $hp$-interpolation error
estimates as can be found, e.g., in \cite[Chapter~3]{Schwab98}.
We recall the error bound on $\hat{I} = (-1,1)$, with the estimates on $I_j$
following by scaling.
\begin{lemma}\label{lem:hp}
On $\hat{I} = (-1,1)$,
for every $p\in \IN$, a projector $\hat{\Pi}^p_{1} \colon \, H^1(\hat{I}) \to \IP^p(\hat{I})$ exists such that,
for all $v\in H^{r+1}(\hat{I})$ with some $r\in \IN$, 
\be\label{eq:hphIH1}
\| v' - (\hat{\Pi}^p_1 v)' \|_{L^2(\hat{I})}^2
\leq 
\frac{(p-s)!}{(p+s)!} \|v^{(s+1)}\|^2_{L^2(\hat{I})} 
\ee
and
\be\label{eq:hphIL2}
\| v - \hat{\Pi}^p_1 v \|_{L^2(\hat{I})}^2
\leq
\frac{1}{p(p+1)} \frac{(p-s)!}{(p+s)!} \|v^{(s+1)}\|^2_{L^2(\hat{I})} 
\ee
are valid for every integer $s$ with $0\leq s \leq \min\{r,p\}$.
Furthermore, 
\[
  \left(\hat{\Pi}^p_1 v\right)(\pm 1) = v(\pm 1) \;.
\]
\end{lemma}
We remark that the projectors $\hat{\Pi}^p_1$ for $p\geq 1$ are
given by 
\[
  \left(\hat{\Pi}^p_1 v \right) (t) := v(-1) + \int_{-1}^t \hat{\Pi}^{p-1}_0 (v') (\xi)\mathrm d \xi \;,
\quad t\in \hat{I}\;,
\]
with $\hat{\Pi}^{p-1}_0$ denoting the $L^2(\hat{I})$ projection onto 
$\IP^{p-1}$.

For $I_j \in \cG^m_\sigma$ with $j\geq 2$,
the global quasi-interpolation projectors $\Pi^{\bmp,1}_{\cG^m_\sigma} $ 
are obtained by transporting $\hat{\Pi}^{p_j}_1$ from $\hat{I}$ to $I_j\in \cG^m_\sigma$ 
via affine transformations $T_j:\hat{I}\to I_j$, resulting in local projections
$\Pi^{p_j}_{1,j}$. 

We scale the projection error bounds \eqref{eq:hphIH1} and \eqref{eq:hphIL2}
to $I_j$, and apply them to strongly measurable maps $v \colon \, I_j\to X$ for 
separable Hilbert space $X$ by Hilbertian tensorization 
of Bochner spaces.
We denote by $\IP^p(I_j;X)$ the linear space of 
polynomial maps of degree $p$ with coefficients in $X$. 
We obtain the following result.
\begin{lemma}\label{lem:hpIj}
For every $I_j\in \cG^m_\sigma$ with $j\geq 2$ with 
time-step size $k_j = |I_j|$, and for every $p\in \IN$, there
exists a projector 
$\Pi^{p}_{1,j}: H^1(I_j;X) \to \IP^{p}(I_j;X)$ 
such that,
for every $v\in H^{r+1}(I_j;X)$ with some $r\in \IN$,
the error bounds
\[
  \| \partial_t v - \partial_t \Pi^{p}_{1,j} v \|_{L^2(I_j;X)}^2
  \leq
  C
  \frac{(p-s)!}{(p+s)!} 
  \left(\frac{k_j}{2}\right)^{2s}
  \|\partial^{s+1}_t v \|^2_{L^2(I_j;X)} 
\]
and
\[
  \| v - \Pi^{p}_{1,j} v \|_{L^2(I_j;X)}^2
  \leq
  C
  \frac{1}{p(p+1)} \frac{(p-s)!}{(p+s)!}
  \left(\frac{k_j}{2}\right)^{2(s+1)}
  \|\partial^{s+1}_t v \|^2_{L^2(I_j;X)} 
\]
are valid for every integer $s$ with 
$0\leq s \leq \min\{r,p\}$.
Furthermore,
\[
  \left(\Pi^{p}_{1,j} v\right)(t) = v(t) \;\;\mbox{in}\;\; X 
\;\;\mbox{for} \;\; t \in \partial I_j = \{t_{j-1},t_j\}
\;.
\]
\end{lemma}

%%%%%%%%%%%%%%%%%%%%%%%%%%%%%%%%%%%%%%%%%%%%%%%%%%%%%%%%%%%%%%%%%%%%%%%%%%%%%%
\subsection{$\IP^1$-FEM Approximation in $\domain$}
\label{Sec:xAppr}
%%%%%%%%%%%%%%%%%%%%%%%%%%%%%%%%%%%%%%%%%%%%%%%%%%%%%%%%%%%%%%%%%%%%%%%%%%%%%%
We consider the choice of subspaces $V^N_x\subset H^1_{\Gamma_D}(\domain)$
in \eqref{eq:xtApprSpc} as standard, conforming $\IP^1$-Lagrangian finite elements
on simplicial meshes $\cT$ of $\domain$. We denote by
$S^1(\domain;\cT)$ the space of continuous, piecewise linear functions on $\cT$,
and further, we define the closed subspace
\be \label{Approx:SGammaD}
    S^1_{\Gamma_D}(\domain;\cT) := S^1(\domain;\cT) \cap H^1_{\Gamma_D}(\domain) \subset H^1_{\Gamma_D}(\domain).
\ee

%%%%%%%%%%%%%%%%%%%%%%%%%%%%%%%%%%%%%%%%%%%%%%%%%%%%%%%%%%%%%%%%%%%%%%5
\subsubsection{{\bf Case $d=1$}}
%%%%%%%%%%%%%%%%%%%%%%%%%%%%%%%%%%%%%%%%%%%%%%%%%%%%%%%%%%%%%%%%%%%%%%5
For any finite partition $\cT$ of the open, bounded and connected
interval $\domain$ into $N$ open subintervals
that is quasi-uniform with mesh width 
$h := \max\{ |I_j|: I_j\in \cT \} >0$,
there exists a constant $c>0$
independent of $N=O(h^{-1})$
such that
the nodal interpolant 
$I^N: C^0(\overline{\domain}) \to S^1(\domain;\cT)$ 
satisfies
\begin{equation} \label{eq:X2d=1I}
\forall v\in X_2: \quad 
\| v - I^N v \|_{L^2(\domain)} 
+ 
N^{-1} \|  v - I^N v \|_{H^1(\domain)}
\leq 
c N^{-2} \| v \|_{H^2(\domain)} 
\;.
\end{equation}
With~\eqref{eq:X2d=1}, 
for any $f\in L^2(\domain)$, we also have that the solution $u = A(\partial_x)^{-1}f$ 
satisfies
\begin{equation*}
\| u - I^N u \|_{L^2(\domain)}
+
N^{-1}
\|  u - I^N u \|_{H^1(\domain)}
\leq
c N^{-2} \| f \|_{L^2(\domain)}
\;.
\end{equation*}

\begin{remark}\label{rem:1Dtransmission_approx}
  For transmission problems with diffusion coefficient $a\in
  W^{1,\infty}(\calD)$ as in Remark~\ref{rem:1Dtransmission},
  assuming that $\cT$
  is compatible with the partition $\calD$
  (i.e., the set of nodes of $\cT$
  includes all interfaces in $\calD$),
  the nodal interpolant $I^N: C^0(\overline{\domain}) \to
  S^1(\domain;\cT)$ satisfies~\eqref{eq:X2d=1I} with $\| v
  \|_{H^2(\calD)}$ instead of $\| v \|_{H^2(\domain)}$
  on the right side. The subsequent estimate for
  $u = A(\partial_x)^{-1}f$, $f\in L^2(\domain)$,
  follows from~\eqref{eq:X2d=1} with $\| v
  \|_{H^2(\calD)}$ on the left side (see Remark~\ref{rem:1Dtransmission}).
\end{remark}
  
%%%%%%%%%%%%%%%%%%%%%%%%%%%%%%%%%%%%%%%%%%%%%%%%%%%%%%%%%%%%%%%%%%%%%%5
\subsubsection{{\bf Case $d=2$}}
%%%%%%%%%%%%%%%%%%%%%%%%%%%%%%%%%%%%%%%%%%%%%%%%%%%%%%%%%%%%%%%%%%%%%%5
$\domain\subset\IR^2$ is a polygon with a finite number of corners and straight sides.
We assume furthermore that each entire side $\Gamma_j$ has either the
Dirichlet or the Neumann boundary condition (this is possible by 
subdividing sides of $\domain$ with changing boundary conditions and by increasing $M$ 
appropriately; points where boundary conditions change become then 
``corner points''). 

As it is well-known (e.g.,~\cite{BKP79,Apel1999,NstrGrd2015} and the references there),
functions $u\in \cK^2_{a+1}(\domain)$ allow for rate-optimal approximation
in $H^1(\domain)$ and $L^2(\domain)$ norms in terms of continuous, piecewise linear nodal 
Lagrangian FEM in $\domain$, on regular, simplicial partitions $\cT^N_\beta$
(see, e.g.,~\cite{BKP79,Apel1999,NstrGrd2015} and the references there for constructions) %\todo{[More details?]}
of $\domain$ with $O(N)$ triangles and algebraic corner-refinement
  towards the vertices of $\domain$. The subscript $\beta\in(0,1]$ denotes
  the corner-refinement parameter, with $\beta=1$
  corresponding to quasi-uniform meshes.
As $\cK^2_{a+1}(\domain) \subset C(\overline{\domain})$ (see, e.g.,~\cite{ BKP79}),
the nodal interpolation operator $I^N_\beta$ is well-defined for $u\in \cK^2_{a+1}(\domain)$.
Also, for $u \in \cK^2_{a+1}(\domain)\cap H^1_{\Gamma_D}(\domain)$,
the interpolants $I^N_\beta u$ 
satisfy exactly the homogeneous Dirichlet boundary conditions on~$\Gamma_D$.
Furthermore, for suitably strong mesh grading as expressed by the
parameter $\beta$ (depending on $\domain$, and the corner angles
at the vertices of $\domain$),
the interpolants $I^N_\beta u$ of $u\in \cK^2_{a+1}(\domain)$
converge at optimal rates under mesh refinement:
there exists a constant $c>0$ such that, for all 
$N = {\rm dim}(S^1_{\Gamma_D}(\domain;\cT^N_\beta)) \in \IN$,
\be\label{eq:hFEMP1Bd}
\| u - I^N_\beta u \|_{L^2(\domain)}  
+ 
N^{-\frac 12} \| u - I^N_\beta u \|_{H^1(\domain)} 
\leq 
c N^{-1} 
\| u \|_{\cK^2_{a+1}(\domain)} 
\leq 
cN^{-1} \| f \|_{L^2(\domain)}
\;.
\ee
Here, we used \eqref{eq:X2apriori} in the last step.
\begin{remark}\label{rmk:BisTree}
The interpolation error bound \eqref{eq:hFEMP1Bd} is based on 
the graded mesh family $\{ \cT^N_\beta\}_{N\geq 1}$.
The bound  \eqref{eq:hFEMP1Bd} also holds on families 
of bisection tree meshes, as shown in 
\cite[Theorems~5.1,~2.1]{GaspMorin}.
Such families are typically generated by adaptive algorithms,
and will also be used in the ensuing numerical experiments
in Section~\ref{sec:NumExp} below.
\end{remark}
\begin{remark}\label{rmk:Transm}
For transmission problems in $\domain$, 
with $A$ as in \eqref{eq:Acoerc}, piecewise smooth 
on a finite partition $\{ \domain_i \}_{i=1}^{n_{\mathrm{sub}}}$ 
of $\domain$ in straight-sided polygons $\domain_i$, 
the results in \cite[Theorem~3.7]{FEMTransmission}
imply that with graded meshes in each $\domain_i$ 
with grading towards multimaterial intersection points,
the interpolation error bound \eqref{eq:hFEMP1Bd} is based on
the graded mesh family $\{ \cT^N_\beta\}_{N\geq 1}$
still remains true by approximating $u_{\mathrm{reg}}$ and $w_s$ 
in the decomposition of \cite[Theorem~3.7]{FEMTransmission}
separately.
\end{remark}

%%%%%%%%%%%%%%%%%%%%%%%%%%%%%%%%%%%%%%%%%%%%%%%%%%%%%%%%%%%%%%%%%%%%%%5
\subsubsection{{\bf Case $d=3$}}
\label{sec:Approxd=3}
%%%%%%%%%%%%%%%%%%%%%%%%%%%%%%%%%%%%%%%%%%%%%%%%%%%%%%%%%%%%%%%%%%%%%%5
Only partial extensions of \eqref{eq:hFEMP1Bd} 
to space dimension $d=3$
are available. We indicate the argument in one particular case. 
Specifically,
we assume~\eqref{eq:Acoerc}, \eqref{ass:ALip} and, in addition,
that $A(x) = a(x){\mathbb{I}}$, with $a\in W^{1,\infty}(\domain)$.
Furthermore, we assume that $\Gamma_D = \Gamma$, i.e., we consider homogeneous
Dirichlet boundary conditions on the entire~$\Gamma$.
The temporal (analytic) regularity in Section~\ref{sec:tReg} is then
still valid and, as outlined in Section~\ref{sec:Regd=3},
the space $X_2$ is continuously embedded into a 
weighted Kondrat'ev space in $\domain$ with corner- and edge-weights.
A convergence estimate analogous to the $H^1$ bound in \eqref{eq:hFEMP1Bd}  
(with rate $N^{-1/3}$ instead of $N^{-1/2}$)
is stated in~\cite[Theorem 2.1]{BNZ2005} with $m=1$, 
and proven in~\cite{BNZ2007}, for standard, first-order 
Langrangian FEM in $\domain$ on regular triangulations of $\domain$
into simplices, with \emph{anisotropic edge refinements}.

%%%%%%%%%%%%%%%%%%%%%%%%%%%%%%%%%%%%%%%%%%%%%%%%%%%%%%%%%%%%%%%%%%%%%%%%%%%%%%
\section{Convergence Rate of the Space-Time Discretization}
\label{sec:ConvRate}
%%%%%%%%%%%%%%%%%%%%%%%%%%%%%%%%%%%%%%%%%%%%%%%%%%%%%%%%%%%%%%%%%%%%%%%%%%%%%%
We are in a position to establish the convergence rate of the space-time
Galerkin discretization \eqref{eq:IBVPVarxt} with 
$V^M_t = S^{\bmp,1}_{0,}(J;\cG^m_\sigma)$ as defined in \eqref{eq:VMt}
and with
$V^N_x = S^1_{\Gamma_D}(\domain;\cT^N_\beta)$ 
as given in \eqref{Approx:SGammaD}, where $\beta=1$ in the case $d=1$.

We will require the temporal $H^{1/2}_{0,}(J)$ projector $Q^{1/2}_t$
onto $V^M_t$ and the spatial $H^1_{\Gamma_D}(\domain)$ 
``Ritz'' projector $Q^1_x$ into $V^N_x$. Being orthogonal projections,
they are stable, i.e., 
$\| Q^{1/2}_t v \|_{H^{1/2}_{0,}(J)} \leq \| v \|_{H^{1/2}_{0,}(J)}$,
$\| Q^1_x v \|_{X_1} \leq \| v \|_{X_1}$, 
and optimal in the respective spaces, i.e.,
\[
\| v - Q^{1/2}_tv \|_{H^{1/2}_{0,}(J)}
=
\min_{w\in V^M_t} \| v - w \|_{H^{1/2}_{0,}(J)}
\;
\mbox{and}
\;
\| v - Q^1_x v \|_{X_1} 
=
\min_{w\in V^N_x} \| v - w \|_{X_1}
\;.
\]
Here, we recall that $X_1 = H^1_{\Gamma_D}(\domain)$ 
denotes the ``energy'' space with 
norm given by $\| v \|_{X_1} := a(v,v)^{1/2}$.
Hence, we may write (for sufficiently regular arguments $v$)
\be\label{eq:xQopt}
\| v - Q^1_x v \|_{X_1}  \leq c \| v - I^N_\beta v \|_{H^1(\domain)} 
\ee
with a constant $c>0$ depending on $\domain$ and on the coefficient $A$.
Assuming a
sufficiently strong corner-mesh refinement in $\domain$ in the
case $d=2$, 
an Aubin--Nitsche duality argument, 
together with~\eqref{eq:X2d=1I} and~\eqref{eq:X2d=1} if $d=1$, 
or~\eqref{eq:hFEMP1Bd} and~\eqref{eq:X2K2a} if $d=2$,
implies that there exists a constant $c>0$ such that,
for all 
$N = {\rm dim}(S^{1}_{\Gamma_D}(\domain;\cT^N_\beta))$
and all $w\in X_2$ (see, e.g., \cite[Thm. 5.2]{BKP79}),
\be\label{eq:hFEML2Bd}
\| w - Q^1_x w\|_{L^2(\domain)}
\leq 
c N^{-2/d} \| w \|_{X_2} \;.
\ee
The optimality of the temporal projection $Q^{1/2}_t$ 
in $H^{1/2}_{0,}(J)$ also implies
\be\label{eq:tQopt}
\| v - Q^{1/2}_tv \|_{H^{1/2}_{0,}(J)}
\leq 
\| v - \Pi^{\bmp,1}_{\cG^m_\sigma}v\|_{H^{1/2}_{0,}(J)}
\ee
for a sufficiently regular $v \colon \, J \to \IR$.
Here, $\Pi^{\bmp,1}_{\cG^m_\sigma}$ 
is the temporal quasi-interpolant of Subsection~\ref{sec:hpApprJ}.
Proceeding as in the proof of \cite[Theorem 3.4]{OStMZ}, 
we obtain the following estimate (see \cite[p.~175 bottom]{OStMZ}). % and the last step of that proof) 
\begin{lemma}\label{lem:xtQopt}
Let $u$ and $u^{MN}$ be the solutions to~\eqref{eq:IBVPVar}
and~\eqref{eq:IBVPVarxt}, respectively. We have
\be \label{eq:xtQopt1}
\begin{split}
&\| u - u^{MN} \|_{H^{1/2}_{0,}(J;L^2(\domain))} \leq  \| u - Q^{1/2}_t u\|_{H^{1/2}_{0,}(J;L^2(\domain))}
  \\ 
&\qquad\qquad+ 
  \| u-Q^1_xu \|_{H^{1/2}_{0,}(J;L^2(\domain))}
+ 
\left\| (I-Q^{1/2}_t)(I-Q^1_x) u
\right\|_{H^{1/2}_{0,}(J;L^2(\domain))}
 \\
&\qquad\qquad+
  \| u - Q^1_x u \|_{H^{1/2}_{0,}(J;L^2(\domain))}
+
\| A(\partial_x)(u-Q^{1/2}_tu)\|_{[H^{1/2}_{,0}(J;L^2(\domain))]'}
\;.
\end{split}
\ee
\end{lemma}
We combine \eqref{eq:xQopt}--\eqref{eq:xtQopt1} with the preceding regularity, proven in Section~\ref{sec:Reg}, and the approximation properties of the projections $Q^{1/2}_t$, $Q^1_x$ to 
obtain our main convergence rate bound. For this purpose, we address 
{\bf Term$1$} through {\bf Term$5$} in the upper bound \eqref{eq:xtQopt1}.
To this end, we use that the solution $u$ to~\eqref{eq:IBVPVar} belongs to $H^{1/2}_{0,}(J;X_2)$, 
which was proven in Lemma~\ref{lem:H12X2Reg}.

We start by deriving upper bounds for {\bf Term$1$} and {\bf Term$5$}. 
We have
    $L^2(Q)\simeq
    [L^2(Q)]'\hookrightarrow [H^{1/2}_{,0}(J;L^2(\domain))]'$ and
    $H^{1/2}_{0,}(J;X_2) \hookrightarrow L^2(J;X_2)$
    with continuous and dense injections.
    This, together with~\eqref{eq:X2L2}, gives the following bound for {\bf Term$5$}:
    \[
      \begin{split}
      \|A(\partial_x)(u-Q^{1/2}_tu)\|_{[H^{1/2}_{,0}(J;L^2(\domain))]'}
      &\le\tilde c (T) \|A(\partial_x)(u-Q^{1/2}_tu)\|_{L^2(Q)}
      \\ &
      = \tilde c (T) \|u-Q^{1/2}_tu\|_{L^2(J;X_2)}
      \\ &
      \le c(T) \|u-Q^{1/2}_t u\|_{H^{1/2}_{0,}(J;X_2)}\;.
      \end{split}
    \]
Using estimate~\eqref{ineq:XrEmbedding} yields that {\bf Term$1$} can be bounded by
\[
  \| u - Q^{1/2}_t u\|_{H^{1/2}_{0,}(J;L^2(\domain))} \leq c \|u-Q^{1/2}_t u\|_{H^{1/2}_{0,}(J;X_2)}
\]
with a constant $c>0$, i.e., for both {\bf Term$1$} and {\bf Term$5$}, we need an estimate of the term $\|u-Q^{1/2}_t u\|_{H^{1/2}_{0,}(J;X_2)}$. For this purpose, we use the temporal quasi-interpolant $\Pi^{\bmp,1}_{\cG^m_\sigma}$ of Subsection~\ref{sec:hpApprJ} and the inequality \eqref{eq:tQopt}. First, note that $\Pi^{\bmp,1}_{\cG^m_\sigma} u$ is well-defined since $u \colon \, [0,T] \to X_2$ is continuous, see estimate \eqref{eq:dltuBd} for $l=0$, $r=2$, and since $u \colon \, [0,T] \to X_2$ is smooth for $t>0$ due to Lemma~\ref{lem:DtBound}. Second, we have $u \in H^{1/2}_{0,}(J;X_2)$ because of Lemma~\ref{lem:H12X2Reg}, hence $u-\Pi^{\bmp,1}_{\cG^m_\sigma} u \in H^{1/2}_{0,}(J;X_2)$. Thus, it remains to estimate $\|u-\Pi^{\bmp,1}_{\cG^m_\sigma} u\|_{H^{1/2}_{0,}(J;X_2)}$, which is done in the following lemmas.

\begin{lemma} \label{lem:BoundGamma}
    Let $\alpha>0$ and $m \in \IN_0$ be given. For $\mu \geq 1$ with $\mu > \alpha$, there exist a constant $C_\Gamma > 0$, depending on $\alpha, \mu$, but independent of $m$ such that
    \begin{equation*}
        \sum_{j=0}^m \alpha^{2j} \frac{\Gamma(\lfloor \mu j \rfloor-j+1)}{\Gamma(\lfloor \mu j \rfloor +j+1)} \Gamma(j+3)^2 \leq C_\Gamma.
    \end{equation*}
\end{lemma}
\begin{proof}
    The proof is based on \cite[Lemma~3.4]{DD20}, see Appendix~\ref{sec:ProofLemBoundGamma}.
\end{proof}

\begin{lemma} \label{lem:hperrorsmoothX2}
  Assume \eqref{eq:TimeReg} with some $\varepsilon \in (0,1)$ and some $\dtime \geq 1$. Let the grading parameter $\sigma \in (0,1)$ be given. Choose the slope parameter $\muhp \geq 1$ such that
  \be \label{muhp}
      \muhp > \frac{(1-\sigma) \dtime}{2 \sigma^{(3+\varepsilon)/2}},
  \ee
  and fix the number of elements $m_2 \in \IN_0$ such that
  \be \label{m2}
    m_2  \begin{cases}
            = 0, & T \leq 1, \\
            > \frac{T-T_1}{4} \cdot \dtime \sigma^{-\frac{1+\varepsilon}{2 \lfloor \muhp \rfloor }}, & T>1,
          \end{cases}
  \ee
  where $T_1 = \min \{1,T\}.$ 
Then, for every $m_1 \in \IN$ with $m_1 \geq \max\{3,m_2\}$ and $m=m_1+m_2$, 
the geometric partition $\cG^m_\sigma$ of $J=(0,T)$, 
which is given by the time steps $t_j$ in \eqref{def:tj} with time-step sizes $k_j$ in \eqref{def:kj}, 
and the temporal order distribution $\bmp \in \IN^m$ defined by \eqref{def:pj},
lead
to the error bound
  \[
    \| u - \Pi^{\bmp,1}_{\cG^m_\sigma} u \|_{H^{1/2}_{0,}((t_2,T);X_2)}^2 \leq C \sigma^{\varepsilon m_1},
  \]
  with $t_2 = T_1 \sigma^{m_1 - 2}$ and a constant $C>0$ independent of $m_1$.
\end{lemma}
\begin{proof}
  Set $w=u - \Pi^{\bmp,1}_{\cG^m_\sigma} u$. Since $w \in H^1_{0,}((t_2,T);X_2)$, see the nodal property \eqref{projNodal}, the interpolation estimate (Lemma~\ref{lem:interpolationEstimate}) yields
  \begin{equation*}
    \| w \|_{H^{1/2}_{0,}((t_2,T);X_2)}^2 \leq \| \partial_t w \|_{L^2((t_2,T);X_2)} \| w \|_{L^2((t_2,T);X_2)}.
  \end{equation*}
  We estimate both factors on the right side using Proposition~\ref{Prop:tExpBd}, which states estimates for $b \leq \min\{1,T\}=T_1$ only. Thus, we split $[0,T]$ into the two intervals $[0,T_1]$ and $[T_1,T]$ for the case $T>1$. Without loss of generality, let us assume that $T>1$, i.e., $T_1=1$ (otherwise we examine only $[0,T] \subset [0,1]$ and omit the considerations for the second interval $[T_1,T]$). We investigate the intervals $[0,T_1]$ and $[T_1,T]$ separately.

  \noindent
  \textbf{Interval $[0,T_1]$:} 
  With $\lambda = \frac{1-\sigma}{\sigma}$, the time-step size fulfills $k_j = t_j - t_{j-1} = t_{j-1} \lambda$ for $j=2,\dots,m_1$. Lemma~\ref{lem:hpIj} with $p_j=\lfloor \muhp j \rfloor$, $s_j=j$ and estimate~\eqref{eq:IntEst} in Proposition~\ref{Prop:tExpBd} yield
  \begin{align*}
        &\| \partial_t w \|_{L^2((t_2,T_1);X_2)}^2 = \sum_{j=3}^{m_1} \| \partial_t w \|_{L^2(I_j;X_2)}^2 \\
        &\leq C \sum_{j=3}^{m_1} \frac{( \lfloor \muhp j \rfloor - j)!}{(\lfloor \muhp j \rfloor + j)!} \left( \frac{\lambda}{2} \right)^{2j} t_{j-1}^{2j} \dtime^{2(j+1)} \Gamma(j+3)^2 t_{j-1}^{-2(j+1)+1+\varepsilon} \\
        &\stackrel{\eqref{def:tj}}{=}C \dtime^2 T_1^{-1+\varepsilon} \sigma^{(m_1+1)(-1+\varepsilon)} \sum_{j=3}^{m_1} \frac{\Gamma(\lfloor \muhp j \rfloor-j+1)}{\Gamma(\lfloor \muhp j \rfloor +j+1)} \left( \frac{\lambda \dtime}{2 \sigma^{(-1+\varepsilon)/2}} \right)^{2j} \Gamma(j+3)^2 \\
        &\leq C_1 \sigma^{m_1(-1+\varepsilon)},
  \end{align*}
  where, in the last step, Lemma~\ref{lem:BoundGamma} is applied for $\muhp > \alpha = \frac{(1-\sigma) \dtime}{2 \sigma^{(3+\varepsilon)/2}} \geq \frac{(1-\sigma) \dtime}{2 \sigma^{(1+\varepsilon)/2}}= \frac{\lambda \dtime}{2 \sigma^{(-1+\varepsilon)/2}}$ with \eqref{muhp} and the constant $C_1>0$ is independent of $m_1$. In the same way, we get from Lemma~\ref{lem:BoundGamma} for $ \muhp > \alpha = \frac{(1-\sigma) \dtime}{2 \sigma^{(3+\varepsilon)/2}} = \frac{\lambda \dtime}{2 \sigma^{(1+\varepsilon)/2}}$ with \eqref{muhp} that 
  \begin{multline*}
    \| w \|_{L^2((t_2,T_1);X_2)}^2 = \sum_{j=3}^{m_1} \| w \|_{L^2(I_j;X_2)}^2 \\
    \leq C \sigma^{m_1(1+\varepsilon)} \sum_{j=3}^{m_1} \frac{\Gamma(\lfloor \muhp j \rfloor-j+1)}{\Gamma(\lfloor \muhp j \rfloor +j+1)} \left( \frac{\lambda \dtime}{2 \sigma^{(1+\varepsilon)/2}} \right)^{2j} \Gamma(j+3)^2 \leq C_2 \sigma^{m_1(1+\varepsilon)}
  \end{multline*}
  with a constant $C_2>0$ independent of $m_1$.

  \noindent
  \textbf{Interval $[T_1,T]$ in the case $T>1$:} 
  First, note that $T_1 = 1$. From Lemma~\ref{lem:hpIj} with the choices 
  $p_j=s_j=p_T:=\lfloor \muhp m_1 \rfloor$ and $k_j=k_T$, 
  estimate~\eqref{eq:L2abX2} in Lemma~\ref{lem:boundab}, 
  and the Stirling's formula
$\sqrt{2\pi n}\left(\frac{n}{\mathrm{e}}\right)^n
<n!<\sqrt{2\pi n}\left(\frac{n}{\mathrm{e}}\right)^n \mathrm{e}^{\frac{1}{12n}}$ with $\mathrm{e}^{\frac{1}{6n}}<2$,
we get
  \begin{align*}
        \| \partial_t w \|_{L^2((1,T);X_2)}^2 =& \sum_{j=m_1+1}^{m} \| \partial_t w \|_{L^2(I_j;X_2)}^2 \\
        \leq& C \frac{1}{(2 p_T )!} \left( \frac{k_T}{2} \right)^{2 p_T} \underbrace{\sum_{j=m_1+1}^m \| \partial_t^{(p_T+1)} u \|_{L^2(I_j;X_2)}^2}_{=\| \partial_t^{(p_T+1)} u \|_{L^2((1,T);X_2)}^2}\\
        \leq& C \dtime^2 \frac{1}{(2 p_T )!} \left( \frac{k_T \dtime}{2} \right)^{2 p_T} \Gamma(p_T+3)^2 C(\varepsilon,1,T)^2 \\
        \leq& C \dtime^2  C(\varepsilon,1,T)^2 (p_T+2)^4 \frac{ 4\pi p_T \left( \frac{p_T}{\mathrm{e}} \right)^{2p_T} } { \sqrt{2 \pi} \sqrt{2p_T} \left( \frac{2p_T}{\mathrm{e}} \right)^{2p_T} }    \left( \frac{k_T \dtime}{2} \right)^{2 p_T} \\
        =& C 2 \sqrt{\pi} \dtime^2  C(\varepsilon,1,T)^2 (p_T+2)^4 \sqrt{p_T}    \left( \frac{k_T \dtime}{4} \right)^{2 p_T} \leq C_3 \sigma^{m_1(1+\varepsilon)}
  \end{align*}
  with a constant $C_3 >0$ independent of $m_1$. In the last step, due to \eqref{m2}, we use that a constant $q \in (0,1)$ exists such that
  \[
    k_T = \frac{T-T_1}{m_2} =   q \frac{4}{\dtime} \sigma^{\frac{m_1(1+\varepsilon)}{2 \lfloor \muhp \rfloor m_1 }} \leq q \frac{4}{\dtime} \sigma^{\frac{m_1(1+\varepsilon)}{2 \lfloor \muhp m_1 \rfloor }} = q \frac{4}{\dtime} \sigma^{\frac{m_1(1+\varepsilon)}{2 p_T}}
  \]
  and therefore, $(p_T+2)^4 \sqrt{p_T} q^{2p_T} \to 0$ as $p_T \to \infty$.
  Analogously, we obtain 
  \[
    \| w \|_{L^2((1,T);X_2)}^2 \leq C_4 \sigma^{m_1(1+\varepsilon)}
  \]
  with a constant $C_4 >0$ independent of $m_1$.

  With all estimates above, we conclude that  
  \begin{multline*}
    \| w \|_{H^{1/2}_{0,}((t_2,T);X_2)}^2 \leq  \| \partial_t w \|_{L^2((t_2,T);X_2)} \| w \|_{L^2((t_2,T);X_2)} \\
    \leq \sqrt{ C_1 \sigma^{m_1(-1+\varepsilon)} +  C_3 \sigma^{m_1(1+\varepsilon)}} \sqrt{ C_2 \sigma^{m_1(1+\varepsilon)} + C_4 \sigma^{m_1(1+\varepsilon)} } \leq C_{\mathrm{est}} \sigma^{\varepsilon m_1},
  \end{multline*}
  where $C_{\mathrm{est}}>0$ is independent of $m_1$.
\end{proof}

\begin{lemma} \label{lem:thpexpconv}
  Under the assumptions of Lemma~\ref{lem:hperrorsmoothX2}, the estimate
  \[
      \| u - \Pi^{\bmp,1}_{\cG^m_\sigma} u \|_{H^{1/2}_{0,}(J;X_2)} \leq C \mathrm{exp}(-b \sqrt{M})
  \]
  holds true with a constant $C$ independent of $b$ and $M$, where $b= - \varepsilon \ln \sigma / \sqrt{8\muhp} > 0$ and $M = {\rm dim}(S^{\bmp,1}_{0,}(J;\cG^m_\sigma)) \leq 2 \muhp m_1^2 \leq 2 \muhp m^2$.
\end{lemma}
\begin{proof}
  Set $w=u - \Pi^{\bmp,1}_{\cG^m_\sigma} u$. Then, for $X_2$-valued functions, the norm equivalence in 
  Lemma~\ref{lem:NormEquivalence} and the localization in 
  Lemma~\ref{lem:FractionalNormPointtau} for $a=0$, $b=T$, $\tau = t_2$ yield
  \begin{align*}
    &(C_{\mathrm{Int},1})^2 \| w \|_{H^{1/2}_{0,}(J;X_2)}^2
    \leq \normiii{w}_{H^{1/2}_{0,}(J;X_2)}^2 \\
    &\qquad \leq
    \| w \|^2_{L^2(J;X_2)} 
    + 
     | w |_{H^{1/2}((0,t_2);X_2)}^2  + 4 \int_0^{t_2} \frac{\|w(t)\|_{X_2}^2}{t_2-t} \mathrm dt \\
     &\qquad\quad + 4 \int_{t_2}^T \frac{\|w(s)\|_{X_2}^2}{s-t_2} \mathrm ds  +   | w |_{H^{1/2}((t_2,T);X_2)}^2
    +
    \int_0^T \frac{\|w(t)\|_{X_2}^2}{t} \mathrm dt \\
    &\qquad\leq \normiii{w}_{H^{1/2}_{0,}((0,t_2);X_2)}^2  + 4 \int_0^{t_2} \frac{\|w(t)\|_{X_2}^2}{t_2-t} \mathrm dt +   5 \normiii{w}_{H^{1/2}_{0,}((t_2,T);X_2)}^2,
  \end{align*}
  where we used the definition~\eqref{Sob:NormTriple} of the triple norm and the bound 
  $\int_{t_2}^T \frac{\|w(t)\|_{X_2}^2}{t} \mathrm dt \leq \int_{t_2}^T \frac{\|w(s)\|_{X_2}^2}{s-t_2} \mathrm ds$.
  Next, we estimate the three terms on the right side.
  
  \noindent
  \textbf{First term:}
The triangle inequality, Lemma~\ref{lem:NormEquivalence}, Lemma~\ref{lem:H12X2Reg}, 
     the Poincar\'e inequality (Lemma~\ref{lem:Poincare}), definition~\eqref{def:projection}, 
     and estimates \eqref{eq:dltuBd}, \eqref{eq:IntEst} yield
  \begin{align*}
    &\normiii{w}_{H^{1/2}_{0,}((0,t_2);X_2)}^2 \leq 2 \normiii{u}_{H^{1/2}_{0,}((0,t_2);X_2)}^2 + 2 \normiii{\Pi^{\bmp,1}_{\cG^m_\sigma} u}_{H^{1/2}_{0,}((0,t_2);X_2)}^2 \\
    &\; \leq C t_2^{\varepsilon} + 2 (C_{\mathrm{Int},2})^2 \sqrt{ 1 + \frac{4t_2^2}{\pi^2}} \underbrace{ \| \Pi^{\bmp,1}_{\cG^m_\sigma} u \|_{H^{1/2}_{0,}((0,t_2);X_2)}^2 }_{\leq \frac{2}{\pi} (t_2-0) \|\partial_t \Pi^{\bmp,1}_{\cG^m_\sigma} u \|_{L^2((0,t_2);X_2)}^2 } \\
    &\; \leq C T_1^{\varepsilon} \sigma^{-2\varepsilon} \sigma^{\varepsilon m_1} + C t_2 \Big[ \underbrace{\int_0^{t_1} \frac{\| u(t_1) \|_{X_2}^2}{t_1^2} \mathrm dt}_{\leq C t_1^{-1+\varepsilon}} +  \underbrace{ \int_{t_1}^{t_2} \|\Pi^{p_2-1}_{L^2(I_2)} \partial_t u (t)\|_{X_2}^2 \mathrm dt }_{\leq \|\partial_t u \|_{L^2(I_2;X_2)}^2 \leq  C \dtime^2 t_1^{-1+\varepsilon}} \Big] \leq C_1 \sigma^{\varepsilon m_1},
  \end{align*}
  with a constant $C_1 >0$ independent of $m_1$, where we used
  \[
    t_2 t_1^{-1+\varepsilon} = T_1 \sigma^{m_1-2} T_1^{-1+\varepsilon} \sigma^{(-1+\varepsilon)(m_1-1)} = T_1^{\varepsilon} \sigma^{-1-\varepsilon} \sigma^{\varepsilon m_1}.
  \]

  \noindent
  \textbf{Second term:}
  With the bound~\eqref{eq:dltuBd}, the nodal property \eqref{projNodal} and $k_1=t_1 = T_1 \sigma^{m_1-1}$, $k_2=T_1 \sigma^{m_1-2}(1-\sigma)$, we find
  \begin{align*}
    4& \int_0^{t_2} \frac{\|w(t)\|_{X_2}^2}{t_2 - t} \mathrm dt = 4 \int_0^{t_1} \frac{\|w(t)\|_{X_2}^2}{t_2 - t} \mathrm dt + 4 \int_{t_1}^{t_2} \frac{\|w(t)\|_{X_2}^2}{t_2 - t} \mathrm dt  \\
    &=  4 \int_0^{t_1} \frac{\|u(t) - u(t_1)t/t_1\|_{X_2}^2}{t_2 - t} \mathrm dt + 4 \int_{t_1}^{t_2} \frac{\| \int_t^{t_2} \partial_t w(\xi) \mathrm d\xi \|_{X_2}^2}{t_2 - t} \mathrm dt  \\
    &\leq \frac{8}{k_2} \int_0^{t_1} \|u(t)\|_{X_2}^2 \mathrm dt + \frac{8}{k_2} \int_0^{t_1} \|u(t_1)\|_{X_2}^2 \frac{t^2}{k_1^2} \mathrm dt + 4 \int_{t_1}^{t_2} \frac{\left[ \int_t^{t_2} \| \partial_t w(\xi) \|_{X_2} \mathrm d\xi \right]^2}{t_2 - t} \mathrm dt \\
    &\leq \frac{C}{k_2} \int_0^{t_1} t^{\varepsilon} \mathrm dt + \frac{C  t_1^{\varepsilon}}{k_1^2 k_2} \int_0^{t_1} t^2 \mathrm dt + 4 \int_{t_1}^{t_2} \| \partial_t w \|_{L^2((t,t_2);X_2)}^2 \mathrm dt \\
    &\leq 2C  \frac{T_1^{\varepsilon} \sigma^{1-\varepsilon}}{1-\sigma} \sigma^{\varepsilon m_1} + 4 k_2 \| \partial_t w \|_{L^2(I_2;X_2)}^2 \leq C_2 \sigma^{\varepsilon m_1},
  \end{align*}
with a constant $C_2 >0$ independent of $m_1$, 
where in the last step we have used the estimate~\eqref{eq:IntEst}. 
This yields
  \begin{equation*}
     4 k_2 \| \partial_t u \|_{L^2(I_2;X_2)}^2 \leq C T_1 \sigma^{m_1-2}(1-\sigma) t_1^{-1+\varepsilon} = C  T_1^{\varepsilon} \frac{1-\sigma}{\sigma^{1+\varepsilon}} \sigma^{\varepsilon m_1}.
  \end{equation*}

  \noindent
  \textbf{Third term:}
  Lemma~\ref{lem:NormEquivalence} and Lemma~\ref{lem:hperrorsmoothX2} give
  \[
    5 \normiii{w}_{H^{1/2}_{0,}((t_2,T);X_2)}^2 
    \leq 
   C (C_{\mathrm{Int},2})^2 \sqrt{ 1 + T^2} \| w \|_{H^{1/2}_{0,}((t_2,T);X_2)}^2 
    \leq 
    C_3 \sigma^{\varepsilon m_1},
  \]
  with a constant $C_3 >0$ independent of $m_1$.

 \noindent
  \textbf{Conclusion of the proof:}
  As the temporal number of degrees of freedom $M$ fulfills
  \be \label{MEstimate}
    M \leq \sum_{j=1}^{m_1} \lfloor \muhp j \rfloor + \lfloor \muhp m_1 \rfloor m_2 
      \leq \muhp \frac{m_1(m_1+1)}{2} + \muhp m_1^2 \leq 2 \muhp m_1^2
  \ee
  with $m_2 \leq m_1$, using all the estimates above, we conclude
  \[
    \| w \|^2_{H^{1/2}_{0,}(J;X_2)} \leq C_4 \sigma^{\varepsilon m_1} \leq C_4 \mathrm{exp}(-2b \sqrt{M}),
  \]
with a constant $C_4 >0$ independent of $m_1$, $M$ 
and 
$b= - \varepsilon \ln \sigma / \sqrt{8\muhp} > 0$, 
i.e., the assertion follows.
\end{proof}
As Lemma~\ref{lem:thpexpconv} implies
exponential convergence bounds on {\bf Term$1$} and {\bf Term$5$},
it remains to treat {\bf Terms}$2$--$4$ in \eqref{eq:xtQopt1}.
{\bf Term$2$} and {\bf Term$4$} are identical.
We focus on {\bf Term$3$}. 
Using that $Q^{1/2}_t$ is a projector in the Hilbert space
$H^{1/2}_{0,}(J)$, 
the triangle inequality gives 
\[
\left\| (I-Q^{1/2}_t)(I-Q^1_x) u \right\|_{H^{1/2}_{0,}(J;L^2(\domain))}
\leq 
2 \left\| u-Q^1_x u \right\|_{H^{1/2}_{0,}(J;L^2(\domain))}\;.
\]
Thus, {\bf Term$3$} can be estimated in the same way as
{\bf Term$2$} and {\bf Term$4$}. 
Using
\[
H^{1/2}_{0,}(0,T;L^2(\domain)) \simeq H^{1/2}_{0,}(J)\otimes L^2(\domain)
\simeq 
L^2(\domain) \otimes H^{1/2}_{0,}(J) \simeq L^2(\domain;H^{1/2}_{0,}(J)),
\]
we may use the $L^2(\domain)$ error bound \eqref{eq:hFEML2Bd} on the 
Ritz projection $Q^1_x$ and the regularity result in Lemma~\ref{lem:H12X2Reg} for $b=T$,
in connection with the norm equivalence in 
Lemma~\ref{lem:NormEquivalence} for $a=0$, $b=T$,
to arrive at
\begin{equation}\label{eq:Q1xErr}
\left\| u-Q^1_x u \right\|_{H^{1/2}_{0,}(J;L^2(\domain))}
\leq 
c N^{-2/d}
\left\| u \right\|_{H^{1/2}_{0,}(J;X_2)} \leq C N^{-2/d},
\end{equation}
with a constants $c>0$, $C>0$ independent of $N$.

We combine the previous estimates %results 
to obtain the main result of this paper.

\begin{theorem}\label{thm:Conv}
  Let the space dimension $d$ be either $d=1$ or $d=2$.
  Assume that the diffusion coefficient 
  $A\in W^{1,\infty}(\domain;\IR^{d\times d}_{\mathrm{sym}})$ 
  is uniformly positive definite, i.e., that \eqref{eq:Acoerc} is satisfied,
  and that the forcing $g$ in \eqref{eq:IBVP} satisfies 
  the temporal analytic regularity~\eqref{eq:TimeReg}. 
Furthermore, assume that the assumptions of Lemma~\ref{lem:hperrorsmoothX2} on the temporal mesh $\cG^m_\sigma$ in \eqref{def:tj} with $\muhp \geq 1$ and $m_2 \in \IN_0$ fulfilling \eqref{muhp} and \eqref{m2}, respectively, and the temporal order distribution $\bmp \in \IN^m$ in \eqref{def:pj} are satisfied.

  Then the space-time Galerkin 
  approximation \eqref{eq:IBVPVarxt} admits a unique solution 
  $u^{MN} \in S^{\bmp,1}_{0,}(J;\cG^m_\sigma)\otimes S^1_{\Gamma_D}(\domain;\cT^N_\beta)$
  with the temporal $hp$-FE space $S^{\bmp,1}_{0,}(J;\cG^m_\sigma)$ 
  of dimension $M = {\rm dim}(S^{\bmp,1}_{0,}(J;\cG^m_\sigma))$
  as defined in \eqref{eq:VMt}, and 
  with the spatial FE space $S^1_{\Gamma_D}(\domain;\cT^N_\beta)$
  of continuous, piecewise linear FEM
  on a sequence of
  suitably graded, regular triangulations $\{\cT^N_\beta\}_{N}$
  in $\domain$ ($\beta=1$, i.e., quasi-uniform partitions, if $d=1$) 
  of dimension $N={\rm dim}(S^1_{\Gamma_D}(\domain;\cT^N_\beta))$.

  Moreover, 
  a constant $C>0$ (independent of $M$ and $N$) exists
  such that the space-time discretization \eqref{eq:IBVPVarxt} 
  based on these spaces satisfies the error bound
  \begin{equation}\label{eq:hpxtErrBd}
  \| u - u^{MN} \|_{H^{1/2}_{0,}(J;L^2(\domain))}  
  \leq 
  C 
  \left(\exp(-b\sqrt{M}) + N^{-2/d} \right)
  \end{equation}
  with $b= - \varepsilon \ln \sigma / \sqrt{8\muhp} > 0.$
\end{theorem}

\begin{proof}
Existence and uniqueness of the solution $u^{MN}$
were established at the end of Section~\ref{sec:xtVarForm}. 
Estimate~\eqref{eq:hpxtErrBd} follows from Lemma~\ref{lem:xtQopt},
taking into account
Lemma~\ref{lem:thpexpconv}, and estimate~\eqref{eq:Q1xErr}.
\end{proof}

Balancing the terms in the upper bound \eqref{eq:hpxtErrBd} results in
\[
  M \simeq O\left((\log N)^2\right) \quad \text{ or } \quad m_1  \simeq O(\log N),
\]
where $M \leq 2 \muhp m_1^2 \leq 2 \muhp m^2,$ see \eqref{MEstimate}.
Then, 
the \emph{number of degrees of freedom for the space-time discretization}
behaves, as $N\to\infty$, 
as
\be \label{ST:dofsBehavior}
  MN \simeq O\left(N (\log N)^2\right),
\ee
i.e., it is essentially (up to the $(\log N)^2$ %$\log(N_x)^2$
factor) equal to the number of
degrees of freedom for the discretization of one spatial problem.
Importantly, in the solution algorithms of~\cite{langer2020efficient}, 
$M \simeq O\left((\log N)^2\right)$ 
will reduce time and memory requirements.

\begin{remark} \label{remk:pFEM}
  Theorem~\ref{thm:Conv} remains valid for solutions $u(t,\circ)$ which depend analytically on 
  $t\in [0,T]$. 
  Classical results on
  exponential rates of convergence for polynomial approximation of analytic functions in $[0,T]$
  (e.g.,~\cite[Chapter~12]{Davis}) imply that for any constant number of temporal elements $m \in \IN$ (e.g., $m=1$) with temporal polynomial degrees $\bmp = (p,\dots,p) \in \IN^m$ with $p \in \IN$, temporal exponential convergence follows when $p\to\infty$ ($p$-method).
  Under the otherwise exact same assumptions as in Theorem~\ref{thm:Conv}, 
  one obtains in place of~\eqref{eq:hpxtErrBd} the error bound
  \be \label{pxtErrBd}
      \| u - u^{MN} \|_{H^{1/2}_{0,}(J;L^2(\domain))}
      \leq
      C
      \left(\exp(-bp) + N^{-2/d} \right)
  \ee
  with $M = m p$ and constants $b>0$, $C>0$ independent of $p$ and $N$.
  This allows to improve \eqref{ST:dofsBehavior} to 
  \be \label{ST:dofsBehaviorpFEM}
      MN \simeq O(N \log N).
  \ee
\end{remark}

%%%%%%%%%%%%%%%%%%%%%%%%%%%%%%%%%%%%%%%%%%%%%%%%%%%%%%%%%%%%%%%%%%%%%%%%%%%%%%
\section{Numerical Experiments}
\label{sec:NumExp}
%%%%%%%%%%%%%%%%%%%%%%%%%%%%%%%%%%%%%%%%%%%%%%%%%%%%%%%%%%%%%%%%%%%%%%%%%%%%%%
%
In this section, we present numerical examples for the space-time Galerkin 
approximation~\eqref{eq:IBVPVarxt} of the heat equation 
 with homogeneous Dirichlet conditions 
\be\label{eq:modelproblem}
    \partial_t u - \Delta_x u = g \quad \text{ in } Q, \quad u|_{t=0} = 0, \quad \gamma_0(u) = 0 \quad \text{ on } \partial \domain,
\ee
i.e., $A(\partial_x) = - \Delta_x$ in \eqref{eq:IBVP}, $u_0=0$ in \eqref{eq:IC} and $u_D= 0$ with $\Gamma_D = \Gamma = \partial \domain$ in \eqref{eq:BC}. We use globally continuous 
functions, which are piecewise linear in space and piecewise polynomials of 
higher-order in time, see Theorem~\ref{thm:Conv}.
We start by deriving the algebraic linear system associated
with~\eqref{eq:IBVPVarxt}, and by describing 
the realization of the operator $\HT$ for a 
temporal $hp$-FEM.

For~\eqref{eq:modelproblem}, we solve
the discrete space-time variational formulation to
find $u^{MN} \in S^{\bmp,1}_{0,}(J;\cG^m_\sigma)\otimes S^1_{\Gamma_D}(\domain;\cT^N_\beta)$
such that
\be\label{eq:IBVPVarxtPerturbed}
\langle \partial_t u^{MN} , v \rangle_{L^2(Q)}
+
\langle \nabla_x u^{MN}, \nabla_x v \rangle_{L^2(Q)}
=
\langle \Pi^{MN} g, v \rangle_{L^2(Q)}
\ee
is satisfied for all $v \in (\HT S^{\bmp,1}_{0,}(J;\cG^m_\sigma))
\otimes S^1_{\Gamma_D}(\domain;\cT^N_\beta).$
Here, we use the notation of Section~\ref{sec:Approx} with
\be \label{Num:BasisVt}
    S^{\bmp,1}_{0,}(J;\cG^m_\sigma) =: V^M_t:= \mathrm{span}\{\varphi_l\}_{l=1}^M,
 \ee
and    
\[
    S^1_{\Gamma_D}(\domain;\cT^N_\beta)=: V^N_x:=\mathrm{span}\{\psi_i\}_{i=1}^N,
\]
where the functions $\varphi_l$ are basis functions in time, and the
functions $\psi_i$ are the usual nodal basis functions in space. The
total number of degrees of freedom is
\[
    M N = \dim \left( S^{\bmp,1}_{0,}(J;\cG^m_\sigma)\otimes S^1_{\Gamma_D}(\domain;\cT^N_\beta) \right).
\]
In
addition, for an easier implementation, we approximate the right-hand
side $g\in L^2(Q)$ by $g \approx \Pi^{MN} g$, where $\Pi^{MN} \colon
\, L^2(Q) \to S^{\bmp,1}(J;\cG^m_\sigma)\otimes
S^1(\domain;\cT^N_\beta)$ is the space-time $L^2(Q)$ projection,
namely $\Pi^{MN} g \in S^{\bmp,1}(J;\cG^m_\sigma)\otimes
S^1(\domain;\cT^N_\beta)$ is such that
\be \label{num:projRHS}
    \forall w \in S^{\bmp,1}(J;\cG^m_\sigma)\otimes S^1(\domain;\cT^N_\beta): \, \langle \Pi^{MN} g, w \rangle_{L^2(Q)} = \langle g, w \rangle_{L^2(Q)}.
\ee
Note that the spaces $S^1(\domain;\cT^N_\beta)$ and
$S^{\bmp,1}(J;\cG^m_\sigma)$ do not necessarily satisfy the
homogeneous Dirichlet and initial conditions, respectively; see
beginning of Sections~\ref{sec:hpApprJ} and~\ref{Sec:xAppr}.
We denote
the temporal mesh width (i.e., the maximal time-step size) by
$k_{\max} = \max_{j=1,\dots,m} k_j,$
the spatial mesh width by $h_x$, and the space-time mesh width by
$h_{xt} = \max \{ k_{\max}, h_x \}.$
The space-time error
$\| u - u^{MN} \|_{H^{1/2}_{0,}(J;L^2(\domain ))}$ 
mandates the numerical evaluation of the fractional order norm $\|
\circ \|_{H^{1/2}_{0,}(J;L^2(\domain ))}$. In order to overcome this
problem, we introduce the quantity 
\[
  [v]_{H^{1/2}_{0,}(J;L^2(\domain ))}
:= 
\sqrt{ \| v \|_{L^2(Q)} \cdot \| \partial_t v \|_{L^2(Q)} }, 
\]
which is defined for $v \in H^1_{0,}(J;L^2(\domain ))$,
and observe that, provided that $u \in H^1_{0,}(J;L^2(\domain ))$,
\[
 \| u - u^{MN} \|_{H^{1/2}_{0,}(J;L^2(\domain ))} 
  \leq  
  [u - u^{MN}]_{H^{1/2}_{0,}(J;L^2(\domain ))},
\]
due to the interpolation estimate (Lemma~\ref{lem:interpolationEstimate}).
Therefore, in the experiments below, 
instead of the space-time error 
$\| u - u^{MN} \|_{H^{1/2}_{0,}(J;L^2(\domain))}$, 
we consider its upper bound $[u- u^{MN}]_{H^{1/2}_{0,}(J;L^2(\domain ))}$,
which can be numerically evaluated via \emph{local} integration.

The fully discrete, space-time
variational formulation \eqref{eq:IBVPVarxtPerturbed} 
is equivalent to the global linear system
\be \label{eq:LS}
    B^{MN} \boldsymbol u = \boldsymbol{G},
\ee
with the system matrix
\[
  B^{MN}=A_t^{\HT} \otimes M_x + M_t^{\HT} \otimes A_x  \in \IR^{ M \cdot N \times M \cdot N},
\]
where $\otimes$ is the Kronecker product,
$M_x \in \IR^{N \times N}$ and $A_x \in \IR^{N \times N}$ denote
the spatial mass and stiffness matrices given by
\[
  M_x[i,j] = \langle \psi_j, \psi_i \rangle_{L^2(\domain)},
        \quad A_x[i,j] = \langle \nabla_x \psi_j, \nabla_x \psi_i
        \rangle_{L^2(\domain)} %, \quad i,j=1,\dots, N,
\]
for $i,j=1,\dots, N$,        
and $M_t^{\HT} \in \IR^{M \times M}$ and $A_t^{\HT} \in \IR^{M \times M}$ are defined by
\be \label{eq:matricesTime}
  M_t^{\HT}[k, l] := \spf{\varphi_l}{\HT \varphi_k}_{L^2(J)}, \quad
  A_t^{\HT}[k, l] := \spf{\partial_t \varphi_l}{\HT
    \varphi_k}_{L^2(J)}
\ee
for $k,l=1,\dots,M$. 
Note that, due to the nonlocality of $\HT$,
  the matrices $M_t^{\HT}$ and $A_t^{\HT}$ are densely populated.
Furthermore, the temporal stiffness matrix $A_t^{\HT}$ 
is symmetric (due to~\eqref{eq:HT3}) and positive definite (due
to~\eqref{eq:HT2}), whereas $M_t^{\HT}$ is nonsymmetric and
positive definite (due to~\eqref{eq:HT7}).
The assembling of the matrices  $M_t^{\HT}$ and $A_t^{\HT}$ is
described in Subsection~\ref{sec:NumHT} below.
For the right-hand side $\boldsymbol{G}$, the integrals for
computing the projection $\Pi^{MN} g$ in \eqref{num:projRHS} are
calculated by using high-order quadrature rules.
The global linear system \eqref{eq:LS} is 
solved in MATLAB by using the Bartels-Stewart method with real-Schur decomposition, 
see \cite[Algorithm~4.1]{langer2020efficient}. 
All calculations presented in this section were performed 
on a PC with two Intel Xeon E5-2687W v4 CPUs 3.00 GHz, i.e., 
in sum $24$ cores and $512$ GB main memory.

%%%%%%%%%%%%%%%%%%%%%%%%%%%%%%%%%%%%%%%%%%%%%%%%%%%%%%%%%%%%%%%%%%%%%%%%%%%%%%
\subsection{Numerical Implementation of $\HT$}
\label{sec:NumHT}
%%%%%%%%%%%%%%%%%%%%%%%%%%%%%%%%%%%%%%%%%%%%%%%%%%%%%%%%%%%%%%%%%%%%%%%%%%%%%%
We describe the assembling of the matrices
$M_t^{\HT}$ and $A_t^{\HT}$ in \eqref{eq:matricesTime}. The crucial
point is the realization of the modified Hilbert transformation $\HT$,
for which different possibilities exist, see
\cite{SteinbachZankNoteHT, ZankExactRealizationHT}. 
In particular, for
a uniform degree vector $\bmp := (p, p, \dots, p)$ with a fixed, low
polynomial degree $p \in \IN$, e.g., $p=1$ or $p=2$, the matrices $M_t^{\HT}$ and
$A_t^{\HT}$ in \eqref{eq:matricesTime} can be calculated using a
series expansion based on the \textit{Legendre chi function}, which
converges very fast, independently of the temporal mesh widths; see
\cite[Subsection~2.2]{ZankExactRealizationHT}. As for the temporal
$hp$-FEM the degree vector $\bmp$ is not uniform, it is convenient to
apply numerical quadrature rules to numerically approximate
the matrix entries.

From the integral representation of $\HT$,
\[
    (\HT v)(t) =
    - \frac{2}{\pi} v(0) \ln \tan \frac{\pi t}{4T} -\frac{1}{\pi} 
    \int_0^{T} \ln \left[ \tan \frac{\pi (s+t)}{4T} \tan \frac{\pi \abs{t-s}}{4T} \right] \partial_t v(s) \mathrm ds,
\]
$t \in J$, $v \in H^1(J),$ as a weakly singular integral, see
\cite[Lemma 2.1]{SteinbachZankNoteHT}, we have
\begin{align}
    A_t^{\HT}[k, l] &=
    \langle \partial_t \varphi_l , \HT \varphi_k
    \rangle_{L^2(J)} \nonumber \\ 
    &= - \frac{1}{\pi} \int_0^T \partial_t \varphi_l(t)
    \int_0^T \ln \left[ \tan \frac{\pi (s+t)}{4T} \tan \frac{\pi \abs{t-s}}{4T} \right] \, \partial_t \varphi_k(s) \, \mathrm ds \, \mathrm dt   \label{Num:Aht_entries}
\end{align}
and
\begin{align}
  M_t^{\HT}[k, l] &= \langle  \varphi_l , \HT \varphi_k \rangle_{L^2(J)} \nonumber \\
  &= - \frac{1}{\pi} \int_0^T \varphi_l(t) \int_0^T \ln \left[ \tan \frac{\pi (s+t)}{4T} \tan \frac{\pi \abs{t-s}}{4T} \right] \, \partial_t \varphi_k(s) \, \mathrm ds \, \mathrm dt  \label{Num:Mht_entries}
\end{align}
for $k,l=1,\ldots,M$, with the temporal basis functions $\varphi_l$ in~\eqref{Num:BasisVt}. 
In the following, we only describe how to compute the matrix entries
$M_t^{\HT}[k, l]$ in \eqref{Num:Mht_entries}, since the matrix entries
$A_t^{\HT}[k, l]$ in \eqref{Num:Aht_entries} can be computed in the
same way.

The matrix entries $M_t^{\HT}[k, l]$ in \eqref{Num:Mht_entries} are computed element-wise for the partition $\cG^m_\sigma = \{I_j\}_{j=1}^m$ of $J$ into time intervals
$I_j = (t_{j-1},t_j)\subset J$, $j=1,\dots, m$. Fix two time
intervals $I_i = (t_{i-1},t_i)$, $I_j = (t_{j-1},t_j)$ with indices
$i, j \in \{1,\dots,m\}$ and related local polynomial degrees $p_i,
p_j \in \IN$.
We define the local matrix $M_t^{\HT,i,j} \in \IR^{(p_i+1) \times (p_j+1)}$ by
\begin{multline} \label{Num:Mht_local}
 M_t^{\HT,i,j}[\kappa, \ell] = \\
 - \frac{1}{\pi} \int_{t_{j-1}}^{t_j} \varphi_{\alpha(\ell,j)}(t) 
\int_{t_{i-1}}^{t_i} \ln \left[ \tan \frac{\pi (s+t)}{4T} \tan \frac{\pi \abs{t-s}}{4T} \right] \, 
                                   \partial_t \varphi_{\alpha(\kappa,i)}(s) \, \mathrm ds \, \mathrm dt
\end{multline}
for $\kappa=1,\dots,p_i+1$ and $\ell=1,\dots,p_j+1$. 
Here,
$\alpha(\kappa,i) \in \{0,1,\dots,M\}$ is the global index related to the local index 
$\kappa$ for the time interval $I_i$; similarly for $\alpha(\ell,j)$.
Notice that the function $\varphi_0$, corresponding to the vertex $t=0$, 
does not contribute to the global matrix $M_t^{\HT}$. On the reference interval $(-1,1)$,
we use the Lobatto polynomials (or integrated Legendre polynomials) as hierarchical shape functions, i.e., 
we set
\[
N_1(\xi) = \frac{1-\xi}{2},
\quad  
N_2(\xi) = \frac{1+\xi}{2},
\quad 
N_\ell(\xi) = \int_{-1}^\xi
L_{\ell-2}(\zeta) \, \mathrm d\zeta \quad \text{ for }\ell \geq 3,
\]
$\xi \in [-1,1]$,
where $L_\ell$ denotes the $\ell$-th Legendre polynomial on $[-1,1]$,
see~\cite[Chapter~3]{Schwab98}. 
With these shape functions and the affine transformation $T_\iota\colon \, [-1,1] \to [t_{\iota-1},t_\iota]$ 
for $\iota \in \{1,\dots,m\}$, the entries \eqref{Num:Mht_local} 
of the local matrix $M_t^{\HT,i,j}$ are 
\begin{multline} \label{Num:Mht_local_referenceelement}
    M_t^{\HT,i,j}[\kappa, \ell] =\\
    - \frac{k_j}{2\pi} \int_{-1}^1
    N_\ell(\eta) \int_{-1}^1
    \ln \left[ \tan \frac{\pi (T_i(\xi) + T_j(\eta))}{4T} 
             \tan \frac{\pi \abs{T_j(\eta)-T_i(\xi)}}{4T} \right] \, N'_\kappa(\xi) \, \mathrm d\xi \, \mathrm d\eta
\end{multline}
for $\kappa=1,\dots,p_i+1$ and $\ell=1,\dots,p_j+1$, where $k_j = \abs{t_j - t_{j-1}}$ 
is the length of the time interval $I_j$. To compute the integrals in \eqref{Num:Mht_local_referenceelement}, 
we split these integrals into regular and singular parts, see \cite[Subsection~3.1]{SteinbachZankNoteHT}. 
For the regular parts, a tensor Gauss quadrature is applied. In \cite[Subsection~3.1]{SteinbachZankNoteHT}, 
it is proposed to calculate the singular parts analytically or with an adapted numerical integration. 
As the polynomial degrees $p_i, p_j$ may be high, we use the latter. 
The singularity of the singular parts is of logarithmic type. 
Thus, we apply so-called classical and nonclassical Gauss--Jacobi quadrature rules of order adapted to $p_i, p_j$, 
see
\cite[Eq. (1.6), (1.7)]{GautscheLog},
to the singular parts. 
These adapted integration rules allow us to calculate the singular
parts exactly.
In summary, the matrix entries of the matrices $M_t^{\HT}$ and $A_t^{\HT}$ in \eqref{eq:matricesTime} are computable
to high float point accuracy efficiently.

%%%%%%%%%%%%%%%%%%%%%%%%%%%%%%%%%%%%%%%%%%%%%%%%%%%%%%%%%%%%%%%%%%%%%%%%%%%%%%
\subsection{Numerical Examples in 1D}
\label{sec:Num1D}
%%%%%%%%%%%%%%%%%%%%%%%%%%%%%%%%%%%%%%%%%%%%%%%%%%%%%%%%%%%%%%%%%%%%%%%%%%%%%%
We present a numerical example in the one-dimensional spatial domain
$\domain = (0,1) \subset \IR$ with final time $T=2$, i.e.,
$Q = J \times \domain = (0,2) \times (0,1)\subset \IR^2$.
We choose the constant right-hand side $g_1 \equiv 1$, 
for which the solution to problem~\eqref{eq:modelproblem} is given by the Fourier series
\be \label{Num:u1}
u_1(t,x) = \sum _{\eta=1}^{\infty} \frac{4 - 4 e^{-\pi ^2 (2 \eta-1)^2 t}}{\pi ^3 (2 \eta-1)^3}  
                                  \sin (\pi  (2 \eta-1) x), \quad (t,x) \in \overline{Q}.
\ee
In the calculation of the errors of the space-time Galerkin
approximation \eqref{eq:IBVPVarxtPerturbed}, we truncate the
series~\eqref{Num:u1} at $\eta=1000$.
For the spatial discretization, we choose a uniform initial mesh with mesh width $h_x$ and apply a uniform refinement strategy.

In the first test, we use a temporal mesh with mesh width
$k_{\max}=k_1=\dots=k_m$ and linear polynomials, i.e., $\bmp =
(1,\dots,1) \in \IN^m$. The errors and the estimated orders of
convergence (eoc) are reported in Table~\ref{Num:Tab:u1}.
We observe a reduced order of convergence, as the compatibility condition between the right-hand side $g_1 \equiv 1$ and the homogeneous initial condition is not satisfied.
Note that the forcing $g_1 \equiv 1$ satisfies the temporal analytic regularity \eqref{eq:TimeReg} for any $\varepsilon \in (0,1/2)$ with $\dtime=1$ and a constant $C=C(\varepsilon)$ depending on $\varepsilon$.
\begin{table}[h]
\begin{center}
\begin{tabular}{rcccccccccc} 
	\hline
 $MN$ & $h_x$ & $k_{\max}$ & $[u_1 - u_1^{MN}]_{H^{1/2}_{0,}(J;L^2(\domain ))}$ & eoc\\
		\hline   
          12 & 0.25000 & 0.50000 & 7.330e-02 &   -  \\ 
          56 & 0.12500 & 0.25000 & 3.423e-02 & 0.99 \\ 
         240 & 0.06250 & 0.12500 & 1.355e-02 & 1.27 \\ 
         992 & 0.03125 & 0.06250 & 5.396e-03 & 1.30 \\ 
        4032 & 0.01562 & 0.03125 & 2.267e-03 & 1.24 \\ 
       16256 & 0.00781 & 0.01562 & 9.531e-04 & 1.24 \\ 
       65280 & 0.00391 & 0.00781 & 4.004e-04 & 1.25 \\ 
      261632 & 0.00195 & 0.00391 & 1.682e-04 & 1.25 \\ 
     1047552 & 0.00098 & 0.00195 & 7.070e-05 & 1.25 \\ 
     4192256 & 0.00049 & 0.00098 & 2.971e-05 & 1.25 \\ 
		\hline
  \end{tabular}
  \caption{Numerical results with
    the space-time Galerkin
    approximation~\eqref{eq:IBVPVarxtPerturbed}
    for the 1D example with the right-hand side $g_1 \equiv 1$
and solution $u_1$ in \eqref{Num:u1}, for a uniform mesh refinement
strategy and piecewise linear polynomials both in space and time.} \label{Num:Tab:u1}
\end{center}
\end{table}
In the second test,
we use the temporal $hp$-approximation of
Subsection~\ref{sec:hpApprJ}. For this purpose, we apply a uniform
refinement strategy for the spatial discretization, i.e., the number
$N$ of degrees of freedom in the spatial discretization doubles with
each uniform refinement. Then, corresponding to a given spatial
discretization with parameter $N$, we choose the temporal mesh as in \eqref{def:tj} with
grading parameter $\sigma = 0.31$, slope parameter $\muhp = 2.0$, numbers of elements $m_1 = \lfloor 1.4 \cdot \ln N \rfloor$, $m_2 = 1$, and temporal polynomial degrees $\boldsymbol{p} \in \IN^m$ as in \eqref{def:pj}. This choice of the discretization parameters fulfills condition~\eqref{muhp} with $\muhp = 2.0 > \frac{345}{31 \sqrt{31}} \approx 1.99883$ and condition~\eqref{m2} with $m_2 = 1 > \frac{5}{2 \sqrt{31}} \approx 0.45$. In addition, this choice balances the terms of the error bound~\eqref{eq:hpxtErrBd}, i.e., the
total number of degrees of freedom $MN$ behaves like in
\eqref{ST:dofsBehavior}. 
The numerical results reported in Figure~\ref{Num:Fig:u1} 
confirm Theorem~\ref{thm:Conv}.
\begin{figure}
\begin{center}
    \begin{tikzpicture}[font=\footnotesize,scale=0.9]
			\begin{axis}[xlabel={degrees of freedom $MN$} ,ylabel={$[u_1 - u_1^{MN}]_{H^{1/2}_{0,}(J;L^2(\domain ))}$} ,ymode=log,xmode=log,
				legend entries={{ $\IP^1$-FEM uniform}, {temporal $hp$-FEM}, $(MN)^{-\frac 5 8} \sim h_{xt}^{1.25}$, $(MN)^{-1} \sim h_{xt}^2$},
				legend pos = {south west},
				legend style={cells={align=left}},
				mark size=3pt]
				\addplot[color=red,mark=otimes, mark options={scale=0.8},line width=1pt] plot file {Numerics/1du1/1dhFEMH12.txt};
                \addplot[color=blue,mark=square, mark options={scale=0.8},line width=1pt] plot file {Numerics/1du1/1dhpFEMH12.txt};
                \addplot[color=black] plot[samples=200,domain=100:8192256] {0.9/(x^(5/8))};
                \addplot[color=black, dashed] plot[samples=200,domain=10000:1092256] {2/(x^(1))};
			\end{axis}
    \end{tikzpicture}
    \caption{
Numerical results with the space-time Galerkin
approximation~\eqref{eq:IBVPVarxtPerturbed} for the 1D example
with the right-hand side $g_1 \equiv 1$ and solution $u_1$
in~\eqref{Num:u1}, for a spatial uniform mesh refinement and temporal
$\IP^1$-FEM approximations with uniform mesh refinement or with temporal
$hp$-FEM with geometric partition of~$J$ with
grading parameter $\sigma = 0.31$, 
slope parameter $\muhp = 2.0$, 
numbers of elements $m_1 = \lfloor 1.4 \cdot \ln N \rfloor$, $m_2 = 1$, 
and temporal polynomial degrees $\boldsymbol{p} \in \IN^m$ as in \eqref{def:pj}.
} 
\label{Num:Fig:u1}
\end{center}
\end{figure}
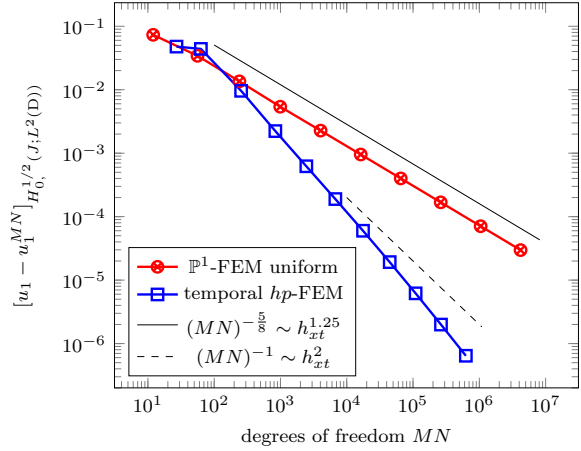

%%%%%%%%%%%%%%%%%%%%%%%%%%%%%%%%%%%%%%%%%%%%%%%%%%%%%%%%%%%%%%%%%%%%%%%%%%%%%%
\subsection{Numerical Examples in 2D}
\label{sec:Num2D}
%%%%%%%%%%%%%%%%%%%%%%%%%%%%%%%%%%%%%%%%%%%%%%%%%%%%%%%%%%%%%%%%%%%%%%%%%%%%%%
We present numerical examples in the two-dimensional spatial L-shaped domain
\[
  \domain =  (-1,1)^2 \setminus [0,1]^2 \subset \IR^2,
\]
and final time $T=2$, i.e., $Q = J \times \domain = (0,2) \times \domain \subset \IR^3$.

%%%%%%%%%%%%%%%%%%%%%%%%%%%%%%%%%%%%%%%%%%%%%%%%%%%%%%%%%%%%%%%%%%%%%%%
\subsubsection{Spatial Meshes}
%%%%%%%%%%%%%%%%%%%%%%%%%%%%%%%%%%%%%%%%%%%%%%%%%%%%%%%%%%%%%%%%%%%%%%%
For the spatial discretization, we consider uniformly refined
meshes, see Figure~\ref{Num:Fig:LShapeUniform}, or
meshes with corner-refinements towards the origin,
where in both cases, the mesh width $h_x$ decreases by a factor 2 with
each refinement.
\begin{figure}
\begin{center}
    \includegraphics[scale=0.5]{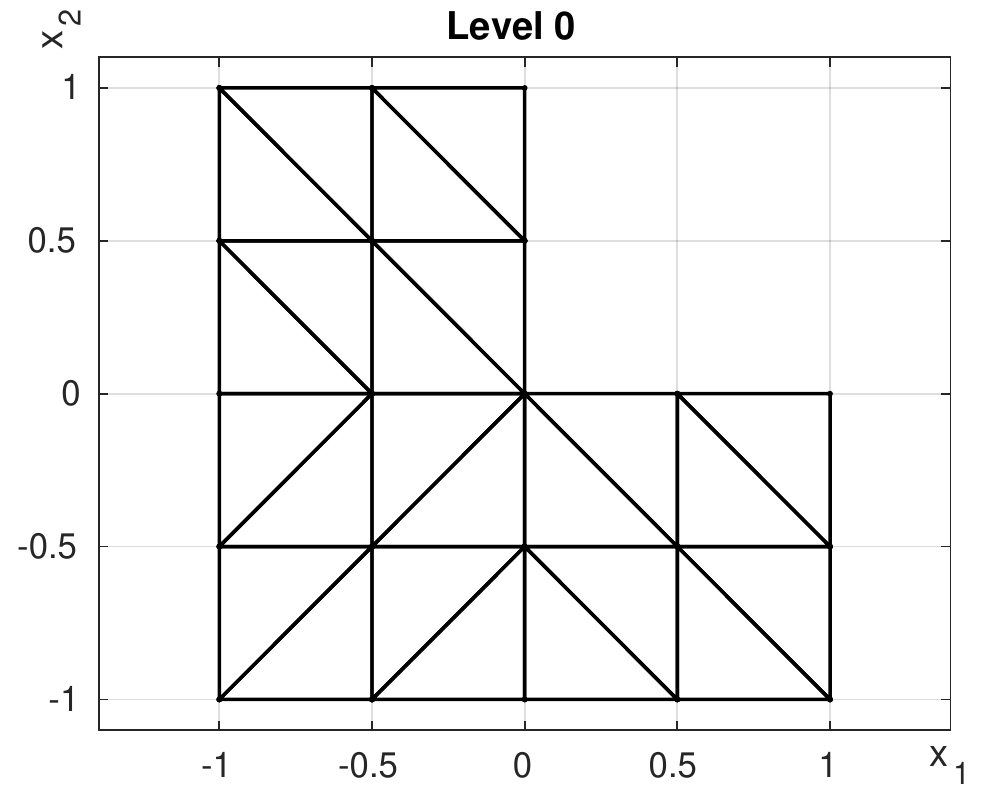}
    \includegraphics[scale=0.5]{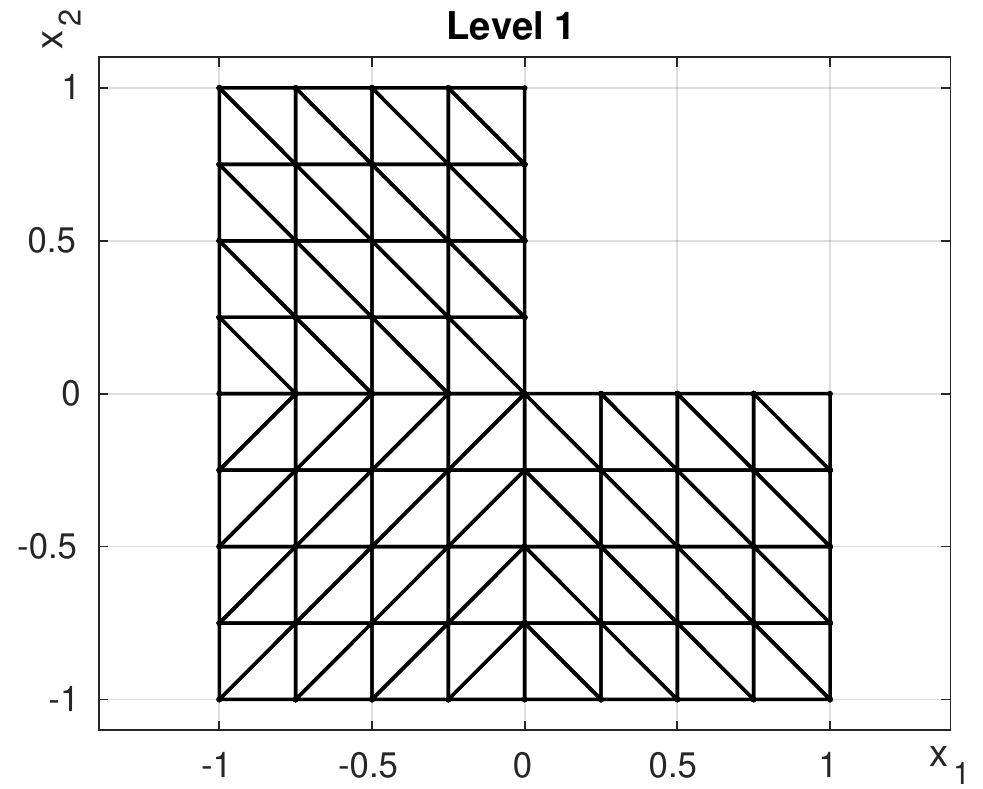}
  \caption{Spatial meshes with uniform refinement strategy: starting mesh and mesh after one refinement step.}
  \label{Num:Fig:LShapeUniform}
\end{center}
\end{figure}
As pointed out in Section~\ref{Sec:P1:d2}, 
spatial meshes with corner-refinements towards the origin are needed to ensure 
second-order convergence in $L^2(\domain )$ for $\IP^1$-FEM approximations in $\domain$. 
For a given maximal mesh width $h_x>0$, we construct
spatial meshes $\cT^N_\beta$ with corner-refinements towards the origin 
fulfilling the grading condition
\be \label{Num:CondGradedMeshes}
    \forall \omega \in \cT^N_\beta \colon \quad h_{x,\omega} \sim \begin{cases}
     h_x^{1/\beta}, & \mathrm{dist}(\omega, \boldsymbol 0) = 0, \\
     h_x \cdot \mathrm{dist}(\omega, \boldsymbol 0)^{1-\beta}, & 0 < \mathrm{dist}(\omega, \boldsymbol 0) \leq R,\\
     h_x, & \mathrm{dist}(\omega, \boldsymbol 0) > R, \end{cases}
\ee
where the mesh grading parameters $\beta \in (0,1]$ and $R > 0$ are fixed. 
Here, $h_{x,\omega}$ is the spatial mesh width of the triangle $\omega \in \cT^N_\beta$, 
and $\mathrm{dist}(\omega, \boldsymbol 0)$ is the distance of the triangle $\omega \in \cT^N_\beta$ 
from the origin $\boldsymbol 0$. 
To get a sequence of these graded spatial meshes, we halve the maximal mesh width $h_x$ 
and use the newest vertex bisection for the refinement, see Remark~\ref{rmk:BisTree}.
Figure~\ref{Num:Fig:LShapeGraded} shows the spatial graded meshes for
the first four levels of refinement with mesh 
grading parameters $\beta = 0.6$ and $R=0.25$, 
which are used in the remainder of this section.
\begin{figure}
\begin{center}
    \includegraphics[scale=0.5]{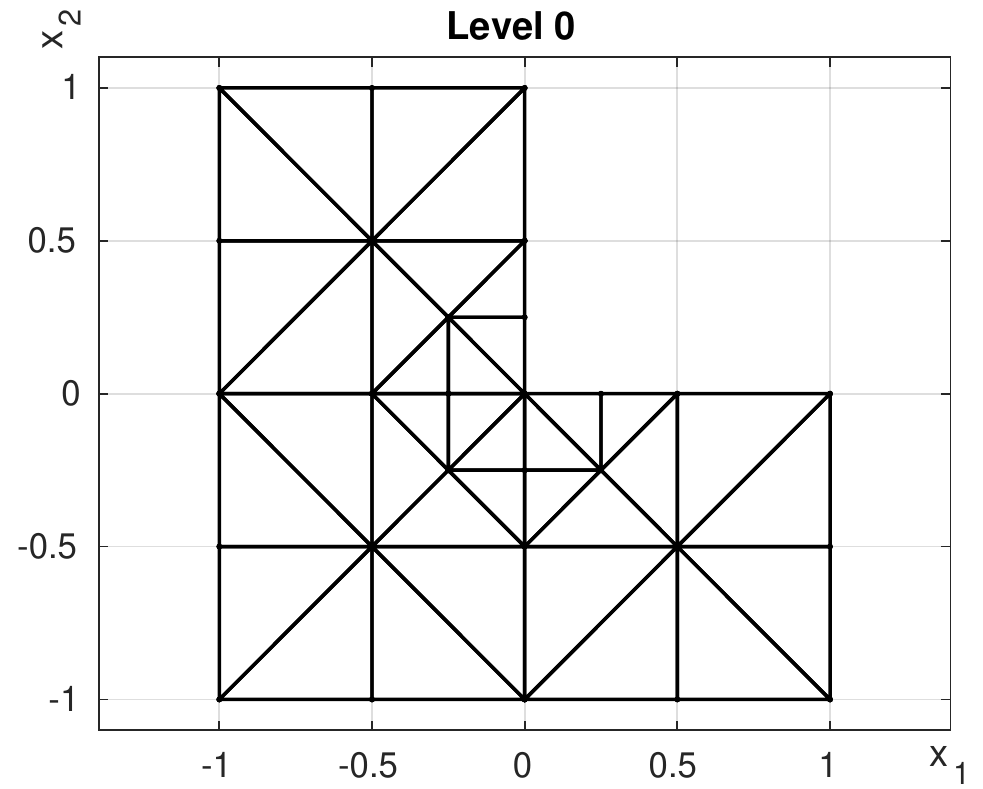}
    \includegraphics[scale=0.5]{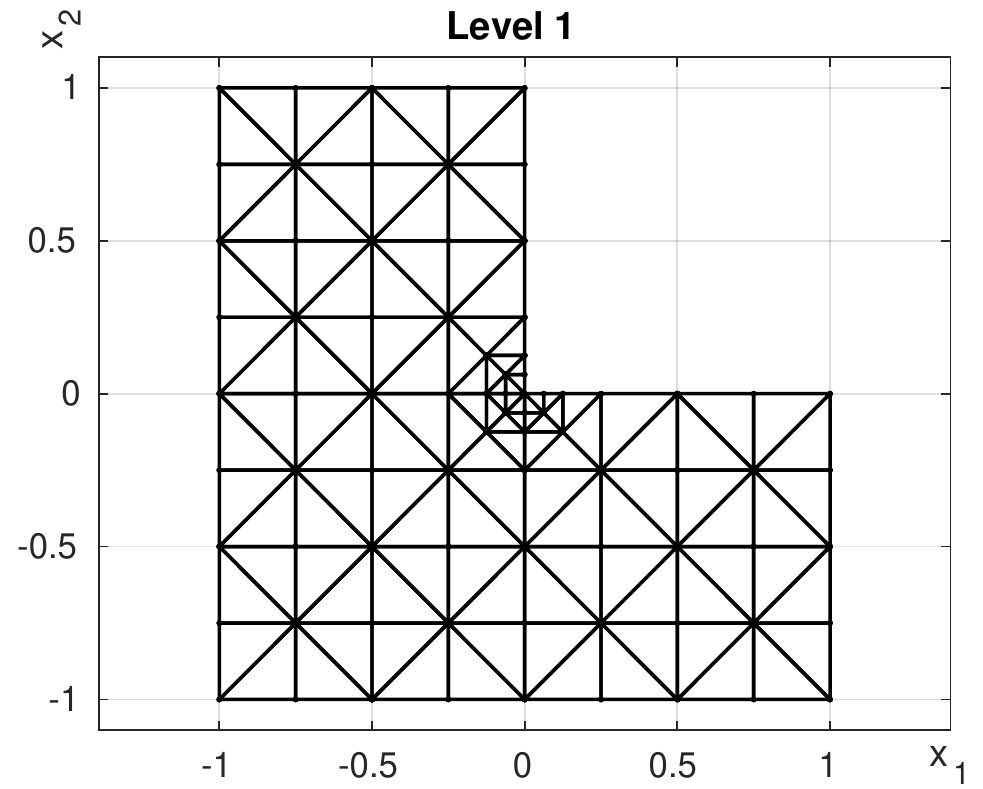}
    \includegraphics[scale=0.5]{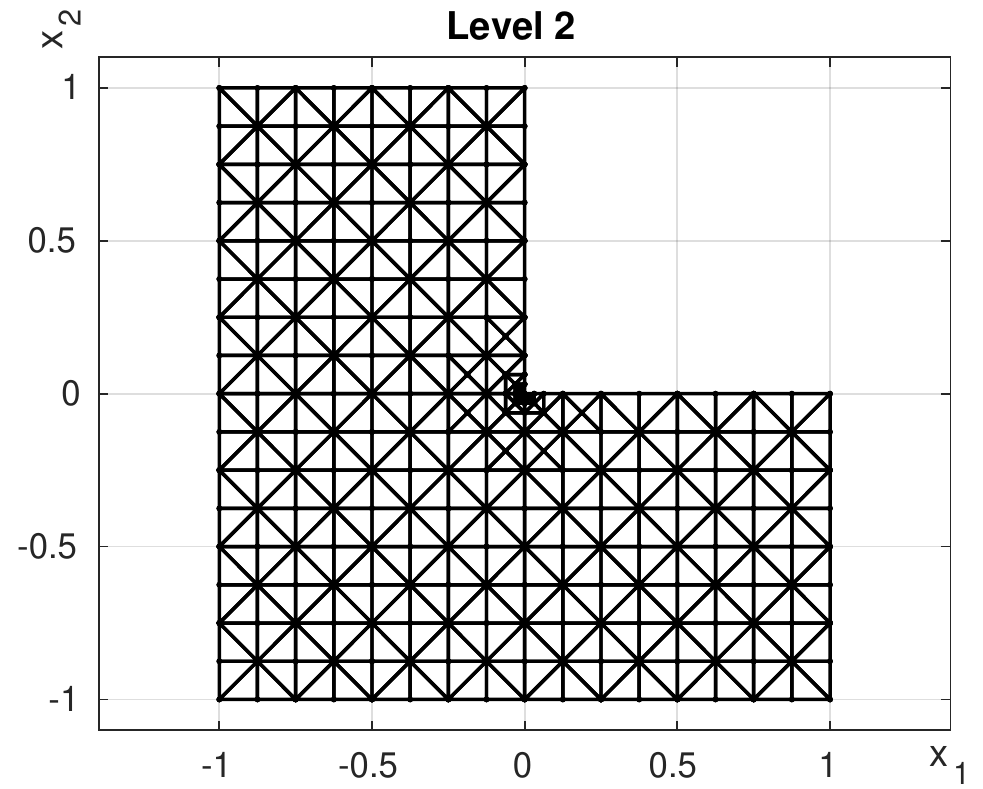}
    \includegraphics[scale=0.5]{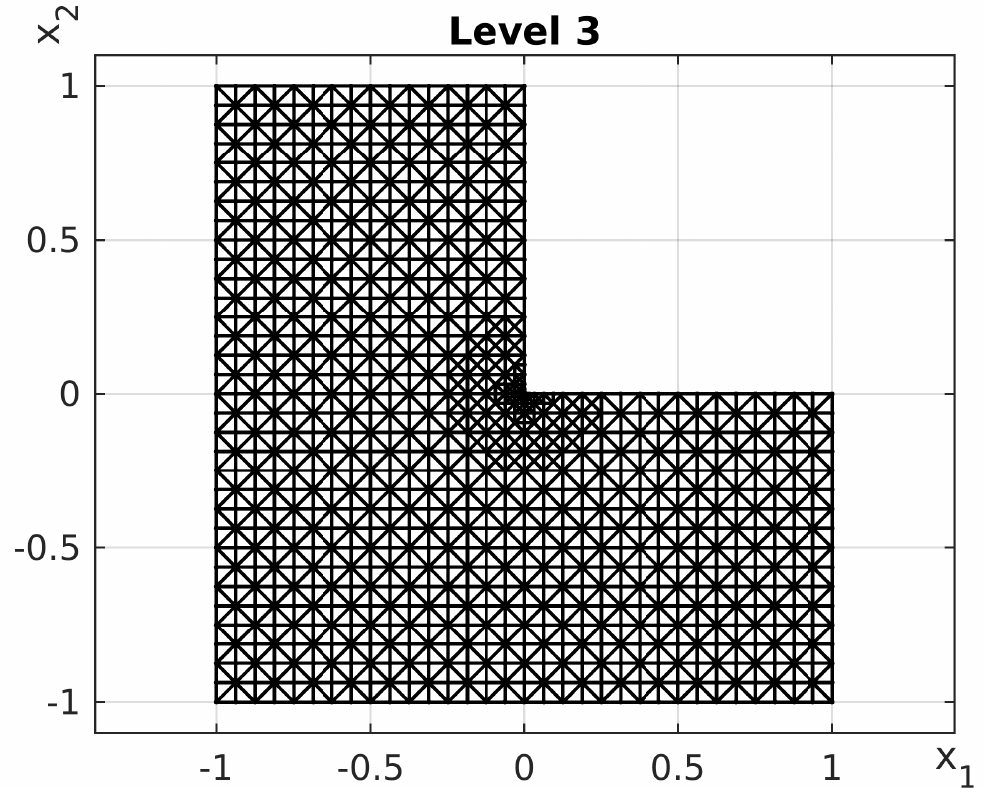}
  \caption{Spatial meshes with corner-refinements towards the origin 
           fulfilling the grading condition \eqref{Num:CondGradedMeshes} with parameters $\beta=0.6$ and $R=0.25$.}
  \label{Num:Fig:LShapeGraded}
\end{center}
\end{figure}

%%%%%%%%%%%%%%%%%%%%%%%%%%%%%%%%%%%%%%%%%%%%%%%%%%%%%%%%%%%%%%%%%%%%%%%
\subsubsection{Spatially Singular Solution}
\label{sec:2dSingSol}
%%%%%%%%%%%%%%%%%%%%%%%%%%%%%%%%%%%%%%%%%%%%%%%%%%%%%%%%%%%%%%%%%%%%%%%
We consider the manufactured solution
\be \label{Num:u2}
    u_2(t,x_1,x_2) 
    = 
    u_{\mathrm{reg}}(t,x_1,x_2) + t \mathrm{e}^{-t} \eta(x_1,x_2) \cdot r(x_1,x_2)^{2/3} 
    \cdot \sin \left(\frac2 3 \left( \arg(x_1,x_2) - \frac{\pi}{2} \right) \right)
\ee
for $(t,x_1,x_2) \in \overline{Q}$ with the smooth part
\be \label{Num:u2reg}
    u_{\mathrm{reg}}(t,x_1,x_2) = \frac{1}{100} t \sin (\pi  x_1) \sin (\pi x_2) \mathrm e^{ -t \left(x_1-\frac{1}{4}\right)^2 - t\left(x_2+\frac{1}{4}\right)^2}, \quad (t,x_1,x_2) \in \overline{Q},
\ee
where $r(x_1,x_2) \in [0,\infty)$ is
the radial coordinate, $\arg(x_1,x_2) \in (0,2\pi]$ is the angular
coordinate, and the cutoff function $\eta
\in C^2(\IR^2)$ is given by
\be \label{Num:Cutoff}
    \eta(x_1,x_2) := \begin{cases}
                        1, & r(x_1,x_2) \leq 1/4, \\
                         \frac{27}{8} -\frac{135}{4} r(x_1,x_2) +180 r(x_1,x_2)^2 \\
                         \quad -440 r(x_1,x_2)^3+480 r(x_1,x_2)^4 \\
                         \quad -192 r(x_1,x_2)^5, & 1/4 < r(x_1,x_2) \leq 3/4, \\
                        0, & 3/4 < r(x_1,x_2).
                     \end{cases}
\ee
Note that the solution $u_2$ is smooth in time but has a corner
singularity in space, which leads to reduced convergence rates,
when the spatial meshes are
refined uniformly. Hence, we use the spatial graded meshes as in
Figure~\ref{Num:Fig:LShapeGraded} in order to recover maximal
convergence rates. We point out that, in numerical tests not reported here, we have verified that, 
for a Poisson problem with a solution of regularity
as the regularity in space of $u_2$ in~\eqref{Num:u2}, one obtains for the $L^2(\domain)$ error
convergence rates $N^{-2/3}\sim h_x^{4/3}$ with uniform meshes, and $N^{-1}\sim h_x^{2}$
with the considered graded meshes.

For the temporal discretizations, we use $\IP^1$-FEM approximation
on uniformly refined meshes, or $p$-FEM for a fixed number $m=4$ of elements.
In connection with the spatial uniform or graded meshes as in
Figure~\ref{Num:Fig:LShapeUniform}, Figure~\ref{Num:Fig:LShapeGraded},
respectively, we investigate four possibilities:
\emph{i)} uniform mesh refinement both in space and in time, 
\emph{ii)} uniform mesh refinement in space and $p$-FEM in time,
\emph{iii)} graded meshes in space and uniform mesh refinement in
time,
\emph{iv)} graded meshes in space and $p$-FEM in time.
For all four cases, the numerical results for the space-time
Galerkin approximation \eqref{eq:IBVPVarxtPerturbed} of the solution
$u_2$ are reported in Figure~\ref{Num:Fig:u2}.
For a given spatial discretization parameter $N$ and $m = 4$ temporal elements, 
we choose the temporal polynomial degrees $\boldsymbol{p} = (p,p,p,p)$ with $p=\lfloor \frac{\ln N}{2} \rfloor$. 
This choice of the discretization parameters balances the terms of the error bound \eqref{pxtErrBd}. 
Hence, the total number of degrees of freedom $MN$ behaves like in
\eqref{ST:dofsBehaviorpFEM}. The numerical results in
Figure~\ref{Num:Fig:u2} confirm Remark~\ref{remk:pFEM}.

\begin{figure}
\begin{center}
    \begin{tikzpicture}[font=\footnotesize,scale=0.9]
			\begin{axis}[xlabel={degrees of freedom $MN$} ,ylabel={$[u_2 - u_2^{MN}]_{H^{1/2}_{0,}(J;L^2(\domain ))}$} ,ymode=log,xmode=log,
				legend entries={{uniform in $xt$}, {uniform in $x$, \\ $p$-FEM in $t$}, {graded in $x$, \\ uniform in $t$}, {graded in $x$, \\ $p$-FEM in $t$}, $(MN)^{-\frac 1 2} \sim h_{xt}^{\frac 3 2}$, $(MN)^{-\frac 2 3} \sim h_{xt}^2$},
                legend pos = {south west},
				legend style={cells={align=left}},
				mark size=3pt]
				\addplot[color=magenta,mark=star, mark options={scale=0.8},line width=1pt] plot file {Numerics/2du2/huniform/heatHilbertTriCylTikzH12.tex};
				\addplot[color=blue,mark=x, mark options={scale=0.8},line width=1pt] plot file {Numerics/2du2/puniform/heatHilbertTriCylTikzH12.tex};
                \addplot[color=orange,mark=square, mark options={scale=0.8},line width=1pt] plot file {Numerics/2du2/hgradedNVBmu06/heatHilbertTriCylTikzH12.tex};
                \addplot[color=red,mark=o, mark options={scale=0.6},line width=1pt] plot file {Numerics/2du2/pgradedNVBmu06/heatHilbertTriCylTikzH12.tex};
                \addplot[color=black] plot[samples=200,domain=100000:803210240] {0.35/(x^(1/2))};
                \addplot[color=black, dashed] plot[samples=200,domain=2000000:707825152] {0.3/(x^(2/3))};
			\end{axis}
    \end{tikzpicture}
    \caption{
    Numerical results with the space-time Galerkin
      approximation \eqref{eq:IBVPVarxtPerturbed} for the 2D
        example with the singular-in-space solution
      $u_2$ in \eqref{Num:u2},
      for all combinations of
        uniform mesh refinement or graded meshes in space 
        (with grading parameter $\beta=0.6$),
        and $\IP^1$-FEM with uniform mesh refinement or $p$-FEM in
        time.
       For the $p$-FEM in time, for a spatial discretization of parameter $N$,
       we use a fixed mesh with $m = 4$ elements and polynomial
        degrees $\boldsymbol{p} = (p,p,p,p)$ with $p=\lfloor \frac{\ln N}{2} \rfloor$.
    }
\label{Num:Fig:u2}    
\end{center}
\end{figure}
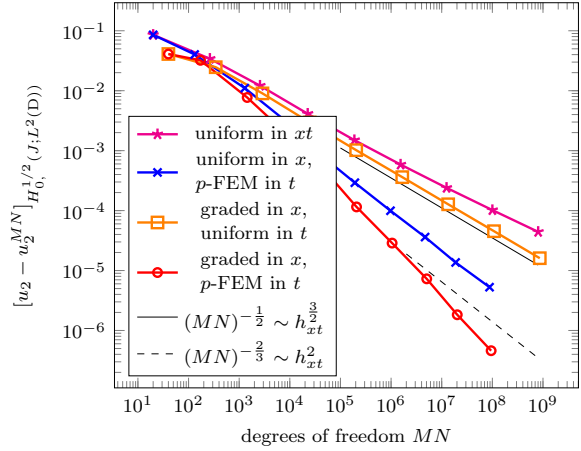

%%%%%%%%%%%%%%%%%%%%%%%%%%%%%%%%%%%%%%%%%%%%%%%%%%%%%%%%%%%%%%%%%%%%%%%
\subsubsection{Singular Solution}
%%%%%%%%%%%%%%%%%%%%%%%%%%%%%%%%%%%%%%%%%%%%%%%%%%%%%%%%%%%%%%%%%%%%%%%
We consider the singular solution
\begin{multline} \label{Num:u3}
    u_3(t,x_1,x_2) = u_{\mathrm{reg}}(t,x_1,x_2) \\
        + t^{3/5} \mathrm{e}^{-t} \eta(x_1,x_2) \cdot r(x_1,x_2)^{2/3} \cdot \sin \left(\frac2 3 \left( \arg(x_1,x_2) - \frac{\pi}{2} \right) \right)
\end{multline}
for $(t,x_1,x_2) \in \overline{Q}$ with the smooth part
$u_{\mathrm{reg}}$ in \eqref{Num:u2reg}, the radial coordinate
$r(x_1,x_2) \in [0,\infty)$, the angular coordinate $\arg(x_1,x_2) \in
(0,2\pi]$, and the cutoff function $\eta \in C^2(\IR^2)$ in
\eqref{Num:Cutoff}. This solution has a temporal singularity at $t=0$.
We observe that the corresponding
right-hand side $g_3$ does not fulfill the temporal analytic
regularity \eqref{eq:TimeReg}. On the other hand, solutions with a
singular behavior as $u_3$ are possible even for sources $g$, which
satisfy the condition~\eqref{eq:TimeReg}. 
As closed-form representations of such singular solutions
do not seem to be available, we perform our numerical tests
with the manufactured solution $u_3$ in \eqref{Num:u3}. 
Furthermore, the solution $u_3$ has the same spatial singularity as
$u_2$ in \eqref{Num:u2}.
Thus, in order to get the full convergence rates, we use the graded meshes in Figure~\ref{Num:Fig:LShapeGraded} 
for the spatial discretization, and a temporal $hp$-FEM. 
We investigate four possibilities:
\emph{i)} uniform mesh refinement both in space and in time, 
\emph{ii)} uniform mesh refinement in space and $hp$-FEM in time,
\emph{iii)} graded meshes in space and uniform mesh refinement in
time,
\emph{iv)} graded meshes in space and $hp$-FEM in time.
For all four cases, the numerical results for the space-time
Galerkin approximation \eqref{eq:IBVPVarxtPerturbed} of the solution
$u_3$ are reported in Figure~\ref{Num:Fig:u3}.
For a given spatial discretization parameter $N$, we choose the temporal mesh as in \eqref{def:tj} with
grading parameter $\sigma = 0.17$, slope parameter $\muhp = 1.0$, numbers of elements $m_1 = \lfloor 2.2 \cdot \ln N \rfloor$, $m_2 = 1$, and temporal polynomial degrees $\boldsymbol{p} \in \IN^m$ as in \eqref{def:pj}.
This choice of the discretization parameters balances the terms of the error bound~\eqref{eq:hpxtErrBd}. 
Hence, the total number of degrees of freedom $MN$ behaves like in \eqref{ST:dofsBehavior}. 
The numerical results in Figure~\ref{Num:Fig:u3} are in accordance with Theorem~\ref{thm:Conv},
when the temporal analytic regularity condition \eqref{eq:TimeReg}, and hence, the conditions on parameters $\muhp$, $m_2$, i.e.,  \eqref{muhp}, \eqref{m2}, are ignored.

\begin{figure}
\begin{center}
    \begin{tikzpicture}[font=\footnotesize,scale=0.9]
			\begin{axis}[xlabel={degrees of freedom $MN$} ,ylabel={$[u_3 - u_3^{MN}]_{H^{1/2}_{0,}(J;L^2(\domain ))}$} ,ymode=log,xmode=log,
				legend entries={{uniform in $xt$}, {uniform in $x$, \\ $hp$-FEM in $t$}, {graded in $x$, \\
				uniform in $t$}, {graded in $x$, \\ $hp$-FEM in $t$}, $(MN)^{-\frac 1 5} \sim h_{xt}^{0.6}$, $(MN)^{-\frac 2 3} \sim h_{xt}^2$},
				legend pos = {south west},
				legend style={cells={align=left}},
				mark size=3pt]
				\addplot[color=purple,mark=otimes, mark options={scale=0.8},line width=1pt] plot file {Numerics/2du3/huniform/heatHilbertTriCylTikzH12.tex};
				\addplot[color=blue,mark=x, mark options={scale=0.8},line width=1pt] plot file {Numerics/2du3/hpuniform/heatHilbertTriCylTikzH12.tex};
                \addplot[color=orange,mark=square, mark options={scale=0.8},line width=1pt] plot file {Numerics/2du3/hgradedNVBmu06/heatHilbertTriCylTikzH12.tex};
                \addplot[color=red,mark=o, mark options={scale=0.6},line width=1pt] plot file {Numerics/2du3/hpgradedNVBmu06/heatHilbertTriCylTikzH12.tex};
                \addplot[color=black] plot[samples=200,domain=300000:803210240] {0.07/(x^(1/5))};
                \addplot[color=black, dashed] plot[samples=200,domain=10000000:1003210240] {35/(x^(2/3))};
			\end{axis}
    \end{tikzpicture}
    \caption{
    Numerical results with the space-time Galerkin
      approximation \eqref{eq:IBVPVarxtPerturbed} for the 2D
        example with the singular solution
      $u_3$ in \eqref{Num:u3},
      for all combinations of
        uniform mesh refinement or graded meshes in space (with
        grading parameter $\beta=0.6$),
        and $\IP^1$-FEM with uniform mesh refinement or $hp$-FEM in
        time.
       For the $hp$-FEM in time, for a spatial discretization of parameter $N$,
       we use a geometric temporal mesh with subdivision
       % to distinguish from the grading parameter of the radical meshes in space
       parameter $\sigma = 0.17$, slope parameter $\muhp = 1.0$, 
       numbers of elements $m_1 = \lfloor 2.2 \cdot \ln N \rfloor$, $m_2 = 1$, 
       and temporal polynomial degrees $\boldsymbol{p} \in \IN^m$ as in \eqref{def:pj}.
        }
\label{Num:Fig:u3}
\end{center}
\end{figure}

%

%%%%%%%%%%%%%%%%%%%%%%%%%%%%%%%%%%%%%%%%%%%%%%%%%%%%%%%%%%%%%%%%%%%%%%%%%%%%%%%
\section{Conclusion}
\label{sec:Concl}
%%%%%%%%%%%%%%%%%%%%%%%%%%%%%%%%%%%%%%%%%%%%%%%%%%%%%%%%%%%%%%%%%%%%%%%%%%%%%%
Based on a variational space-time formulation of the IBVP \eqref{eq:IBVP}--\eqref{eq:BC},
we analyzed tensorized discretization consisting of an exponentially convergent
time-discretization of $hp$-type, combined with a first-order Lagrangian FEM in the 
spatial domain, with corner-mesh refinement to account for the
presence of spatial singularities.
Stability of the considered discretization scheme is 
achieved by Hilbert-transforming the temporal $hp$-trial spaces. 
Details on the efficient, exponentially accurate,
numerical realization of this transformation were presented.
Several numerical examples in space dimension $d=2$ in nonconvex polygonal domains
confirmed the asymptotic error bounds.
In effect, the overall number of degrees of freedom scales essentially 
as those for one instance of the spatial problem.

The presented proof of time-analyticity via eigenfunction expansions is limited to self-adjoint,
elliptic spatial differential operators.
Nonselfadjoint spatial operators which are $t$-independent 
allow similar analytic regularity results 
via semigroup theory (see, e.g., \cite{DSChS2001}).

The adopted space-time formulation and its operator perspective 
and the error analysis extend \emph{verbatim} to self-adjoint, 
elliptic spatial operators of positive order.
Also, certain nonlinear evolution equations allow for corresponding formulations
(see, e.g., \cite{ScSt17}).
Moreover, transmission problems with piecewise Lipschitz coefficients in the spatial
operators can be covered (with the local mesh refinement also at multi-material
interface points).

The present error analysis with the same convergence rates is readily
extended to a coefficient in the temporal derivative that is time-dependent and analytic in~$[0,T]$.

We finally remark that the presently adopted space-time variational formulation 
will also allow for \emph{a~posteriori} time-discretization error estimation, 
which is reliable and robust uniformly with respect to $p$. 
Details shall be developed elsewhere.

%%%%%%%%%%%%%%%%%%%%%%%%%%%%%%%%%%%%%%%%%%%%%%%%%%%%%%%%%%%%%%%%%%%%%%%%%%%%%
\bibliographystyle{amsplain}
\bibliography{xthp}{}

%%%%%%%%%%%%%%%%%%%%%%%%%%%%%%%%%%%%%%%%%%%%%%%%%%%%%%%%%%%%%%%%%%%%%%%%%
\appendix
%%%%%%%%%%%%%%%%%%%%%%%%%%%%%%%%%%%%%%%%%%%%%%%%%%%%%%%%%%%%%%%%%%%%%%%%%
%
%%%%%%%%%%%%%%%%%%%%%%%%%%%%%%%%%%%%%%%%%%%%%%%%%%%%%%%%%%%%%%%%%%%%%%%%%
\section{Some Properties of $H^{1/2}_{0,}(a,b)$} \label{sec:NormEquivalence}
%%%%%%%%%%%%%%%%%%%%%%%%%%%%%%%%%%%%%%%%%%%%%%%%%%%%%%%%%%%%%%%%%%%%%%%%%
In this appendix, 
we provide proofs of Lemma~\ref{lem:NormEquivalence} and Lemma~\ref{lem:FractionalNormPointtau}
of Section~\ref{sec:FctSpc} concerning the Sobolev space $H^{1/2}_{0,}(a,b)$, and
state Poincar\'e and interpolation inequalities in $H^{1/2}_{0,}(a,b)$.

This result of Lemma~\ref{lem:NormEquivalence} is well-known, but 
we need to make explicit the dependency of the involved constants on the interval $(a,b)$, 
which is essential for the derivation of the temporal $hp$-error estimates in Section~\ref{sec:ConvRate}. 
For simplicity, we restrict to the case of real-valued functions $v \colon \, (a,b) \to \IR$.
All results and proofs can be generalized straightforwardly 
to $X$-valued functions $v \colon \, (a,b) \to X$ for a Hilbert space $X$. 
We introduce the following notation. 
For the classical Sobolev space
\begin{equation*}
    H^{1/2}(\IR) = (H^1(\IR),  L^2(\IR))_{1/2,2},
\end{equation*}
where $H^1(\IR)$ is equipped with the norm 
$\| \circ \|_{H^1(\IR)} = ( \| \circ \|_{L^2(\IR)}^2 + \| \partial_t \circ \|_{L^2(\IR)}^2 )^{1/2}$, 
we consider the interpolation norm $\| \circ \|_{H^{1/2}(\IR)}$ and the Slobodetskii norm
\[
\normiii{v}_{H^{1/2}(\IR)} 
:= 
\left(
 \| v \|_{L^2(\IR)}^2
+ 
 |v|_{H^{1/2}(\IR)}^2 
 \right)^{1/2} 
\]
for $v \in H^{1/2}(\IR)$, with
\[
|v|_{H^{1/2}(\IR)} 
:= 
\left( \int_{-\infty}^{\infty} \int_{-\infty}^{\infty} \frac{|v(s)-v(t)|^2}{|s-t|^2} \mathrm ds \mathrm dt \right)^{1/2}.
\]

\begin{proof}[Proof of Lemma~\ref{lem:NormEquivalence}]
The equivalence of norms is proven in, e.g., \cite{McLean2000}. 
We give more details about the norm equivalence constants. 
For this purpose, 
we introduce an extension operator and establish bounds on its norm.
Define
$\mathcal E_1 \colon \, H^1_{0,}(a,b) \to H^1(\IR)$,
    \begin{equation*}
        \mathcal E_1 v(t) := \begin{cases}
                                v(t),       & t \in [a,b], \\
                                v(2b-t),    & t \in (b,2b-a], \\
                                0,          & \text{otherwise}
                           \end{cases}
    \end{equation*}
    for $v \in H^1_{0,}(a,b)$. 
The mapping $\mathcal E_0 \colon \, L^2(a,b) \to L^2(\IR)$ 
is defined for $v \in L^2(a,b)$ 
as
    \begin{equation*}
        \mathcal E_0 v(t) := \begin{cases}
                                v(t),       & t \in (a,b), \\
                                v(2b-t),    & t \in (b,2b-a), \\
                                0.          & \text{otherwise}
                           \end{cases}
    \end{equation*}
Evidently, $\mathcal E_1 v = \mathcal E_0 v$ for $v \in H^1_{0,}(a,b)$. 
Next, for $v \in L^2(a,b)$, 
    \begin{equation*}
        \norm{\mathcal E_0 v}_{L^2(\IR)}^2 
        = 
        \int_a^b \abs{v(t)}^2 \mathrm dt +  \int_b^{2b-a} \abs{v(2b-t)}^2 \mathrm dt = 2 \norm{v}_{L^2(a,b)}^2
    \end{equation*}
    and, for $v \in H^1_{0,}(a,b)$, 
    \begin{equation*}
        \norm{\partial_t \mathcal E_1 v}_{L^2(\IR)}^2 
      = \int_a^b \abs{\partial_t v(t)}^2 \mathrm dt +  \int_b^{2b-a} \abs{\partial_t v(2b-t)}^2 \mathrm dt 
      = 2 \norm{\partial_t v}_{L^2(a,b)}^2 \;.
    \end{equation*}
    Hence, for $v \in H^1_{0,}(a,b)$, it holds true that
    \begin{equation*}
         \norm{\mathcal E_1 v}_{H^1(\IR)}^2 
         = 2 \norm{v}_{H^1(a,b)}^2  
        \leq 2 \left( 1 + \frac{4(b-a)^2}{\pi^2} \right) \norm{\partial_t v}_{L^2(a,b)}^2,
    \end{equation*}
    where the Poincar\'e inequality (see Lemma~\ref{lem:Poincare} below) 
    is used in the last step. 
    Interpolation yields an operator 
    $\mathcal E_{1/2} \colon \, H^{1/2}_{0,}(a,b) \to H^{1/2}(\IR)$ 
    with $\mathcal E_{1/2} v = \mathcal E_0 v$ for $v \in H^{1/2}_{0,}(a,b)$ 
    and
    \begin{equation} \label{NormEquivalenceExtension}
        \forall v \in H^{1/2}_{0,}(a,b) : \quad  \norm{\mathcal E_{1/2}v}_{H^{1/2}(\IR)}^2 \leq 2 \sqrt{ 1 + \frac{4(b-a)^2}{\pi^2}} \norm{v}_{H^{1/2}_{0,}(a,b)}^2.
    \end{equation}
    Next, we estimate $\normiii{\mathcal E_{1/2}v}_{H^{1/2}(\IR)}$ for $v \in H^{1/2}_{0,}(a,b)$. For this purpose, we compute
    \begin{align}
        \abs{\mathcal E_{1/2} v}_{H^{1/2}(\IR)}^2 &=  \int_a^\infty \int_a^\infty()  +2 \int_a^\infty \int_{-\infty}^a () + \int_{-\infty}^a \int_{-\infty}^a () \nonumber \\
        &= \abs{(\mathcal E_{1/2} v)_{|(a,\infty)}}_{H^{1/2}(a,\infty)}^2 + 2 \int_a^\infty \int_{-\infty}^a \frac{|\mathcal E_{1/2}v(t)|^2}{|s-t|^{2}} \mathrm ds  \mathrm dt + 0 \nonumber \\
        &=\abs{(\mathcal E_{1/2} v)_{|(a,\infty)}}_{H^{1/2}(a,\infty)}^2 + 2 \int_a^\infty \frac{\abs{\mathcal E_{1/2}v(t)}^2}{t-a}  \mathrm dt \label{NormEquivalenceRepresentationZero}
    \end{align}
    for $v \in H^{1/2}_{0,}(a,b)$, where the seminorm $|\circ|_{H^{1/2}(a,\infty)}$ 
    is defined by \eqref{SlobodetskiiSemi} with $b=\infty$.
    The integral in the bound 
    \eqref{NormEquivalenceRepresentationZero} 
    is finite due to $v \in H^{1/2}_{0,}(a,b)$, cf. \eqref{Sob:NormTriple}. 
    Thus, we get
    \begin{multline*}
     \normiii{\mathcal E_{1/2}v}_{H^{1/2}(\IR)}^2 
     = 
     2 \norm{v}_{L^2(a,b)}^2 + \abs{(\mathcal E_{1/2} v)_{|(a,\infty)}}_{H^{1/2}(a,\infty)}^2 
   + 2 \int_a^\infty \frac{\abs{\mathcal E_{1/2}v(t)}^2}{t-a}  \mathrm dt 
     \\
     = 2 \norm{v}_{L^2(a,b)}^2 + \abs{v}_{H^{1/2}(a,b)}^2 +  2 \int_b^\infty \int_a^b () + \int_b^\infty \int_b^\infty() + 2 \int_a^\infty \frac{\abs{\mathcal E_{1/2}v(t)}^2}{t-a} \mathrm dt.
    \end{multline*}
    The third term on the right side is bounded by
    \begin{align*}
        2 \int_b^\infty \int_a^b () &= 2 \int_b^{2b-a} \int_a^b \frac{|v(s) - v(2b-t)|^2}{|s-t|^{2}} \mathrm ds \mathrm dt + 2\int_{2b-a}^\infty \int_a^b \frac{|v(s)|^2}{|s-t|^{2}} \mathrm ds \mathrm dt \\
        &= 2 \int_a^b \int_a^b \frac{|v(s) - v(t)|^2}{\underbrace{ |2b-s-t|^{2}}_{\geq |s-t|^2}} \mathrm ds \mathrm dt + 2\int_a^b \frac{|v(s)|^2}{\underbrace{2b-a-s}_{\geq s-a}} \mathrm ds \\
        &\leq 2 \abs{v}_{H^{1/2}(a,b)}^2 + 2\int_a^b \frac{|v(s)|^2}{s-a} \mathrm ds,
    \end{align*}
    the fourth term is
    \begin{align*}
        &\int_b^\infty \int_b^\infty() = \int_b^{2b-a} \int_b^{2b-a}() + 2\int_{2b-a}^\infty \int_b^{2b-a}() +\int_{2b-a}^\infty \int_{2b-a}^\infty() \\
        &= \int_b^{2b-a} \int_b^{2b-a} \frac{|v(2b-s) - v(2b-t)|^2}{|s-t|^{2}} \mathrm ds \mathrm dt + 2 \int_{2b-a}^\infty \int_b^{2b-a} \frac{|v(2b-s)|^2}{|s-t|^{2}} \mathrm ds \mathrm dt + 0\\
        &= \abs{v}_{H^{1/2}(a,b)}^2 + 2 \int_b^{2b-a} \frac{|v(2b-s)|^2}{2b-a-s} \mathrm ds = \abs{v}_{H^{1/2}(a,b)}^2 +  2\int_a^b \frac{|v(s)|^2}{s-a} \mathrm ds,
    \end{align*}
    whereas for the fifth term, we have 
    \begin{align*}
        2 \int_a^\infty \frac{\abs{\mathcal E_{1/2}v(t)}^2}{t-a} \mathrm dt &= 2 \int_a^b \frac{\abs{v(t)}^2}{t-a} \mathrm dt + 2 \int_b^{2b-a} \frac{\abs{v(2b-t)}^2}{t-a} \mathrm dt \\
        &= 2 \int_a^b \frac{\abs{v(t)}^2}{t-a} \mathrm dt + 2 \int_a^b \frac{\abs{v(t)}^2}{\underbrace{2b-a-t}_{\geq t-a}} \mathrm dt  \leq 4\int_a^b \frac{\abs{v(t)}^2}{t-a} \mathrm dt.
    \end{align*}
    Using the above estimates gives for all $v \in H^{1/2}_{0,}(a,b)$
    \begin{align*}
         \normiii{\mathcal E_{1/2}v}_{H^{1/2}(\IR)}^2 \leq 2 \norm{v}_{L^2(a,b)}^2 + 4\abs{v}_{H^{1/2}(a,b)}^2 + 8\int_a^b \frac{\abs{v(t)}^2}{t-a} \mathrm dt \leq 8 \normiii{v}_{H^{1/2}_{0,}(a,b)}^2\;.
    \end{align*}
    With these properties, we have for all $v \in H^{1/2}_{0,}(a,b)$ the lower bound in the norm equivalence: 
    \begin{align*}
        \norm{v}_{H^{1/2}_{0,}(a,b)} \leq \norm{\mathcal E_{1/2}v}_{H^{1/2}(\IR)} \leq \frac{1}{C_{\IR,1}} \normiii{\mathcal E_{1/2}v}_{H^{1/2}(\IR)} \leq \frac{2\sqrt{2} }{C_{\IR,1}} \normiii{v}_{H^{1/2}_{0,}(a,b)}\;.
    \end{align*}
    Here, the first inequality is proven by interpolation, the second estimate follows from
    \begin{equation} \label{NormEquivalenceR}
        \forall z \in H^{1/2}(\IR) : \quad C_{\IR,1} \norm{z}_{H^{1/2}(\IR)} \leq  \normiii{z}_{H^{1/2}(\IR)} \leq  C_{\IR,2} \norm{z}_{H^{1/2}(\IR)} 
    \end{equation}
    with constants $C_{\IR,1}$, $C_{\IR,2} > 0$, see \cite[Theorem~B.7]{McLean2000} and \cite[Lemma~4.1]{Eskin1981}.

    For the upper bound, relations 
    \eqref{NormEquivalenceRepresentationZero}, \eqref{NormEquivalenceR} and \eqref{NormEquivalenceExtension} 
    yield
    \begin{align*}
        \normiii{v}_{H^{1/2}_{0,}(a,b)}^2 &\leq \norm{\mathcal E_{1/2}v}_{L^2(\IR)}^2 +\abs{(\mathcal E_{1/2} v)_{|(a,\infty)}}_{H^{1/2}(a,\infty)}^2 + \int_a^\infty \frac{\abs{\mathcal E_{1/2}v(t)}^2}{t-a}  \mathrm dt \\
        &= \norm{\mathcal E_{1/2}v}_{L^2(\IR)}^2 + \frac 12 \abs{(\mathcal E_{1/2} v)_{|(a,\infty)}}_{H^{1/2}(a,\infty)}^2 +  \frac 12 \abs{\mathcal E_{1/2} v}_{H^{1/2}(\IR)}^2 \\
        &\leq  \normiii{\mathcal E_{1/2}v}_{H^{1/2}(\IR)}^2 \\
        &\leq (C_{\IR,2})^2 \norm{\mathcal E_{1/2}v}_{H^{1/2}(\IR)}^2  \leq (C_{\IR,2})^2 2 \sqrt{ 1 + \frac{4(b-a)^2}{\pi^2}} \norm{v}_{H^{1/2}_{0,}(a,b)}^2,
    \end{align*}
    i.e., the assertion is proven.
\end{proof}
The following proof of Lemma~\ref{lem:FractionalNormPointtau} 
restricts the argument of~\cite{Faermann2000} to our particular case.
\begin{proof}[Proof of Lemma~\ref{lem:FractionalNormPointtau}]
 Let $v \in H^{1/2}(a,b)$ and $\tau \in (a,b$) be given. Then, we
 split the integral in the definition~\eqref{SlobodetskiiSemi} as follows:
 \begin{align*}
  | v |_{H^{1/2}(a,b)}^2 
&= \int_a^{\tau} \int_a^b \Big( \cdots \Big) \mathrm ds \mathrm dt + \int_{\tau}^b \int_a^b \Big( \cdots \Big) \mathrm ds \mathrm dt \\
  &= \int_a^{\tau} \int_a^{\tau} \Big( \cdots \Big) \mathrm ds \mathrm dt 
   + 2 \int_a^{\tau} \int_{\tau}^b \Big( \cdots \Big) \mathrm ds \mathrm dt + \int_{\tau}^b \int_{\tau}^b \Big( \cdots \Big) \mathrm ds \mathrm dt\\
&=| v |_{H^{1/2}(a,\tau)}^2+2 \int_a^{\tau} \int_{\tau}^b \Big( \cdots \Big) \mathrm ds \mathrm dt+| v |_{H^{1/2}(\tau,b)}^2.
 \end{align*}
For the integral on the right side, we get
 \begin{align*}
     2 \int_a^{\tau} \int_{\tau}^b \frac{|v(s)-v(t)|^2}{|s-t|^{2}} \mathrm ds \mathrm dt 
     &\leq 4 \int_a^{\tau} \int_{\tau}^b \frac{|v(s)|^2}{|s-t|^{2}} \mathrm ds \mathrm dt 
     + 4 \int_a^{\tau} \int_{\tau}^b \frac{|v(t)|^2}{|s-t|^{2}} \mathrm ds \mathrm dt \\
     &= 4  \int_{\tau}^b |v(s)|^2 [ (s-\tau)^{-1} - (s-a)^{-1} ] \mathrm ds \\
     &\quad + 4 \int_a^{\tau} |v(t)|^2 [(\tau-t)^{-1} - (b-t)^{-1}] \mathrm dt  \\
     &\leq 4 \int_{\tau}^b \frac{|v(s)|^2}{s-\tau} \mathrm ds + 4 \int_a^{\tau} \frac{|v(t)|^2}{\tau-t} \mathrm dt.
 \end{align*}
 Thus, the assertion follows.
\end{proof}

\begin{lemma} \label{lem:Poincare}
  For $a,b \in \IR$, $a<b$, the Poincar\'e inequalities
  \begin{align*}
    \forall v \in H^{1/2}_{0,}(a,b) : \quad &\| v \|_{L^2(a,b)} \leq \sqrt{\frac{2 (b-a)}{\pi}} \| v \|_{H^{1/2}_{0,}(a,b)}, \\
    \forall v \in H^1_{0,}(a,b) : \quad &\| v \|_{H^{1/2}_{0,}(a,b)} \leq \sqrt{\frac{2 (b-a)}{\pi}} \| \partial_t v \|_{L^2(a,b)}, \\
    \forall v \in H^1_{0,}(a,b) : \quad &\| v \|_{L^2(a,b)} \leq \frac{2 (b-a)}{\pi} \| \partial_t v \|_{L^2(a,b)}
  \end{align*}
  hold true, where the constants are sharp.
\end{lemma}
\begin{proof}
    By interpolation, we have the Fourier series representations
    \be \label{norm:representations}
        \| v \|_{L^2(a,b)}^2 = \sum_{k=0}^\infty |v_k|^2, \quad  \| v \|_{H^{1/2}_{0,}(a,b)}^2 = \sum_{k=0}^\infty \sqrt{\lambda_k} |v_k|^2, \quad  \| \partial_t v \|_{L^2(a,b)}^2 = \sum_{k=0}^\infty \lambda_k |v_k|^2
    \ee
    with coefficients $v_k$ as in \eqref{eq:FourierRepresentation} and eigenvalues $\lambda_k = \frac{\pi^2 (2k+1)^2}{4(b-a)^2}$ of the eigenvalue problem \eqref{time:eigenvalues}.
    Hence, all Poincar\'e inequalities follow from these representations. The constants are sharp since for $v$ with $v_0 \neq 0$ and $v_k=0$ for $k \in \IN$, equality holds true.
\end{proof}

\begin{lemma} \label{lem:interpolationEstimate}
    For $a,b \in \IR$ with $a<b$, the interpolation estimate
    \begin{equation*}
        \forall v \in H^1_{0,}(a,b) : \, \|v\|_{H^{1/2}_{0,}(a,b)} \leq \sqrt{ \| v \|_{L^2(a,b)} \| \partial_t v \|_{L^2(a,b)} }
    \end{equation*}
    holds true, where $\| \circ \|_{H^{1/2}_{0,}(a,b)}$ denotes the interpolation norm \eqref{eq:H120def}.
\end{lemma}
\begin{proof}
    Using the Cauchy--Schwarz inequality, the assertion follows immediately from the Fourier representations \eqref{norm:representations}.
\end{proof}

%%%%%%%%%%%%%%%%%%%%%%%%%%%%%%%%%%%%%%%%%%%%%%%%%%%%%%%%%%%%%%%%%%%%%%%%%
\section{Proof of Lemma~\ref{lem:H12X2Reg}} \label{sec:ProofLemH12X2Reg}
%%%%%%%%%%%%%%%%%%%%%%%%%%%%%%%%%%%%%%%%%%%%%%%%%%%%%%%%%%%%%%%%%%%%%%%%%

  Let $b \in (0,T]$ be fixed.
  According to~\eqref{eq:T'TBd} for $l=0$, the estimate 
  \begin{equation*}
    \forall t>0 \colon
    \norm{\semigroup(t)}_{\mathcal L(X_\varepsilon, X_2)}^2 \leq
    \frac{1}{\sqrt{2\pi}} \left(\frac{1}{2}\right)^{2-\varepsilon}
    \Gamma(3-\varepsilon)
    t^{-2+\varepsilon}.
  \end{equation*}
  holds true. 
  The logarithmic convexity of the gamma function gives
  $\Gamma(3-\varepsilon)=\Gamma\left(2\varepsilon+3(1-\varepsilon)\right)\le\Gamma(2)^\varepsilon\Gamma(3)^{1-\varepsilon}=
  2^{1-\varepsilon}$ and we obtain
  \begin{equation}\label{eq:Tt}
    \forall t>0 \colon \norm{\semigroup(t)}_{\mathcal L(X_\varepsilon, X_2)} 
    \leq 
    \sqrt{\frac{1}{\sqrt{2\pi}} \left(\frac{1}{2}\right) t^{-2+\varepsilon}} 
    = 
      \frac{1}{\sqrt[4]{8\pi}} \, t^{-1+\varepsilon/2}. 
  \end{equation}
  The solution $u$ admits the representation
  \begin{equation}\label{eq:repr}
      u(t) = \int_0^t \semigroup(\tau) g(t-\tau) \mathrm d\tau, \quad 0 \leq t \leq b,
  \end{equation}
  see~\eqref{eq:Duhamel}, 
  and for $t \in [0,b]$, it follows that
  \begin{align*}
      \norm{u(t)}_{X_2}   &\leq \int_0^t \norm{\semigroup(\tau) g(t-\tau)}_{X_2} \mathrm d\tau 
                          \leq \int_0^t \norm{\semigroup(\tau)}_{\mathcal L(X_\varepsilon,X_2)} \norm{g(t-\tau)}_{X_\varepsilon} \mathrm d\tau \\
                          &\leq \frac{1}{\sqrt[4]{8\pi}} C_g \int_0^t \tau^{-1+\varepsilon/2} \mathrm d\tau 
                          = \frac{1}{\sqrt[4]{8\pi}} C_g\, \frac{2}{\varepsilon}\, t^{\varepsilon/2} 
                          = \sqrt[4]{\frac{2}{\pi}}\, C_g\, \frac{1}{\varepsilon}\, t^{\varepsilon/2}.
  \end{align*}
  We estimate the three terms of $\normiii{u}_{H^{1/2}_{0,}((0,b);X_2)}$
  expressed as in~\eqref{Sob:NormTriple}.

  \noindent
  \textbf{First term:} From the previous bound for $\norm{u(t)}_{X_2}$, we derive
  \begin{equation*}
      \| u \|_{L^2((0,b);X_2)}^2 = \int_0^b \norm{u(t)}_{X_2}^2 \mathrm dt \leq \sqrt{\frac{2}{\pi}}\, C_g^2\, \frac{1}{\varepsilon^2} \int_0^b t^\varepsilon \mathrm  dt = \sqrt{\frac{2}{\pi}}\, C_g^2\, \frac{1}{\varepsilon^2(1+\varepsilon)}\, b^{1+\varepsilon}.
  \end{equation*}
    
  \noindent
  \textbf{Third term:} Similarly, we obtain
  \begin{equation*}
      \int_0^b \frac{\norm{u(t)}_{X_2}^2}{t} \mathrm dt \leq \sqrt{\frac{2}{\pi}}\, C_g^2\, \frac{1}{\varepsilon^2} \int_0^b t^{\varepsilon-1} \mathrm dt = \sqrt{\frac{2}{\pi}}\, C_g^2\, \frac{1}{\varepsilon^3}\, b^\varepsilon.
  \end{equation*}

  \noindent
  \textbf{Second term:} Recalling~\eqref{SlobodetskiiSemi}, we need to
estimate $\int_0^b \int_0^b \frac{\|u(s)-u(t)\|_{X_2}^2}{|s-t|^2} \mathrm ds \mathrm dt$.\\
  For $b \geq s \geq t \geq 0$, we have
  \begin{align*}
      &\norm{u(s)-u(t)}_{X_2} \stackrel{\eqref{eq:repr}}{=} \norm{ \int_0^s \semigroup(\tau) g(s-\tau) \mathrm d\tau - \int_0^t \semigroup(\tau) g(t-\tau) \mathrm d\tau }_{X_2}  \\
      &\qquad=  \norm{ \int_t^s \semigroup(\tau) g(s-\tau) \mathrm d\tau + \int_0^t \semigroup(\tau) g(s-\tau) \mathrm d\tau - \int_0^t \semigroup(\tau) g(t-\tau) \mathrm d\tau }_{X_2}  \\
      &\qquad\leq  \norm{ \int_t^s \semigroup(\tau) g(s-\tau) \mathrm d\tau }_{X_2} + \norm{ \int_0^t \semigroup(\tau) [g(s-\tau) - g(t-\tau)] \mathrm d\tau }_{X_2}  \\
      &\qquad\leq  \int_t^s \norm{\semigroup(\tau)}_{\mathcal L(X_\varepsilon,X_2)} \norm{g(s-\tau)}_{X_\varepsilon} \mathrm d\tau \\
      &\qquad\quad + \int_0^t \norm{ \semigroup(\tau) }_{\mathcal L(X_\varepsilon,X_2)} \norm{g(s-\tau) - g(t-\tau)}_{X_\varepsilon} \mathrm d\tau \\
      &\qquad\stackrel{\eqref{eq:Tt}}{\leq}  \frac{1}{\sqrt[4]{8\pi}} \left(\int_t^s \tau^{-1+\varepsilon/2} \norm{g(s-\tau)}_{X_\varepsilon} \mathrm d\tau  + \int_0^t  \tau^{-1+\varepsilon/2}  \norm{\int_{t-\tau}^{s-\tau} g'(r) dr}_{X_\varepsilon} \mathrm d\tau\right) \\
      &\qquad\stackrel{\eqref{eq:boundg}}{\leq}  \frac{1}{\sqrt[4]{8\pi}}\, C_g \,\frac{2}{\varepsilon}\, \left((s^{\varepsilon/2}-t^{\varepsilon/2}) + t^{\varepsilon/2}\, (s-t)\right).
  \end{align*}
  Analogously, for $b \geq t \geq s \geq 0$, the estimate
  \begin{equation*}
    \norm{u(s)-u(t)}_{X_2} \leq
  \frac{1}{\sqrt[4]{8\pi}}\, C_g \,\frac{2}{\varepsilon}\, \left((t^{\varepsilon/2}-s^{\varepsilon/2}) + s^{\varepsilon/2}\, (t-s)\right)
  \end{equation*}
  holds true. 
  We conclude that 
  \begin{align*}
      \int_0^b  \int_0^b &\frac{\|u(s)-u(t)\|_{X_2}^2}{|s-t|^2} \mathrm ds \mathrm dt \\
      &= \int_0^b  \int_0^t \frac{\|u(s)-u(t)\|_{X_2}^2}{|s-t|^2} \mathrm ds \mathrm dt + \int_0^b  \int_0^s \frac{\|u(s)-u(t)\|_{X_2}^2}{|s-t|^2} \mathrm dt \mathrm ds \\
      &\leq \frac{1}{\sqrt{8\pi}}\, C_g^2\, \frac{16}{\varepsilon^2} \int_0^b
        \int_0^s
        \left(\frac{(s^{\varepsilon/2}-t^{\varepsilon/2})^2}{(s-t)^2}+t^\varepsilon\right)
        \mathrm dt \mathrm ds\\
      &= \sqrt{\frac{32}{\pi}}\, \frac{1}{\varepsilon^2}\, C_g^2 \left(\int_0^b s^{-1+\varepsilon} \int_0^1 \frac{(1-r^{\varepsilon/2})^2}{(1-r)^2} \mathrm dr \mathrm ds  +  \frac{b^{\varepsilon + 2}}{(\varepsilon +1)(\varepsilon + 2)}\right) \\
      &\stackrel{\text{$\varepsilon/2\le 1$}}{\leq} \sqrt{\frac{32}{\pi}}\, \frac{1}{\varepsilon^2}\, C_g^2 \left(\int_0^b s^{-1+\varepsilon} \int_0^1 \frac{(1-r^{\varepsilon/2})^2}{(1-r^{\varepsilon/2})^2} \mathrm dr \mathrm ds  +  \frac{b^{\varepsilon + 2}}{(\varepsilon +1)(\varepsilon + 2)}\right) \\
      &= \sqrt{\frac{32}{\pi}}\, \frac{1}{\varepsilon^2}\, C_g^2 \left(\frac{b^\varepsilon}{\varepsilon}  +  \frac{b^{\varepsilon + 2}}{(\varepsilon +1)(\varepsilon + 2)}\right).
  \end{align*}

\noindent
  \textbf{Conclusion of the proof:}
  By combining the bounds of the three terms, we arrive at the \emph{a~priori} estimate
  \begin{align*}
  \normiii{u}_{H^{1/2}_{0,}((0,b);X_2)} \leq \sqrt[4]{\frac{2}{\pi}}\, \frac{1}{\varepsilon}\, b^{\varepsilon/2} 
  \left( \frac{b}{1+\varepsilon}   + \frac{3}{\varepsilon} + 
          \frac{4 b^2}{(\varepsilon +1)(\varepsilon + 2)}  
  \right)^{1/2} C_g\;,
  \end{align*}
  which gives the assertion.

%%%%%%%%%%%%%%%%%%%%%%%%%%%%%%%%%%%%%%%%%%%%%%%%%%%%%%%%%%%%%%%%%%%%%%%%%
\section{Proof of Lemma~\ref{lem:BoundGamma}} \label{sec:ProofLemBoundGamma}
%%%%%%%%%%%%%%%%%%%%%%%%%%%%%%%%%%%%%%%%%%%%%%%%%%%%%%%%%%%%%%%%%%%%%%%%%
This proof is a slight modification of the proof of \cite[Lemma~3.4]{DD20}. 
We use Stirling's inequalities
\be \label{Stirling}
  \forall x > 0 : \quad \sqrt{2 \pi} x^{x-1/2} \mathrm{e}^{-x} 
   \leq \Gamma(x) \leq \sqrt{2 \pi} x^{x-1/2} \mathrm{e}^{-x} \mathrm{e}^{\frac{1}{12x}}.
\ee
For $j \geq 1$, \eqref{Stirling} yields
\begin{multline*}
  \frac{\Gamma(\lfloor \mu j \rfloor-j+1)}{\Gamma(\lfloor \mu j \rfloor +j+1)} \leq \frac{\Gamma(\mu j-j+1)}{\Gamma(\mu j +j)} \\
  \leq \frac{\sqrt{2\pi} (\overbrace{\mu j-j+1}^{\leq \mu j \text{ as } j\geq 1})^{\overbrace{\mu j-j+1/2}^{\geq 0 \text{ as } \mu \geq 1}} \mathrm e^{-(\mu j-j+1)} \overbrace{\mathrm{e}^{1/(12( \mu j -j+1))}}^{\leq 2} }{ \sqrt{2\pi} (\underbrace{\mu j+j}_{\geq \mu j})^{\mu j + j -1/2} \mathrm e^{-(\mu j+j)} } \leq \frac{2 \mu j }{\mathrm e} \left( \frac{ \mathrm e }{\mu j}  \right)^{2j}
\end{multline*}
and
\begin{equation*}
  \Gamma(j+3)^2 = (\underbrace{j+2}_{\leq 3j})^2 (\underbrace{j+1}_{\leq 2j})^2 j^2 \Gamma(j)^2 \leq j^6 \cdot 72 \pi j^{2j - 1} \mathrm{e}^{-2j} \underbrace{\mathrm{e}^{\frac{1}{6j}}}_{< 2} \leq 144 \pi j^5 j^{2j} \mathrm{e}^{-2j}.
\end{equation*}
Thus, we have
\begin{equation*}
  \forall j \in \IN : \quad \alpha^{2j} \frac{\Gamma(\lfloor \mu j \rfloor-j+1)}{\Gamma(\lfloor \mu j \rfloor +j+1)} \Gamma(j+3)^2 \leq \frac{288 \pi \mu }{\mathrm e} j^6 \left( \frac{\alpha }{\mu}  \right)^{2j}.
\end{equation*}
Hence, we conclude that
\begin{multline*}
  \sum_{j=0}^m \alpha^{2j} \frac{\Gamma(\lfloor \mu j \rfloor-j+1)}{\Gamma(\lfloor \mu j \rfloor +j+1)} \Gamma(j+3)^2 = 4 + \sum_{j=1}^m \alpha^{2j} \frac{\Gamma(\lfloor \mu j \rfloor-j+1)}{\Gamma(\lfloor \mu j \rfloor +j+1)} \Gamma(j+3)^2 \\
  \leq 4 + \frac{288 \pi \mu }{\mathrm e} \sum_{j=1}^{\infty} j^6 \left( \frac{\alpha }{\mu}  \right)^{2j} < \infty,
\end{multline*}
since the ratio test gives
\begin{equation*}
  \lim_{j\to \infty} \frac{ (j+1)^6 \left( \frac{\alpha }{\mu}  \right)^{2(j+1)} }{ j^6 \left( \frac{\alpha }{\mu}  \right)^{2j} } = \left( \frac{\alpha }{\mu}  \right)^2 < 1,
\end{equation*}
i.e., the assertion is proven.

\end{document}